\documentclass[preprint]{imsart}
\usepackage[pdftex,colorlinks]{hyperref}
\RequirePackage{natbib}
\usepackage{amsmath,amssymb,amsthm,verbatim,amscd,mathrsfs}
\usepackage[usenames]{color}
\usepackage{mathrsfs,amsfonts,latexsym} 
\usepackage{graphicx}
\usepackage{multirow} 
\usepackage[ruled,vlined]{algorithm2e}
\usepackage[utf8]{inputenc}

\startlocaldefs

\newcommand{\pr}{\mathbb{P}}
\newcommand{\esp}{\mathbb{E}}
\newcommand{\Var}{\mathrm{Var}}
\newcommand{\Cov}{\mathrm{Cov}}
\newcommand{\argmax}{\mathop{\textrm{argmax}}}
\newcommand{\I}{\mathcal I}
\newcommand{\Ik}{{\mathcal I}_k}
\newcommand{\Q}{\mathcal Q}
\newcommand{\bX}{\mathbb{X}}
\newcommand{\ui}{\underline {\mathfrak{i}}}
\newcommand{\uq}{\underline q}

\newcommand{\Rset}{\mathbb{R}}

\theoremstyle{plain} 
\newtheorem{theorem}{Theorem}

\newtheorem{example}{Example}
\newtheorem{lemma}{Lemma}
\newtheorem{hyp}{Assumption}

\endlocaldefs

\begin{document}

\begin{frontmatter}

\title{New consistent and asymptotically normal estimators for random graph mixture models}
\runtitle{Estimation in random graph mixtures}
\thankstext{T1}{The authors have been supported  by the French Agence Nationale de la Recherche 
under grant NeMo ANR-08-BLAN-0304-01.}

\begin{aug}\author{Christophe Ambroise \ead[label=e1]{christophe.ambroise@genopole.cnrs.fr}}

\author{Catherine Matias\ead[label=e2]{christophe.ambroise,catherine.matias@genopole.cnrs.fr}}

\address{Universit\'e d'\'Evry val d'Essonne - CNRS UMR 8071\\Laboratoire Statistique et G\'enome\\523, place des Terrasses de l'Agora\\91 000 \'Evry, FRANCE\\ \printead{e2}}

\runauthor{C. Ambroise and C. Matias}

\end{aug}

\begin{abstract}
Random  graph mixture  models are now  very  popular for  modeling  real   data  networks. In these setups, parameter estimation   procedures usually  rely  on   variational   approximations,  either combined with the expectation-maximisation  (\textsc{em}) algorithm  or with   Bayesian  approaches. Despite  good results  on synthetic  data, the
  validity of the variational approximation is however not established. 
Moreover,  the behavior of  the maximum  likelihood or of the maximum a posteriori estimators approximated by these procedures is not known in these models, due to  the dependency structure on the variables. In this work, we show that in   many  different  affiliation  contexts
(for binary or weighted graphs), estimators based either on moment
equations  or on the maximization of some composite likelihood  are  strongly consistent  and $\sqrt{n}$-convergent,  where $n$ is the number of nodes. As a consequence, our result  establishes  that the overall  structure of  an affiliation model can be caught by the description of the network in terms of its number of triads (order 3 structures) and edges (order 2 structures).
 We illustrate the  efficiency of our method  on simulated  data  and compare  its  performances with  other existing procedures. A data set of cross-citations among economics journals is also analyzed. 
\end{abstract}


\end{frontmatter}
%%%%%%%%%%%%%%

\maketitle

\section{Introduction}

The analysis of  network data appears in different  scientific fields, such
as social sciences, communication  networks and many others, including
a recent explosion  in the field of molecular  biology (with the study
of  metabolic   networks,  transcriptional  regulatory   networks  and
proteins interactions networks). The  literature is vast, and we
refer for instance to \cite{Boccaletti,Goldenberg} and the book by \cite{Kolaczyk} for interesting introductions to networks.

\cite{ER1} introduced one of the earliest and most
studied  random  graph  model,  in  which  binary  random  graphs  are
considered  as  a  set  of  independent  and  identically  distributed
(i.i.d.)  Bernoulli edge  variables over  a fixed  set of  nodes. This
model is however too homogeneous to capture some important features of
real networks, such as the presence of 'hubs', namely highly connected
nodes. This lack  of heterogeneity led to the  introduction of mixture
versions of  the simple  Erd\H{o}s-Rényi model.  So-called 'stochastic
blockmodels'  \citep{Daudin,FH82,Holland_etal_83,SN97}  were introduced
in  various forms, primarily  in social  sciences to  study relational
data. In  this context, the  nodes are partitioned into  latent groups
(blocks)  characterizing the  relations between  nodes. Blockmodeling
thus refers to the particular structure of the adjacency matrix of the
graph (\emph{i.e.} the matrix containing the edges indicators). By reordering
the  nodes with  respect to  the groups  they belong  to,  this matrix
exhibits  blocks.  Diagonal   and  off-diagonal  blocks  respectively
represent  intra-group and  inter-group connections.  In  case where
blocks  exhibit the  same  behaviour within  their  type (diagonal  or
off-diagonal),  we  further  obtain   what  we  call  an  affiliation
structure. Affiliation  structures are  parsimonious in the  number of
parameters they use  and may model a lot  of situations. For instance,
affiliation models encompass both community
structures and disassortative mixing \citep{Newman_Leicht}. In the first case (community
structure)    the   intra-group connectivities    are    high   while    the
inter-group connectivities are low. Disassortative mixing rather corresponds
to high inter-group connectivities and low intra-group connectivities.

Many  networks  are or  can  be  \emph{weighted}  (or in  other  words
\emph{valued}). Those  weights are precious  additional information on
the   graph   and   should    be   taken   into   account   in   their
analysis. Well-known examples of weighted networks include airline traffic
data between airports, co-authorship networks of scientists
\citep{PNAS_Barrat} or when rather considering the corresponding adjacency matrix, financial correlation matrices \citep{Laloux}. While
the two first examples correspond to sparse weighted networks, the last one
concerns dense (or complete) weighted graphs. 
%In  the  field of  molecular  biology,  weighted  networks have  been
%recently introduced as a  way of integrating heterogeneous data. This
%concerns   for   instance   proteins  interactions   networks   (PIN)
%\citep{gao09},  genes   co-expressions  networks  \citep{Miller}  and
%integrated PIN-mRNA expressions networks \citep{Chen,Tesch}. 
Weighted  networks are  a way  of integrating  heterogeneous  data and
their analysis is thus of primary importance \citep{Newman_weighted}. Community detection (\emph{i.e.} the problem of finding clusters of nodes with many edges joining vertices of the same cluster and comparatively few edges joining vertices of different clusters) has been widely considered in the context of weighted graphs \citep[see for instance][]{Fortunato}. While community detection methods are mainly algorithmic, another approach is to rely on generative models and random graphs mixtures. Stochastic blockmodels for analyzing random graphs with non binary relations between nodes have been considered either in the case of a finite number of possible relations \citep{NS01} or for more general weighted graphs \citep{Maria_Robin_Vacher}. Our approach builds on these latter references.  We also point out the existence of 
\emph{generalized  blockmodels}  for  valued networks  \citep{Ziberna,
  Book_Blockmodel} which however do not rely on a probabilistic model as we shall do here.

In this  article, we  will be interested  in both binary  and weighted
random  graphs  and will  focus  on  mixture  models. We  mention  the
existence  of  an  increasing  literature  on  two  different  related
concepts: \emph{mixed membership} \citep{Airoldi,Erosheva} and \emph{overlapping} \citep{Latouche_overlap} stochastic blockmodels for binary networks, in which nodes may belong to several classes.  However these models are beyond the scope of the present work. \\

Current  estimation procedures in random graph mixture models  rely  on approximations  of the likelihood, which is itself intractable due to the presence of the non observed   groups.  Either   expectation   maximization  (\textsc{em})
algorithm \citep{DLR}, or Bayesian approaches  are at the core of these
strategies. Both  rely on the  computation of the distribution  of the
hidden nodes states, conditional on the observed edges variables. However, in the particular case of random graph
mixtures, the exact computation of this conditional distribution can
not  be obtained, due  to its  non-factorized form.  Thus, approximate
computations  are  made,  leading  to  what  is  called  'variational'
\textsc{em} or Bayes strategies \citep{Daudin,Latouche_Bayes_EM,Picard_BMC,2008_PR_Zanghi}.
The major drawback of these methods is their relatively large 
computational time. Besides, even  if these methods exhibit good behavior 
on   simulated    data,   they   suffer from a lack of theoretical support. Indeed,  two major features  of these procedures  still lack
understanding. First, the quality  of the variational approximation is
not known, and this approximation may even prevent convergence to local maxima of the
likelihood  \citep{cv_varEM}.  Second,  the  consistency of  the  maximum
likelihood or of the maximum a posteriori estimators is still an open question in these models. \\

Here, we propose simple strategies for estimating the parameters of 
mixture random graph models, in the particular affiliation case.  The methods not
only   rely  on   established  convergence   results,  but   are  also simpler than variational approaches. 
By focusing on small structures (edges and triads) and treating these as if
they were (but never assuming they are) independent, we prove that we may recover the
main features of an  affiliation model. We adopt strategies based on either solving  moment equations %(Section~\ref{sec:bin_poly}) 
or maximizing  a composite marginal likelihood. %(Sections~\ref{sec:bin_multidim} and ~\ref{sec:weighted}). 
A composite marginal likelihood consists in  the product of marginal distributions and may replace the likelihood in models with some dependency structure \cite[see for instance][]{Cox_Reid_04,Varin08}. In the weighted random graphs case, our result shows that parameters may be estimated relying on a composite likelihood of univariate marginals. This is not the case for binary random graphs, because parameters of mixtures of univariate Bernoulli distributions are not identifiable. However, parameters of mixtures of $3$-variate Bernoulli are identifiable \cite[see][Corollary 5]{ECJ}. Thus, in the binary random graph case, we develop moment or composite likelihood methods based on the marginals of triads, namely the $3$ random variables $(X_{ij},X_{ik},X_{jk})$ induced by a set of $3$ nodes $(i,j,k)$.   \\

Once the convergence of our estimators, let us say $\hat \theta_n$ to $\theta$, has been established, the next question of interest concerns the order at which the discrepancy $\hat \theta_n-\theta$ converges to zero. We establish asymptotic normality results, thus obtaining rates of convergence of our procedures. This is in sharp contrast  with existing methods and the  first insight on the difficult  issue  of exhibiting  (optimal)  rates  of convergence  for estimation procedures in these random graphs  models. Indeed, a still open  problem  may be stated as follows: what is the parametric rate of convergence when observing $n(n-1)/2$ (non independent) random variables over a set of $n$ nodes, distributed according to a random graph model? Is it $1/\sqrt n$ or $1/n$? In other words, the issue is whether the observation of these potentially $n(n-1)/2$ dependent edges variables over a set of $n$ nodes enables existence of estimation procedures with rates of convergence of the order $1/n$ or rather $1/\sqrt{n}$.  We obtain here theoretical results with rates of convergence of the order at least $1/\sqrt{n}$ (which might not be optimal). Moreover, in the \emph{degenerate} case where the group proportions are equal, the rates of convergence increase to $1/n$. Our simulations seem also to indicate rates of convergence faster than $1/\sqrt{n}$, that might be due to degeneracies in the limiting variances of our central limit results, \emph{i.e.} the fact that these variances might be zero.\\

The paper is organized as follows. 
In Section~\ref{sec:model_gen},  we state the different notations, present the general assumptions of our model as well as the main result:  a  law of  large  numbers and  a  central  limit theorem  for normalized  sums  of  functions  of variables over a $k$-tuple of nodes.  Section~\ref{sec:binary} focuses on binary random graphs: after introducing the specific model for binary variables, we present two different estimation procedures. The first one (Section~\ref{sec:bin_poly}) relies on moment equations and assumes that the group proportions $\boldsymbol \pi$ are known, while the second one (Section~\ref{sec:bin_multidim}) is more general and relies on composite likelihood. Section~\ref{sec:weighted} presents the weighted random graph model as well as the parameter estimation procedure, relying also on a composite likelihood approach. While a first part of our work focuses on theoretical results about consistency of the procedures, a second part is dedicated to algorithmic issues as well as experiments. In Section~\ref{sec:algos}, we present the implementation of the estimation procedures. A particular attention is paid to the problem of unraveling the latent structure of the model (Section~\ref{sec:latent}). In Section~\ref{sec:simus}, the performances of our procedures are illustrated on synthetic data and we also provide the analysis of a real data example. Finally, all the proofs are postponed to Section~\ref{sec:proofs}.

\section{Model and main result}
\label{sec:model_gen}
Let us first give some notations that will be useful throughout this article. 
 For any $Q\geq 1$, let $\mathcal{S}_Q$ denote the simplex $\{(\pi_1,\ldots,\pi_Q) ; \pi_i\ge 0 ; \sum_{i=1}^Q \pi_i =1\}$ and $\mathcal{V}_Q=\{(v_1,\ldots,v_Q), v_i\in \{0,1\}, \sum_{i=1}^Q v_i=1\}$.
For  the  sake  of  simplicity,  we only  consider  in  the  following
undirected graphs with no self-loops. Easy generalizations may be done
to handle directed graphs, with or without self-loops.  \\

In  this section, we  define a  general  mixture model  of random
graphs in the following way. 
First, let $\{Z_i\}_{1\le i\le n}$ be i.i.d. vectors $Z_i=(Z_{i1},\ldots ,
Z_{iQ})  \in  \mathcal{V}_Q$,  following  a  multinomial  distribution
$\mathcal{M}(1,\boldsymbol  \pi)$, where  $\boldsymbol  \pi =  (\pi_1,
\ldots,\pi_Q) \in  \mathcal{S}_Q$. Random variable  $Z_i$ indicates to
which group  (among $Q$ possibilities) node $i$  belongs. These random
variables are used to introduce heterogeneity in the random graph model.

Next, the observations $\{X_{ij}\}_{1\le i  <j\le n}$  are indexed by the
node pairs $\{i,j\}$ and take  values in a  general normed vector
space $\mathcal{X}$ (in the next sections,
$\mathcal{X}=\{0,1\}$  or $\mathbb{N}$  or  $\mathbb{R}^l$).  We  then
assume that conditional on the latent classes $\{Z_i\}_{1\le i\le n}$, the random variables
$\{X_{ij}\}_{1\le i <j\le n}$ are independent. Moreover, the conditional
distribution of $X_{ij}$ depends only on $Z_i, Z_j$ and has finite variance.
The model may thus be summarized in the following way
\begin{equation}
  \label{eq:model_gen}
\textbf{General model}  \left\{
    \begin{array}{cl}
\cdot  &\{Z_{i}\}_{1\le   i  \le  n}  \text{  i.i.d.    vectors  in  }
\mathcal{V}_Q , \text{ with distribution }\mathcal{M}(1,\boldsymbol \pi),\\
\cdot &\{X_{ij}\}_{1\le i <j\le n} \text{ observations in } \mathcal{X} , \\
 \cdot & \pr (\{X_{ij}\}_{1\le i <j\le n} | \{Z_{i}\}_{1\le i \le n})
 =\mathop{\otimes}_{1\le i <j\le n} \pr (X_{ij} | Z_{i}, Z_j ), \\
\cdot & \esp(\| X_{i,j} \|^2|Z_i, Z_j) <+ \infty .
    \end{array}
\right.
\end{equation}

It may be worth noting that the variables $\{X_{ij}\}_{1\le i <j\le n}$ are not independent in general, but we often make use of the fact that sets of non adjacent edges induce independent random variables. More precisely, if $I,J\subset \{1,\ldots,n\}$ with $I\cap J=\emptyset$, then $\{X_{ij}\}_{(i,j)\in I^2}$ and $ \{X_{ij}\}_{(i,j)\in J^2} $ are independent.

In the next sections, we will focus on the particular affiliation mixture model, where the conditional distribution of an edge variable $X_{ij}$ only depends on whether the  endpoints $i,j$ belong to the same group (\emph{i.e.} $Z_i=Z_j$). We shall thus refer to the assumption  
\begin{equation}
 \textbf{Affiliation structure}:   \pr (X_{ij} | Z_{i}, Z_j )=  \pr (X_{ij} | 1_{Z_{i} = Z_j} ), \label{eq:model_affil} 
\end{equation}
where $1_A$  is the indicator function of the set  $A$.

Moreover,  in  the particular  case  of  equal  group proportions  and
affiliation    structure,   we    shall   observe    some   degeneracy
phenomenas. These  are due to  the fact that the  distribution becomes
invariant under permutation of the  specific values of the node groups
(see Lemma~\ref{lem:degeneracy} in Section~\ref{sec:proofs} for 
more details). 
For later use, we thus also introduce the equal group proportions setting
\begin{equation}
\textbf{Equal group proportions case}: \pi_q=1/Q \text{ for any } q\in \{1,\ldots,Q\}. \label{eq:model_equalgroup}
\end{equation}

Let us now  motivate the following developments. Under the affiliation structure assumption, the distribution of a single edge follows a two-components mixture of the form 
\begin{equation*}
  X_{ij}\sim \gamma \pr(X_{ij}|Z_i=Z_j)+(1-\gamma) \pr(X_{ij}|Z_i\neq Z_j).
\end{equation*}
For weighted random graphs, we shall assume a parametric form for this absolutely continuous conditional distribution, namely $\pr(X_{ij}|Z_i=Z_j)=\pr_{\theta_{\text{in}}}(X_{ij})$ and $\pr(X_{ij}|Z_i\neq Z_j)=\pr_{\theta_{\text{out}}}(X_{ij})$. The vast majority of families of parametric absolutely continuous distributions give finite mixtures whose parameters are identifiable. This is equivalent to saying that $\esp[\log (
 \gamma \pr_{\theta_{\text{in}}}(X_{12})+(1-\gamma) \pr_{\theta_{\text{out}}}(X_{12}))]$ has a unique maximum at the true parameter value $(\theta_{\text{in}},\theta_{\text{out}})$. This reasoning is at the core of maximum likelihood estimation and motivates the introduction of a \emph{composite} log-likelihood
\begin{equation*}
 \mathcal{L}^{\text{compo}}_{X}(\theta)
= \sum_{1\le i<j\le  n }  \log (\gamma \pr_{ \theta_{\text{in}}} (X_{ij})  +(1-\gamma)  \pr_{\theta_{\text{out}}}(X_{ij} ) ),
\end{equation*}
which is not the model likelihood as the  random variables $X_{ij}$ are not  independent.
%As the  random variables $X_{ij}$ are not  independent, this quantity
%is called a composite log-likelihood. 
Its  usefulness  to estimate  the  parameters  relies  on whether  the
renormalized criterion $ \mathcal{L}^{\text{compo}}_{X}(\theta)/$ $(n(n-1))$ converges to the expectation $\esp[\log (
 \gamma \pr_{\theta_{\text{in}}}(X_{12})+(1-\gamma) \pr_{\theta_{\text{out}}}(X_{12}))]$.
We shall prove below that the answer is yes and thus, maximizing $ \mathcal{L}^{\text{compo}}_{X}(\theta)$ with respect to $\theta$ is a good strategy.

In the binary random graph case however, the strategy has to be modified because each random variable $X_{ij}$ follows a mixture of univariate Bernoulli distributions whose parameters are not identifiable. We thus rather consider mixtures of $3$-variate Bernoulli distributions which appear to be sufficient to consistently estimate the parameters.\\

Thus, we are now  interested more generally in the  behavior of empirical
sums of functions of the random variables induced by a $k$-tuple of
nodes. These  empirical estimators are  at the core of  the estimation
procedures that we shall later consider. To this aim, let us introduce some more notations. 

Define the set of nodes $\I= \{1,\ldots, n\}$ and the set
of $k$ distinct nodes $\Ik=\{(i_1,\ldots,i_k) \in \I^k ; i_j\neq i_l \text{ for any } j \neq
l \}$. ($\Ik$ is also  the set of injective maps from $\{1,\ldots,k\}$
to $\I=\{ 1,\ldots, n\}$). 
%Let us introduce some notations. 
 For any fixed integer $k\geq 1$, and
any $k$-tuple  of nodes $\ui=(i_1,\ldots,i_k)  \in \Ik$, we
let   $\bX^{\ui}=(X_{i_1i_2},\ldots   ,X_{i_1i_k},X_{i_2i_3},\ldots  ,
X_{i_{k-1}i_k})$ be the vector of $p=\binom k 2$ random variables induced
by the $k$-tuple of nodes $\ui$.
Moreover, for any $s\ge 1$ and any measurable function $g : \mathcal{X}^p \to \Rset^s$, we let 
\begin{equation*}
  \hat m_g   = \frac {(n-k)!} {n!}  \sum_{\ui \in \I^k} g(\bX^{\ui})
  \quad \text{and} \quad m_g=\esp(g(\bX^{(1,\ldots,k)})) .
\end{equation*}

The next theorem establishes a strong  law of large numbers as well as
asymptotic normality of the estimator $\hat m_g$. 
As the random variables  $\{X_{ij}\}$ are not independent, consistency
(as well as asymptotic normality) of this empirical estimator is not trivial and has to be established carefully.

\begin{theorem}\label{thm:main}
Under the assumptions of model \eqref{eq:model_gen}, 
for any $k,s\ge 1$ and  $p=\binom {k} {2}$ and any measurable function
$g:\mathcal{X}^p \to \mathbb{R}^s$ such that
$\mathbb{E}(\|g(\bX^{(1,\ldots,k)})\|^2)$ is finite,  the estimator $\hat m_g$ is consistent 
\begin{equation*}
 \hat m_{g} \mathop{\to}_{n\to\infty}  m_{g} \text{ almost surely}, 
\end{equation*}
as well as asymptotically normal 
$  \sqrt{n}(\hat m_{g}-m_{g})\leadsto_{n\to\infty} \mathcal{N}(0,\Sigma_g)$.
If we moreover assume an affiliation structure \eqref{eq:model_affil} with equal group proportions \eqref{eq:model_equalgroup}, then $\Sigma_g=0$ and $n(\hat m_{g}-m_{g})$ converges in distribution as $n$ tends to infinity.
\end{theorem}

Let us  now give  some comments about  the previous result.  First, an
expression for  the limiting distribution  $\Sigma_g$ is given  in the
proof of the theorem. Such  an expression is useful for instance
in the construction of confidence intervals. However, although our estimators
of the model parameters are  derived from estimators of the form $\hat
m_g$, we did not obtain here simple expressions for their limiting variance from an expression of $\Sigma_g$. Thus, rather than the exact form of the limiting distribution, we are more interested here in rates of convergence.

The  theorem states that the convergence of $\hat m_g$ to $m_g$ happens with a rate at least  $1/\sqrt{n}$. In the case where we consider an affiliation structure with equal group proportions, we prove that the limiting variance is null (\emph{i.e.} $\Sigma_g=0$), meaning that $\sqrt{n}(\hat m_{g}-m_{g})$ converges in probability to zero. We then further prove that the sequence $n(\hat m_g-m_g)$ converges in distribution (to some non Gaussian limit). Thus in this degenerate case, the convergence of $\hat m_g$ happens  at the faster rate $1/n$.

We shall see that consistency as well as rates of convergence are preserved in the estimation procedures that we deduce from moment estimators of the form $\hat m_g$. To our knowledge, this work is the first one giving some insights about consistency and rates of convergence of estimation procedures in random graphs mixtures models.\\

In  the next  sections, we  consider two  particular instances  of the
mixture  model defined in \eqref{eq:model_gen}:  the  binary  affiliation  model
(Section~\ref{sec:binary}) and the weighted affiliation model (Section~\ref{sec:weighted}).

\section{Binary affiliation model}\label{sec:binary}
In the case of binary random graphs, we observe binary random variables
$\{X_{ij}\}_{1\le  i <j\le  n}$ indicating  presence ($1$)  or absence
($0$)  of  an edge  between  nodes $i$  and  $j$.  The latent  classes
$\{Z_i\}_{1\le i\le  n}$ are  still distributed as  i.i.d. multinomial
vectors on $\mathcal{V}_Q$.
Conditional on these latent classes  $\{Z_i\}_{1\le i\le n}$, we assume that $\{X_{ij}\}_{1\le i <j\le n}$ are independent Bernoulli $\mathcal{B}(\cdot)$ random variables, with parameters  depending on the node groups. More precisely, we restrict our attention to the affiliation structure model \eqref{eq:model_affil}, where 
nodes connect differently whether they belong to the same group or not. 
We let 
\begin{equation}\label{eq:model_bin}
\forall q,\ell \in \{1,\ldots,Q\}, \quad  X_{ij} | Z_{iq}Z_{j\ell}=1 \sim 
\left\{
  \begin{array}{ll}
\mathcal{B}(\alpha) & \text{ if } q=\ell, \\
\mathcal{B}(\beta) & \text{ if } q\neq \ell .
  \end{array}
\right. 
\end{equation}
Here, $\alpha$  and $\beta$ respectively  are the intra-group  and the
inter-group connectivities and we let $p_{q\ell}=\alpha 1_{q=\ell}+\beta 1_{q\neq \ell}$, for any $1\le q,\ell\le Q$. In the following, we always assume $\alpha
\neq \beta$.  

The whole parameter space is given by  
\begin{equation*}
  \Pi  =\{ (\boldsymbol  \pi, \alpha,  \beta) ; 
  \boldsymbol \pi\in \mathcal{S}_Q\cap (0,1)^Q,  \alpha \in (0,1), \beta \in
  (0,1), \alpha \neq \beta \} .
\end{equation*}
We  will  use  the  notation  $b(x,p)=p^x  (1-p)^{1-x}$  (where  $x\in
\{0,1\}$ and $p\in  [0,1]$) for  Bernoulli density  with respect to
counting measure. 
Note that in this setup, the complete data log-likelihood simply writes
\begin{multline}\label{eq:loglik1}
  \mathcal{L}_{X,Z}(\boldsymbol \pi, \alpha, \beta) = \log \mathbb{P}_{\boldsymbol \pi, \alpha, \beta}(\{X_{ij}\}_{1\le i<j\le n}, \{Z_i\}_{1\leq i \le n}) 
= \sum_{i=1}^n \sum_{q=1}^{Q} Z_{iq} \log \pi_q \\
+ \sum_{1\le i<j \le n}\sum_{ q=1 }^Q  Z_{iq}Z_{jq} \big\{X_{ij} \log \alpha  +(1-X_{ij})\log(1-\alpha)\big\} \\
+ \sum_{1\le i<j \le n}\sum_{1\le q\neq \ell \le Q} Z_{iq}Z_{j\ell} \big\{X_{ij} \log \beta +(1-X_{ij})\log(1-\beta)\big\}.  
\end{multline}
Figure~\ref{fig:example} (left part) displays an example of a binary random graph
distributed according to this affiliation model.

\subsection{Moment estimators in the binary affiliation model with known group proportions}
\label{sec:bin_poly}
The  following  approach  based  on  moment  equations  was  initially
proposed  by  \cite{FH82}  to  estimate  the  connectivity  parameters
$\alpha$ and $\beta$ (as well as, in some cases, the number of groups $Q$). The core idea is simple: the moment equations corresponding to the distribution of a triplet $(X_{ij},X_{ik},X_{jk})$ give three equations which can be used to estimate the two parameters $\alpha$ and $\beta$, as soon as the group proportions (also appearing in these equations) are known. However, this method has not been thoroughly checked by Frank and Harary and may give rise to multiple solutions. Indeed, these authors never discuss uniqueness of the solutions to
the  system of  (non linear)  equations they  consider.  This
point  has  been  partly  discussed  in  \cite{ident_mixnet}  and  the
estimation procedures proposed here are an echo to the identifiability
results obtained there. 

The following method applies only when the mixture proportions $\boldsymbol{\pi}$ are known. We develop in Section~\ref{sec:algos}  an algorithmic procedure that iteratively estimates the group proportions in a first step, and the connectivity parameters $(\alpha,\beta)$ in a second step. This second step uses the method we shall now describe.\\

First, we let $s_2=\sum_q \pi_q^2$ and $s_3=\sum_q\pi_q^3$. Then, one easily gets the formulas 
\begin{equation}
  \begin{array}{ccl}
  m_1&:=&\mathbb{E}(X_{ij})= s_2\alpha+(1-s_2)\beta , \\
 m_2&:=& \mathbb{E}(X_{ij}X_{ik})= s_3\alpha^2+2(s_2-s_3)\alpha\beta+(1-2s_2+s_3)\beta^2 , \\
m_3 &:=&\mathbb{E}(X_{ij}X_{ik}X_{jk})= s_3\alpha^3+3(s_2-s_3)\alpha\beta^2+(1-3s_2+2s_3)\beta^3. 
  \end{array} \label{eq:moments}
\end{equation}
Since  any triplet $(X_{ij},X_{ik},X_{jk})$ takes finitely many states, its distribution is completely characterized by a finite number of its moments. In the binary  affiliation mixture model context, there are in fact only three different moments induced by a triplet distribution. Thus, the previous three moment equations completely characterize the distribution of any triplet $(X_{ij},X_{ik},X_{jk})$. Note that looking at higher order motifs, namely at the distribution of a set of $p=\binom k 2$ random variables over a set of $k$ nodes for $k\ge 4$ would provide more equations but would also lead to more intricate methods \cite[see for instance][]{ident_mixnet}.

In the article by \cite{ident_mixnet}, the possible solutions (with respect to $\alpha$ and $\beta$) of this set of moment equations are examined. Their result distinguishes the equal group proportions case ($\pi_q=1/Q, \forall 1\le q \le Q$) where a degeneracy phenomenon takes place.

\begin{theorem} \label{thm:moment}\citep{ident_mixnet}.
If  $m_2\ne m_1^2$, then the $\pi_q$'s are unequal  and
we can recover the parameters $\beta$ and $\alpha$ via the rational formulas
\begin{equation}
\beta                                    =\frac{(s_3-s_2s_3)m_1^3
  +(s_2^3-s_3)m_2m_1+(s_3s_2-s_2^3)m_3 }{(m_1^2-m_2)(2s_2^3-3s_3s_2+s_3)}
\quad \text{ and } \quad \alpha =\frac {m_1+(s_2-1)\beta}{s_2}.  \label{eq:bin_non_uniform_estim}
\end{equation}
If $m_2=m_1^2$, then the $\pi_q$'s are equal and we have 
\begin{equation}  \label{eq:bin_uniform_estim}
\beta =m_1+\left (  \frac {m_1^3-m_3}{Q-1} \right )^{1/3} \quad
\text{and} \quad 
\alpha =Q m_1+(1-Q)\beta.
\end{equation}
\end{theorem}

As soon as $s_2$ and $s_3$ are known, by plugging  estimators of the moments $m_i$ into these equations, we obtain 
 simple estimates for  parameters $\alpha$ and $\beta$. 
We thus first introduce  empirical moment estimators $\hat m_i$, defined by   
\begin{align}
\hat m_1 &=  \frac 1{n(n-1)} \sum_{(i,j)\in \I_2} X_{ij} ,\qquad
\hat m_2 = \frac 1{n(n-1)(n-2)} \sum_{ (i,j,k)\in \I_3 } X_{ij}X_{ik} ,\nonumber \\ 
 \hat m_3&=\frac 1 {n(n-1)(n-2)} \sum_{(i,j,k)\in \I_3} X_{ij}X_{ik}X_{jk} . \label{eq:estim_moment}
\end{align}
Note that those estimators are all of the form $\hat m_{g}$ for some specific function $g$. Thus, their consistency is a consequence of Theorem~\ref{thm:main}.
We are then able to prove the following result.

\begin{theorem}
\label{thm:bin_moment} 
In  the binary affiliation model specified by \eqref{eq:model_gen} and \eqref{eq:model_bin}, when the group proportions $\boldsymbol \pi$ are supposed to be known, we have the following results.
\begin{itemize}
\item [i)]
When the  $\pi_q$'s are unequal, the estimators $(\hat \alpha,\hat \beta)$ defined through \eqref{eq:bin_non_uniform_estim} where the $m_i$'s are replaced by the $\hat m_i$'s, converge almost surely to $(\alpha,\beta)$. Moreover, the rate of this convergence is at least $1/\sqrt{n}$.
\item [ii)]
When the  $\pi_q$'s are equal, the estimators $(\hat \alpha,\hat \beta)$ defined through \eqref{eq:bin_uniform_estim} where the $m_i$'s are replaced by the $\hat m_i$'s, converge almost surely to $(\alpha,\beta)$. Moreover, the rate of this convergence is at least $1/n$.
\end{itemize}
\end{theorem}

The performances of this method, combined with an iterative procedure to uncover the latent structure and estimate the group proportions are illustrated in Section~\ref{sec:simus}.

\subsection{$M$-estimators in the binary affiliation model} \label{sec:bin_multidim}
We shall now describe  another parameter estimation procedure based on
$M$-estimators  \cite[see for  instance][Chapter  5]{VDV}, \emph{i.e.}
estimators    maximizing   some    criterion   (here,    a   composite
likelihood). This  procedure is more direct than  the previous moments
method developed in Section~\ref{sec:bin_poly},  as it does not assume
a preliminary knowledge of the group proportions $\boldsymbol \pi$.

Let us recall that  $\bX^{(i,j,k)}= (X_{ij},X_{ik},X_{jk})$. The random  vectors $\bX^{(i,j,k)}$ form  a set
of  non   independent, but identically distributed  vectors,  with
distribution of each $\bX^{(i,j,k)}$ given by the mixture 
\begin{equation*}
\mathbb{P}_{\boldsymbol\pi, \alpha,\beta} (\bX^{(1,2,3)} )
           = \sum_{1\le q,\ell, m\le Q} {\pi}_q {\pi}_{\ell} {\pi}_m b\left(  X_{12} , p_{q\ell} \right)  b\left(  X_{13} , p_{qm} \right)  b\left(  X_{23} , p_{\ell m} \right),
\end{equation*}
where we recall that $p_{q\ell}=\alpha 1_{q=\ell}+\beta 1_{q\neq \ell}$.
In this mixture, many components are in fact equal. 
Indeed, the components reduce to only four (when $Q=2$) or five (when $Q\ge 3$) different distributions. More precisely, we may write 
\begin{multline}
\label{eq:multiBernAffil}
\mathbb{P}_{\boldsymbol\pi, \alpha, \beta} (\bX^{(1,2,3)} )= \gamma_1 b(  X_{12}, \alpha)  b( X_{13}, \alpha) b(  X_{23}, \alpha) \\
                        +  \gamma_2 b(  X_{12}, \beta)  b(  X_{13}, \beta) b(  X_{23}, \alpha) 
                         + \gamma_3   b(  X_{12}, \beta)  b(  X_{13}, \alpha) b(  X_{23},\beta) \\
                        + \gamma_4  b(  X_{12}, \alpha)  b(  X_{13}, \beta) b(  X_{23}, \beta)
 + \gamma_5  b(  X_{12}, \beta)  b(  X_{13}, \beta) b(  X_{23}, \beta) ,
\end{multline}
where the five proportions $\boldsymbol {\gamma}=(\gamma_1,\ldots, \gamma_5) \in \mathcal{S}_5$  appearing in this mixture are related to the original proportions $\boldsymbol \pi$, by the following relations 
\begin{equation}
  \label{eq:gamma}
\begin{cases}
  \gamma_1 = \sum_{ q=1}^Q \pi_q^3 =s_3 , \\
  \gamma_j = \sum_{1\le q\neq \ell \le Q} \pi_q^2\pi_{\ell} = s_2-s_3,  \mbox{ for }j\in \{2,3,4\} , \\
\gamma_5= \sum_{\substack{1\le q, \ell,m \le Q\\ |\{q,\ell , m \}|=3}} \pi_q\pi_{\ell}\pi_m = 1-3s_2+2s_3.
\end{cases}  
\end{equation}
Note that when $Q=2$, the fifth proportion $\gamma_5$ is automatically equal to zero. Moreover, as soon as $Q\le 3$, the set of equations \eqref{eq:gamma} defines a one-to-one relation between $\boldsymbol \pi$ and $\boldsymbol \gamma$. However, when $Q>3$, the parameter $\boldsymbol \pi$ is not uniquely defined from $\boldsymbol \gamma$ and  is not identifiable from the mixture distribution \eqref{eq:multiBernAffil}.

We  emphasize  that the  distribution  \eqref{eq:multiBernAffil} is  a
constrained $3$-variate  Bernoulli mixture. Parameters identifiability
of such  a distribution is  further discussed below. However  we shall
already remark that while parameters  of mixture models may in general
be  identified  only up  to  a permutation  on  the  node labels,  the
constrained  form  of the  mixture  \eqref{eq:multiBernAffil} has  the
following  consequence: the  parameters $\alpha$  and $\beta$  will be
exactly  recovered as soon  as the  mixture components  are identified
from  \eqref{eq:multiBernAffil} and  whatever the  labelling  of these
mixture components. Indeed, among the five unordered components of the
mixture,  only three  of them  will be  the product  of  two identical
one-dimensional  distributions, times a  different one.  The parameter
$\beta$ is then the parameter appearing in exactly two marginals in any of those three components.\\

Let us consider as our criterion a composite marginal log-likelihood 
 of the observations %$\{X_{ij}\}_{1\le i<j\le n}$ defined by 
\begin{equation}\label{eq:pseudologlik1}
  \mathcal{L}^{\text{compo}}_{X}(\boldsymbol \pi, \alpha,\beta)
= \sum_{(i,j,k)\in \I_3}\log  \pr_{\boldsymbol \pi,\alpha,\beta} (\bX^{(i,j,k)}). 
\end{equation}
We stress that  this quantity is not derived from  the marginal of the
model  complete data likelihood  (expressed in  \eqref{eq:loglik1})  and is
simpler.  It would be the log-likelihood of the observations if the
triplets  $\{\bX^{(i,j,k)}\}_{(i,j,k)\in \I_3}$   were  independent,  which  is
obviously not the case. We now define our estimators as 
\begin{equation}
 (\hat {\boldsymbol \pi}_n , \hat \alpha_n,\hat \beta_n) =\argmax_{\boldsymbol \pi, \alpha, \beta}  \mathcal{L}^{\text{compo}}_{X}(\boldsymbol \pi, \alpha,\beta).
\label{eq:estim_alpha_beta}
\end{equation}
Note  that  according  to  the  non uniqueness  of  group  proportions
$\boldsymbol  \pi$ corresponding  to mixture  proportions $\boldsymbol
\gamma$, the  maximum with respect  to $\boldsymbol \pi$ in  the above
equation may not be unique.  We also let $\hat {\boldsymbol \gamma}_n$
be defined from $\hat {\boldsymbol \pi}_n$ through \eqref{eq:gamma}.

Using     Theorem~\ref{thm:main},    the     renormalized    criterion
\eqref{eq:pseudologlik1} converges  to a limit. The key  point here is
that under an identifiability assumption on the model parameters, this limit is a
function whose maximum is attained only at the true parameter value $(\boldsymbol
\gamma,  \alpha,\beta)$.  Using  classical  results from  $M$-estimators
\citep{VDV,Wald},  we  can  then  obtain  consistency  and  asymptotic
normality of the estimators defined through \eqref{eq:estim_alpha_beta}. 
We  thus  need here  to  assume  the  identifiability of  the  model 
parameters.

 \begin{hyp}
   \label{hyp:ident_binary}
The  parameters $\boldsymbol \gamma,\alpha, \beta$ of the model defined by \eqref{eq:multiBernAffil} are identifiable. 
In other words, if there exist $(\boldsymbol\pi, \alpha, \beta)$ and $ (\boldsymbol\pi ', \alpha', \beta') $ such that for any $(x,y,z)\in \{0,1\}^3$ we have 
  \begin{equation*}
\mathbb{P}_{\boldsymbol\pi, \alpha, \beta}  (X_{12}=x,X_{13}=y,X_{23}=z )= \mathbb{P}_{\boldsymbol\pi ', \alpha', \beta'} (X_{12}=x,X_{13}=y,X_{23}=z ) , 
  \end{equation*}
then $(\boldsymbol \gamma ,\alpha, \beta) = ( \boldsymbol \gamma',\alpha', \beta')$, where $\boldsymbol \gamma, \boldsymbol \gamma'$ are defined through \eqref{eq:gamma} as functions of $\boldsymbol \pi, \boldsymbol \pi'$, respectively. 
 \end{hyp}

Let us now give some comments on this assumption. We first mention that identifiability of all the parameters $(\boldsymbol \pi,\alpha,\beta)$ in the model defined by \eqref{eq:model_gen} and \eqref{eq:model_bin}, \emph{i.e.} relying on the full distribution over $\cup_{n\ge 1}\{0,1\}^{\binom n 2}$ (comprising the marginal distributions of the random graphs over a set of $n$ nodes, for any value of $n$), is a difficult issue, for which only partial results have been obtained in \cite{ident_mixnet}. Surprisingly, the results under the affiliation assumption are  more difficult to obtain than in the non affiliation case. 
The question here is slightly different and we ask whether a triplet distribution \eqref{eq:multiBernAffil} is sufficient to identify only $\alpha$ and $\beta$ (as well as the corresponding proportions $\boldsymbol \gamma$). As already pointed out, the distribution \eqref{eq:multiBernAffil} is a constrained distribution from the larger class of $3$-variate Bernoulli mixtures. 
In the case of (unconstrained) finite mixtures of multivariate (or $3$-variate) Bernoulli distributions, while the models have been used for decades and were strongly believed to be identifiable \citep{Carreira}, the rigorous corresponding result has been established  only very recently and using rather elaborate techniques \cite[see][Corollary 5]{ECJ}. Unfortunately, this latter result does not  apply directly  here. 
While this might be hard to establish,  we strongly believe that $\boldsymbol \gamma,\alpha, \beta$ are identifiable from the distribution \eqref{eq:multiBernAffil}  and  we advocate that from the simulations we performed, it seems a reasonable assumption to make. \\

In the following, we also restrict our  attention to compact parameter spaces,  as this greatly simplifies the proofs and is not much restrictive. Generalizations could be done at the cost of technicalities \cite[see for instance][Chapter 5]{VDV}.
\begin{hyp}
  \label{hyp:compact1}
Assume  that there  exists some  $\delta >0$  such that  the parameter
space is restricted to $\Pi_{\delta}= \{(\boldsymbol \pi,\alpha,\beta)
\in \Pi ; \forall 1\le q \le Q, \pi_q\ge
\delta, \alpha \in [\delta, 1-\delta] , \beta \in [\delta, 1-\delta] \}$. 
\end{hyp}

We are then able to prove the following result. 

\begin{theorem}\label{thm:multiv_bernoulli}
In the model defined by \eqref{eq:model_gen} and \eqref{eq:model_bin}, under  Assumptions~\ref{hyp:ident_binary} and \ref{hyp:compact1}, the estimators $(\hat {\boldsymbol \gamma}_n, \hat \alpha_n,\hat \beta_n)$ defined by \eqref{eq:estim_alpha_beta} are consistent,  as the sample size $n$ grows to infinity. 
Moreover, the rate of this convergence is at least $1/\sqrt{n}$ and increases to $1/n$ in the particular case of equal group proportions \eqref{eq:model_equalgroup}. 
\end{theorem}

Let us now comment  this result. 
We prove that the rate of convergence of our estimators is at least $1/\sqrt{n}$. However, our simulations (see Section~\ref{sec:simus}) seem to exhibit a faster rate, indicating that the limiting covariance matrix of the discrepancy $\sqrt{n} (\hat {\boldsymbol \gamma}_n-\boldsymbol {\gamma}, \hat \alpha_n- \alpha,\hat \beta_n-\beta)$ might be zero, even beyond the case of equal group proportions. 
Note also that when $Q\le 3$, a consequence of the above result is that the estimator of the group proportions $\hat {\boldsymbol \pi}_n$ defined through $\hat {\boldsymbol\gamma}_n$ as the unique solution to the system of equations \eqref{eq:gamma}, is also consistent and converges with a rate at least $1/\sqrt{n}$.\\

As it is always the case for mixture models, the (composite) log-likelihood  \eqref{eq:pseudologlik1} cannot be computed exactly (except for very small sample sizes).
Approximate   computation   of  the   estimators   in
\eqref{eq:estim_alpha_beta} can  be done using  an  \textsc{em}
algorithm  \citep{DLR}. This procedure is presented in Section~\ref{sec:EM_triplet}. 
It is  known \citep{Wu} that,  under reasonable
assumptions,  the \textsc{em} algorithm will  give a  solution converging  to the
estimators defined by \eqref{eq:estim_alpha_beta}, as the number of iterates grows to infinity.

\section{Weighted random graphs}\label{sec:weighted}
In  this  section,  we  focus   on  a  particular  instance  of  model
\eqref{eq:model_gen} for weighted random graphs. The observations are random variables $\{X_{ij}\}_{1\leq i< j\le n}$ that are either equal to $0$, indicating the absence of an edge between nodes $i$ and $j$, or a non-null real number, indicating the weight of the corresponding edge. 
We   still   assume  that   conditional   on   the  latent   structure
$\{Z_i\}_{1\le  i\le n}$, the  random variables  $\{X_{ij}\}_{1\leq i<
  j\le n}$ are independent, and the distribution of each $X_{ij}$ only
depends  on $Z_i$  and  $Z_j$. We  now  further specify  the model  by
assuming the following form for this distribution
\begin{equation}\label{eq:model_weighted}
  \forall q,\ell \in \{1,\ldots,Q\}, \quad  X_{ij} | Z_{iq}Z_{j\ell}=1 \sim  p_{q\ell} f(\cdot, \theta_{q\ell}) +(1-p_{q\ell}) \delta_0(\cdot),
\end{equation}
where $\{f(\cdot,\theta), \theta \in  \Theta\}$ is a parametric family
of  distributions,  $\delta_0$  is   the  Dirac  measure  at  $0$  and
$p_{q\ell}  \in (0,1]$  are  sparsity parameters.  We  let $\mathbf  p
=\{p_{q\ell}\}$   and  $\boldsymbol   \theta=\{\theta_{q\ell}\}$.  The
conditional  distribution of  $X_{ij}$ is  thus a  mixture of  a Dirac
distribution at zero accounting for non present edges, with proportion
given by  the sparsity  parameter $\mathbf{p}$ (which  can be  $1$ in
case of a complete weighted  graph) and a parametric distribution with
density $f$ that gives the weight of present edges. 
We focus on two different sparsity structures
\begin{itemize}
\item either  the sparsity is constant across  the graph: $p_{q\ell}=
  p, \forall 1\le q,l \le Q$; 
\item  or  the  sparsity  parameters model  an  affiliation  structure
  $p_{q\ell}=\alpha  1_{q=\ell} + \beta  1_{q\neq \ell}$  and we assume $\alpha
  \neq \beta$. 
\end{itemize}
We moreover assume that we know the sparsity structure type. 
In any  case, the connectivity  parameter $\boldsymbol \theta$ is assumed  to take
exactly two different values
\begin{equation*}
\forall q,\ell \in \{1,\ldots,Q\}, \theta_{q\ell} =
\left\{
  \begin{array}{ll}
\theta_{\text{in}}& \text{ if } q=\ell, \\
\theta_{\text{out}} & \text{ if } q\neq \ell ,      
  \end{array}
\right. 
\end{equation*}
with $\theta_{\text{in}}\neq \theta_{\text{out}}$.
For identifiability  reasons, we also constrain  the parametric family
$\{f(\cdot,\theta), \theta \in \Theta\}$ such that any distribution in
this set admits a continuous cumulative distribution function (c.d.f.)
at zero. Indeed, if this were  not the case, it would not be possible
to distinguish between a zero weight and an absent edge. Note that this model satisfies the affiliation assumption given by \eqref{eq:model_affil}.
Here, the complete data log-likelihood simply writes
\begin{multline}\label{eq:loglik2}
  \mathcal{L}_{X,Z}(\boldsymbol \pi, \mathbf p, \boldsymbol \theta) = \log \mathbb{P}_{\boldsymbol \pi,  \mathbf p, \boldsymbol \theta}(\{X_{ij}\}_{1\le i<j\le n}, \{Z_i\}_{1\leq i \le n}) = \sum_{i=1}^n \sum_{q=1}^{Q} Z_{iq} \log \pi_q \\
+ \sum_{1\le i<j \le n}\sum_{1\le q,\ell \le Q} Z_{iq}Z_{j\ell}\Big\{ 1_{X_{ij}\neq 0}( \log f(X_{ij},\theta_{q\ell}) + \log p_{q\ell}) + 1_{X_{ij}= 0}\log (1-p_{q\ell})\Big\}. 
\end{multline}
We now  give some examples of  parametric families $\{f(\cdot,\theta),
\theta \in \Theta\}$ that could be used as weights (or values) on the edges.

\begin{example}\label{ex:normal}
  Let $\theta=(\mu,\sigma^2) \in \Rset\times(0,+\infty)$ and consider $f(\cdot,\theta)$ the density of the Gaussian distribution with mean $\mu$ and variance $\sigma^2$.
\end{example}

\begin{example}\label{ex:poisson}
Let $\theta \in (0,+\infty)$ and consider $f(\cdot,\theta)$ the density (with respect to the counting measure) of the Poisson distribution, with parameter $\theta$, truncated at zero. Namely,
\begin{equation*}
 \forall k \ge 1 , \quad   f(k,\theta) = \frac {\theta^k} {k !}(e^{\theta}-1)^{-1} .
\end{equation*}
\end{example}

Note that in the above  example, the Poisson distribution is truncated
at zero because, as previously  mentioned, it would not be possible to
distinguish a zero-valued weight from an absent edge. Sparsity of the
graph  is modeled  through  the  parameter $\mathbf  p$  only and  the
density $f(\cdot,\theta)$ concerns weights on present edges.

\begin{figure}[htbp]
  \centering
\begin{tabular}{c  c}
%\hline
%\hline 
\includegraphics[width=3.5cm]{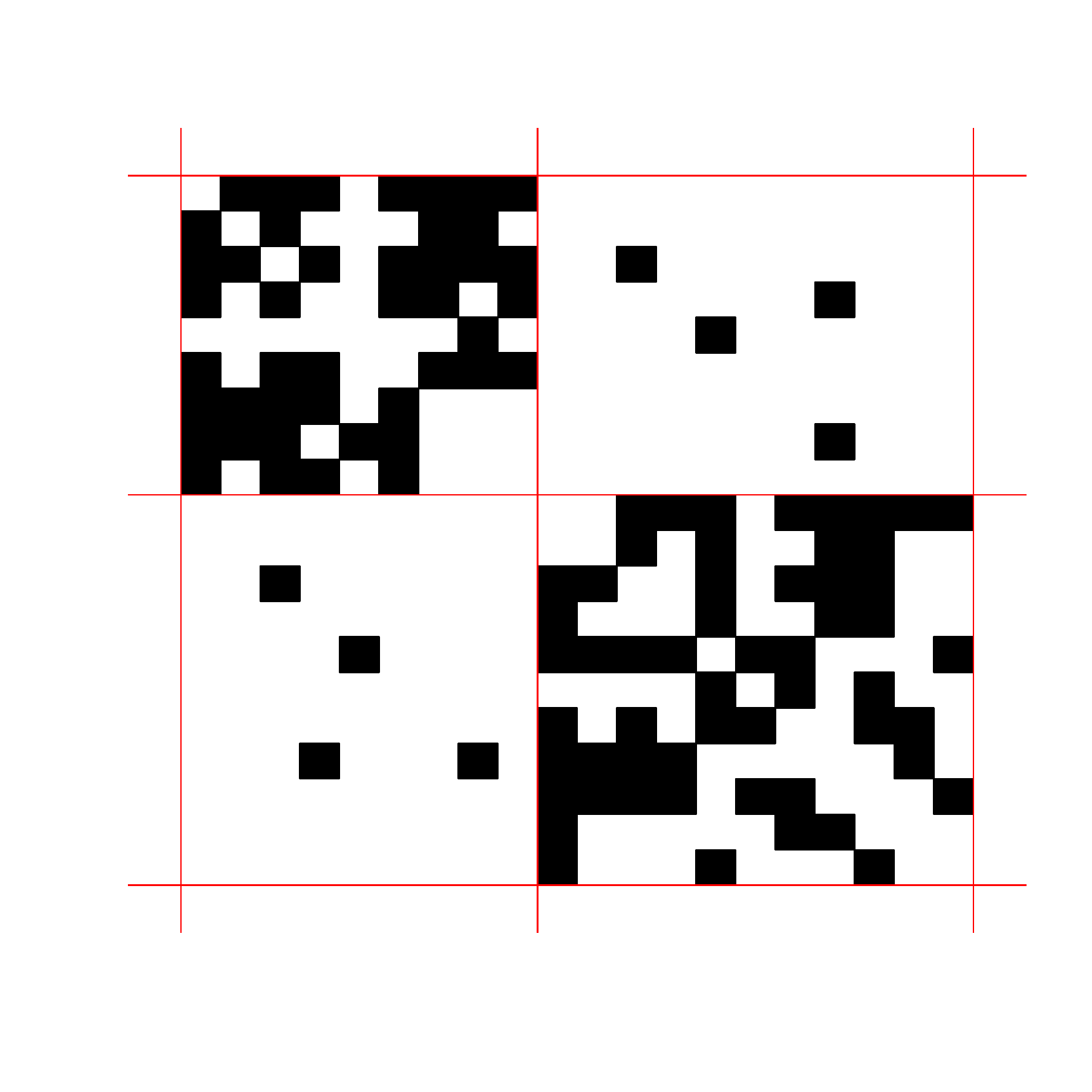} 
\includegraphics[width=3cm]{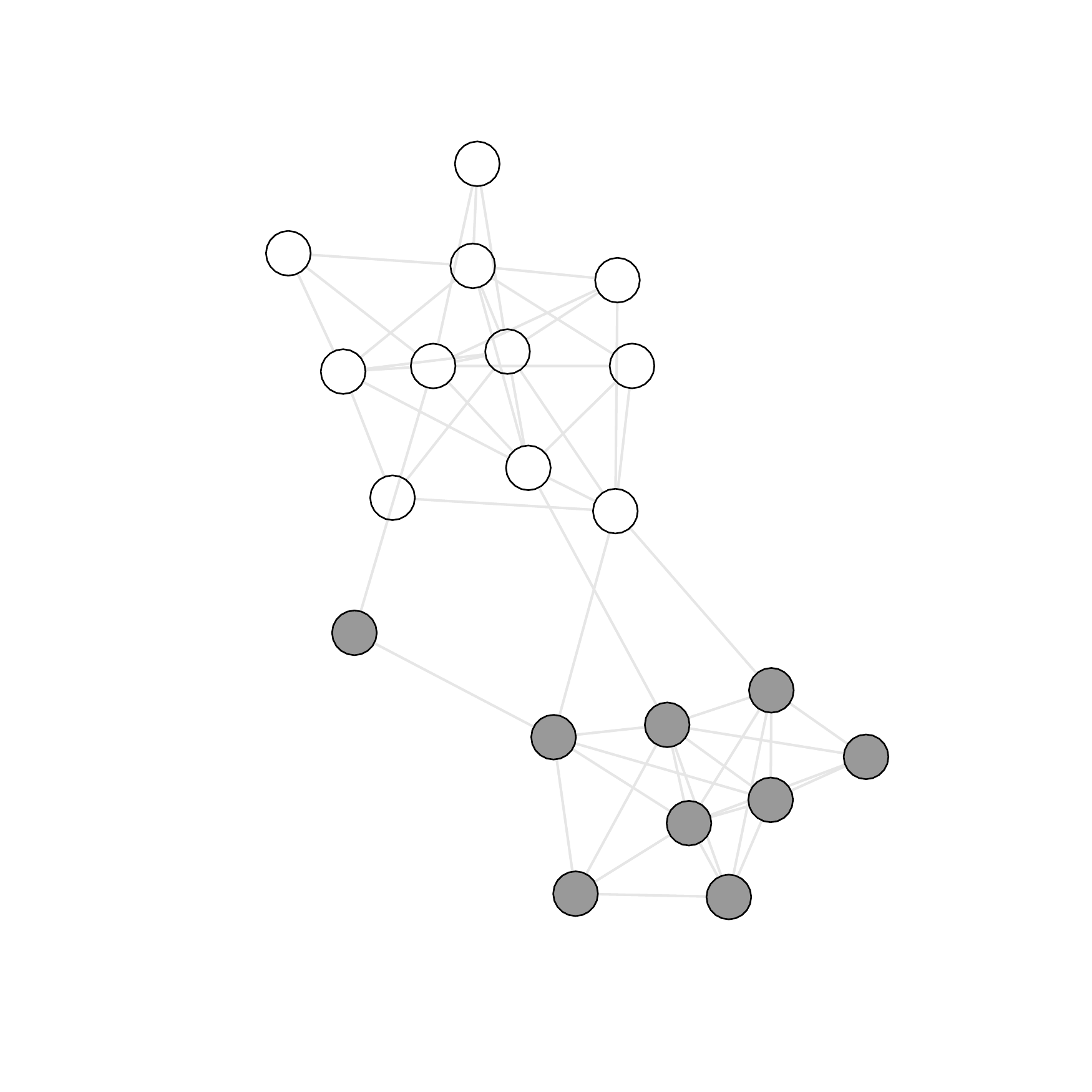} &
 \includegraphics[width=3.5cm]{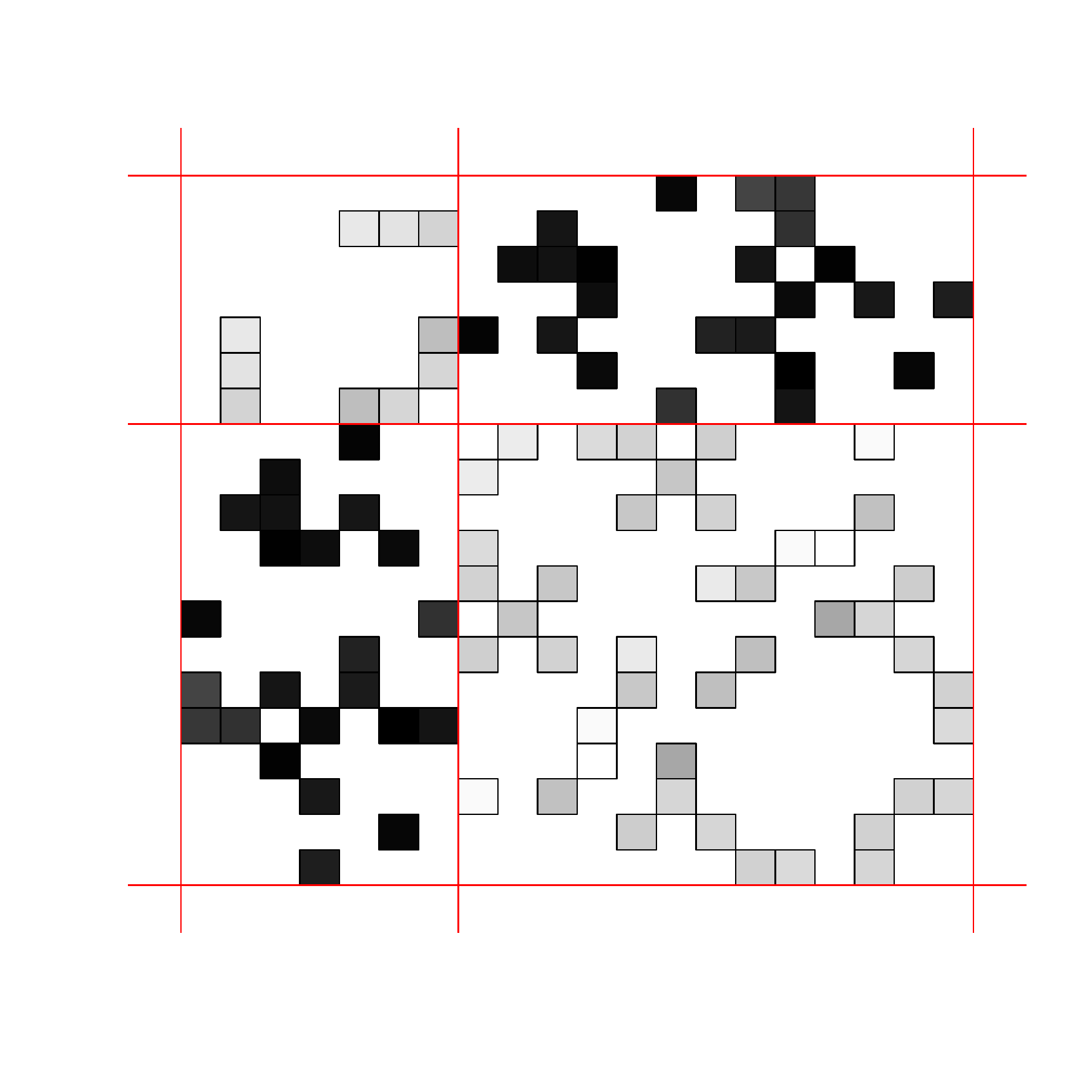} 
 \includegraphics[width=3cm]{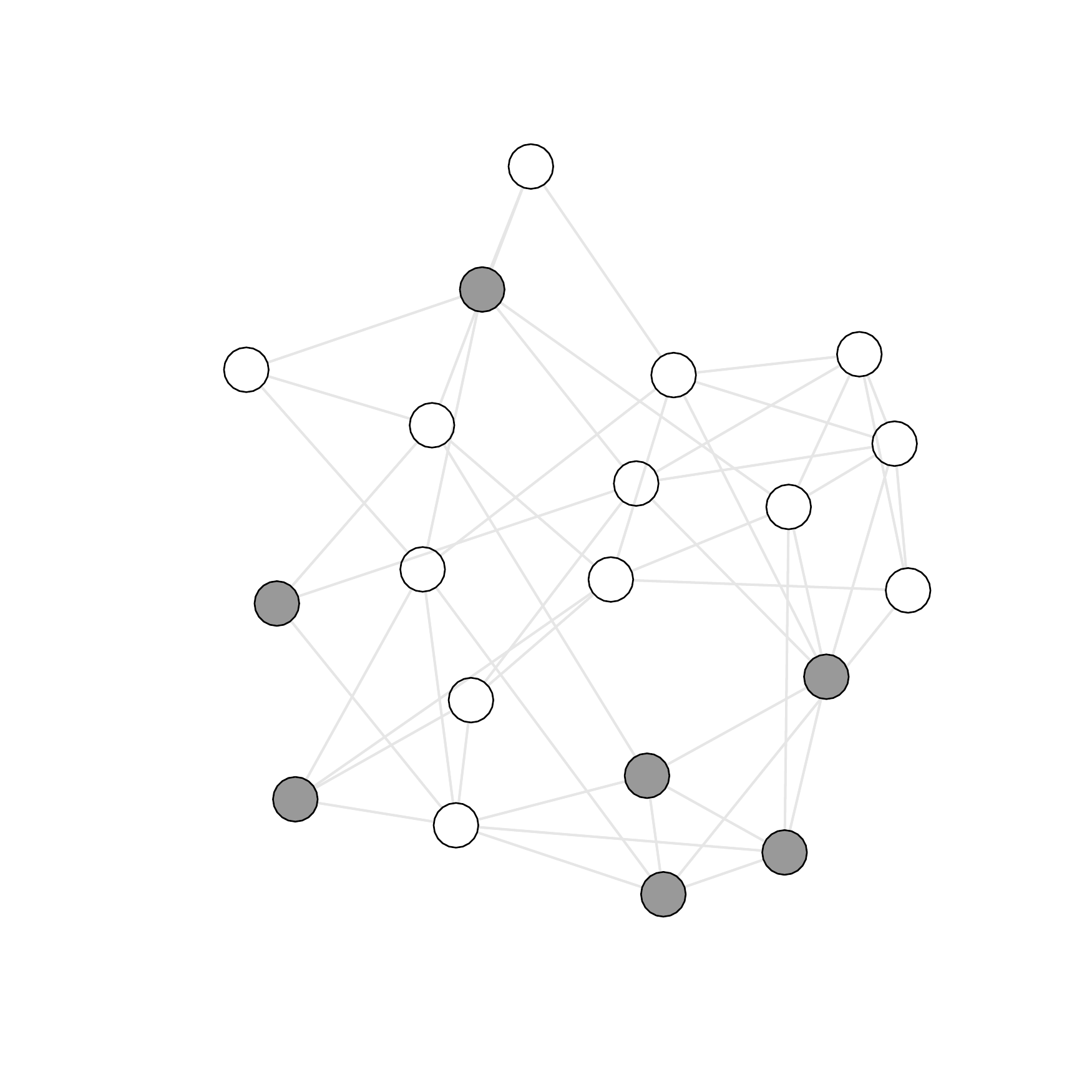} \\ 
 Binary affiliation model & Weighted affiliation model \\ 
\end{tabular}
   \caption{Simulation of binary (left) and Gaussian weighted (right) graphs with two
     classes  and 20  nodes. In  each case,  the picture  displays the
     graph representations with vertices colored according to classes, 
     as well as  the adjacency matrices where each entry $X_{ij}$  of the matrix is the  binary/weight value of the     edge between vertex  $i$ and vertex $j$. The  rows and columns of
     these matrices are organized according to the classes.}
  \label{fig:example}
\end{figure}

Figure~\ref{fig:example} illustrates the difference between  binary and  weighted random  graph affiliation models.
For example the weighted graph of Figure~\ref{fig:example}
displays no binary affiliation structure: if the weights where truncated using the function $x\to 1_{x\neq 0}$, we would not obtain that the two groups have different intra-group and inter-group connectivities. 
This means that classical community clustering algorithms would fail to find any meaningful  structure on this type of  graph. 

To  our knowledge, this model  has never been proposed
in this  form in the literature.   In particular, the  closest form is
given in \cite{Maria_Robin_Vacher} who do not introduce a possible Dirac mass at
zero to enable sparsity of the graph. \\

Let us now describe our estimation procedure based on $M$-estimators and a composite likelihood criterion. 
We proceed  in two  steps and first  estimate the  sparsity parameter,
relying on an induced binary random graph. In a second step, we plug-in this estimator and
focus  on  the  connectivity  parameters
$\boldsymbol \theta$ by relying only on the present edges.

\paragraph*{Estimating the sparsity parameter.}
Let us first consider the case where $p_{q\ell}=p$.
Then, we naturally estimate the sparsity parameter $p$ by 
\begin{equation*}
\hat p_n=\frac 2 {n(n-1)} \sum_{1\le i<j\le n} 1_{X_{ij}=0} .
\end{equation*}

The  consistency, as well  as well as the rate of convergence of  this estimator
follows from Theorem~\ref{thm:main}. 

In   the  case   where   the  sparsity   parameter  rather   satisfies
$p_{q\ell}=\alpha 1_{q=\ell} +\beta  1_{q\neq \ell}$, with $\alpha\neq
\beta$,  we rely on
the   underlying   binary    random   graph   (obtained   by   setting
$Y_{ij}=1_{X_{ij}\neq    0}$)    and     apply    the    results    of
Sections~\ref{sec:bin_poly} or \ref{sec:bin_multidim} to consistently 
estimate $\alpha$ and $\beta$.

\paragraph*{Estimating the connectivity parameter $\boldsymbol \theta$.}

The \emph{present} edges $X_{ij}$ (where $i,j$ are such that $X_{ij}\neq 0$) are non independent random
variables, distributed according to  a simple univariate mixture model
$ \sum_{q\ell}\pi_q\pi_\ell p_{q\ell} f(\cdot,\theta_{q\ell})$. 
For classical distributions $f(\cdot,\theta)$, it is possible to estimate the
connectivity parameters $\theta_{q\ell}$ of this univariate mixture directly. 
In  fact,  we  prove that  as  soon  as  the parameters $\{\theta_{q\ell}\}$ are  uniquely
identified  from  the  mixture $  \sum_{q\ell}\pi_q\pi_\ell  p_{q\ell}
f(\cdot,\theta_{q\ell})$   and   for   regular   parametric   families
$\{f(\cdot,\theta), \theta \in \Theta\}$, a consequence of Theorem~\ref{thm:main} is
that  maximizing a composite likelihood  of the  set  of present  edge
variables provides a consistent estimator of the parameters. 
Let us introduce the needed assumptions.

\begin{hyp} \label{hyp:ident}
  The  parameters  of  finite  mixtures  of  the  family  of  measures
  $\mathcal{F}= $ $\{ f(\cdot ;\theta) ;$ $ \theta \in \Theta\}$ are identifiable (up to label swapping). In other words, for any integer $m\ge 1$, 
  \begin{equation*}
\text {if }  \quad 
\sum_{i=1}^{m} \lambda_i f(\cdot, \theta_i) =  \sum_{i=1}^{m} \lambda_i'  f(\cdot, \theta_i') 
\quad 
\text{ then } \quad \sum_{i=1}^{m} \lambda_i \delta_{\theta_i} (\cdot) = \sum_{i=1}^{m} \lambda_i' \delta_{\theta_i'} (\cdot) .
  \end{equation*}
\end{hyp}

\paragraph*{Example \ref{ex:normal}, \ref{ex:poisson} (continued).}
\emph{Note that  both the families  of Gaussian   and truncated
  Poisson densities  satisfy  Assumption~\ref{hyp:ident}.  More  generally,  a wide range of  parametric  families  of  densities  on  $\Rset$  satisfy 
  Assumption~\ref{hyp:ident} \citep[see Section 3.1 in][for more details]{Titterington}.}  \\

The  next assumption  deals with  regularity conditions  on the
model. Note that this assumption could be weakened using the concept of differentiability in quadratic mean \citep[see for instance][]{VDV}.

\begin{hyp}\label{hyp:regular}
The functions $\theta \mapsto f(\cdot,\theta)$ are twice continuously differentiable on $\Theta$.
\end{hyp}

The last  assumption is  only technical and  not very  restrictive. It
requires the parameter set to be compact and could be weakened at the cost of some technicalities.

\begin{hyp}
  \label{hyp:compact2}
Assume  that there  exists some  $\delta >0$ and some compact subset $\Theta_c\subset \Theta$ such that  the parameter
space is restricted to the set $ \{(\boldsymbol \pi , \mathbf p, \boldsymbol \theta)  ; 
\forall 1\le q \le Q, \pi_q\ge \delta , \mathbf p \in [\delta, 1-\delta], 
\boldsymbol \theta\in \Theta_c \}$. 
\end{hyp}

Now, each \emph{present} edge variable $X_{ij}$ such that $X_{ij}\neq 0$ is distributed according to the mixture $\sum_{1\le q,\ell\le Q} \pi_q\pi_\ell p_{q\ell} f(\cdot;
\theta_{q\ell}) $.  As there are  only two different components in this
mixture, we express it in the more convenient form 
\begin{multline}\label{eq:weighted_mixture}
 \pr_{\boldsymbol \pi,
    \mathbf p, \boldsymbol \theta} (X_{ij}) = \Big\{ \sum_{q=1}^Q \pi_q^2p_{qq} \Big \} f(X_{ij}; \theta_{\text{in}}) +
  \Big \{\sum_{1\le q \neq \ell\le Q} \pi_q\pi_\ell p_{q\ell} \Big \} f(X_{ij}; \theta_{\text{out}})\\
:=\gamma_{\text{in}} f(X_{ij}; \theta_{\text{in}}) +\gamma_{\text{out}} f(X_{ij}; \theta_{\text{out}}) .
 \end{multline}
We consider   a  \emph{composite} log-likelihood   of   present   edges
%$\{X_{ij}; X_{ij}\neq 0, 1\le i<j\le n\}$ 
defined by 
\begin{equation}\label{eq:pseudologlik2}
  \mathcal{L}^{\text{compo}}_{X}(\boldsymbol    \pi,    \mathbf    p,
  \boldsymbol \theta)
%= \sum_{1\le i<j\le  n }  1_{X_{ij}\neq 0}\log \pr_{\boldsymbol \pi,  \mathbf p, \boldsymbol \theta} (X_{ij}) \\
= \sum_{1\le i<j\le  n }\log (\gamma_{\text{in}} f(X_{ij}; \theta_{\text{in}}) +\gamma_{\text{out}} f(X_{ij}; \theta_{\text{out}}) ). 
\end{equation}
We stress that  this quantity is not derived from  the marginal of the
model  complete data likelihood  (expressed in  \eqref{eq:loglik2})  and is
simpler.  
%This composite log-likelihood acts as if the random variables $X_{ij}$ were independent, which is not the case. 
We now define estimators as  
\begin{equation}
  \label{eq:estim_gamma_rho}
 \hat   {\boldsymbol  \theta}_n =\{\hat \theta_{\text{in}},\hat \theta_{\text{out}}\}  =   \argmax_{\boldsymbol  \theta   }
 \mathcal{L}^{\text{compo}}_{X}(\boldsymbol \pi, \hat {\mathbf p}_n, \boldsymbol \theta),
\end{equation}
where $\hat {\mathbf p}_n$ is  a preliminary step estimator of $\mathbf p$. Note that
due to the label swapping issue on the hidden states, we estimate the set of
values  $\{\theta_{\text{in}},\theta_{\text{out}}\}$  and  can  not  distinguish  $\theta_{\text{in}}$  from
$\theta_{\text{out}}$. Section~\ref{sec:latent} further deals with this issue. 

We are now able to prove the following theorem. 

\begin{theorem}\label{thm:cv_weighted}
In the model defined by \eqref{eq:model_gen} and \eqref{eq:model_weighted}, under Assumptions \ref{hyp:ident},  \ref{hyp:regular} and \ref{hyp:compact2}, 
the  set  of unordered $M$-estimators $  \hat
{\boldsymbol \theta}_n =\{\hat \theta_{\text{in}}, \hat \theta_{\text{out}} \}$ defined by \eqref{eq:estim_gamma_rho} 
is consistent, as the sample size $n$ grows to infinity. 
Moreover, the rate of this convergence is at least $1/\sqrt{n}$ and increases to $1/n$ in the particular case of equal group proportions \eqref{eq:model_equalgroup}. 
\end{theorem}

The proof mainly relies on the consistency of the normalized criterion 
\eqref{eq:pseudologlik2}.  This
point is a direct consequence of Theorem~\ref{thm:main}.  Then, from the criterion consistency, the
identifiability  and the  regularity assumptions,  one can  derive the
consistency of the corresponding $M$-estimator from   classical  theory \citep{VDV,Wald}. 

As already noted in the case of Theorem~\ref{thm:multiv_bernoulli}, our result establishes a rate of convergence equal at least to $1/\sqrt{n}$. The simulations (Section~\ref{sec:simus}) seem to indicate that the rate may in fact be faster, a phenomenon that may be due to the degeneracy of the limiting variance of $\sqrt{n}(\hat \theta_n-\theta)$.

As for $M$-estimators in the binary case, we shall approximate this maximum (composite) likelihood estimator using an \textsc{em} procedure \citep{DLR} whose convergence properties are   well-established   \citep{Wu}.   Contrarily   the procedure presented   in Section~\ref{sec:bin_multidim} where we  need to adapt the \textsc{em}
framework to our specific model, we rely here on the classical
\textsc{em} algorithm and thus do not recall it.

\section{Algorithms}\label{sec:algos}

In this section, we provide tools to implement the
procedures previously described, as well  as a complement on the issue
of recovering the latent structure of a graph. 

\subsection{EM algorithm with triplets}\label{sec:EM_triplet}

Let us  describe in this section the \textsc{em} algorithm developed to approximate the estimators defined by \eqref{eq:estim_alpha_beta}.

In the following, each set of three nodes $\{i,j,k\}$ corresponds to an index $\ui$ ranging over the set $\{1,\ldots,N\}$, where $N=n(n-1)(n-2)$ is the total number of triplets.   
We let $\bX^{\ui}=(X^{\ui,1},X^{\ui,2},X^{\ui,3})$  be one of the observed  triplets (namely each $X^{\ui,j}$ for $1\le j \le 3$ corresponds to some former random variable $X_{st}$ for some $1\le s,t\le n$) and $U_{\ui} \sim \mathcal{M}(1, \boldsymbol \gamma)$ is the vector encoding the  corresponding hidden state. Namely, $U_{\ui} \in \mathcal{V}_5$. 
We also denote by  $\tau_{\ui k}$ the posterior probability of node triplet $\ui$ being in state $k$, conditional on the observation $\bX^{\ui}$, namely  $\tau_{\ui k}= \mathbb{P}(U_{\ui k}=1|\bX^{\ui})$, for $1\le k\le 5$ and $1\le \ui\le N$. Moreover, we encode the fact that, conditional on the $5$ different hidden states of $U$, each coordinate of $\bX$ is distributed according either to $\mathcal{B}(\alpha) $ or $\mathcal{B}(\beta) $, using the following notation
\begin{equation*}
  \delta_{jk}= (\delta_{jk}^1,\delta_{jk}^2) =(1_{X^{\ui,j} |U_{\ui k}=1 \sim \mathcal{B}(\alpha) }, 1_{X^{\ui, j} |U_{\ui k}=1 \sim \mathcal{B}(\beta) }), 
\end{equation*}
for all $1\le j\le 3, \; 1\le k\le 5 $ and any  $1\le \ui \le N$.
Note that $\delta_{jk}$ is deterministic and that $\delta_{jk}\in \mathcal{V}_2$.
With these notations, we are in the situation where we consider a composite likelihood \eqref{eq:pseudologlik1} of random vectors $\{\bX^{\ui}\}_{1\le \ui \le N}$ from the mixture of $5$ different $3$-dimensional Bernoulli distributions, the latent classes being the random vectors $\{U_{\ui} \}_{1\le \ui\le N}$.

The  \textsc{em}-algorithm is intended to iteratively compute and optimize, with respect to $(\boldsymbol \gamma , \alpha,\beta)$ the function
\begin{equation*}
  Q((\boldsymbol \gamma, \alpha,\beta);(\boldsymbol {\gamma}^{(s)}, \alpha^{(s)},\beta^{(s)})) 
= \esp_{\boldsymbol {\gamma}^{(s)}, \alpha^{(s)},\beta^{(s)}} \big[
\log \mathbb{P}_{\boldsymbol \gamma, \alpha,\beta}(\{U_{\ui}\}_{1\leq \ui \le N},\{\bX^{\ui}\}_{1\le \ui\le N}) | \{\bX^{\ui}\}_{1\le \ui \le N} \big] ,
\end{equation*}
using the current value of the parameter $(\boldsymbol {\gamma}^{(s)}, \alpha^{(s)},\beta^{(s)})$.
If we let $\tau_{\ui k}^{(s)}= \mathbb{P}_{\boldsymbol {\gamma}^{(s)}, \alpha^{(s)},\beta^{(s)}}(U_{\ui k}=1|\bX^{\ui})$, we can write 
\begin{multline}
  Q((\boldsymbol \gamma, \alpha,\beta) ; (\boldsymbol {\gamma}^{(s)}, \alpha^{(s)},\beta^{(s)})) 
= \sum_{\ui=1}^N \sum_{k=1}^{5} \tau_{\ui k}^{(s)} \log  \gamma_k 
+ \sum_{ \ui =1}^N\sum_{ k=1}^5 \tau_{\ui k}^{(s)} \\
\times\sum_{j=1}^3 \delta_{jk}^1\Big\{X^{\ui, j} \log \alpha +(1-X^{\ui,j })\log(1-\alpha)\Big\} + \delta_{jk}^2 \Big\{X^{\ui, j} \log \beta +(1-X^{\ui,j})\log(1-\beta)\Big\} .  
\label{eq:Q}
\end{multline}
Starting from an initial value $(\boldsymbol {\gamma}^{(1)}, \alpha^{(1)},\beta^{(1)})$, the \textsc{em} algorithm proceeds in two iterative steps. At iteration $s$, the \textsc{e}-step computes the posterior distribution of $U_{\ui}$ conditional on $\bX^{\ui}$. Namely,  
\begin{equation*}
  \tau_{\ui k}^{(s)} = \mathbb{P}_{\boldsymbol {\gamma}^{(s)} , \alpha^{(s)},\beta^{(s)}}(U_{\ui k}=1 |\bX^{\ui}) =\frac{\gamma_k^{(s)} \prod_{j=1}^3 b(X^{\ui,j} , \delta_{jk}^1 \alpha^{(s)} +\delta_{jk}^2\beta^{(s)})} {\sum_{\ell=1}^5   \gamma_\ell^{(s)} \prod_{j=1}^3 b(X^{\ui,j} , \delta_{j\ell}^1\alpha^{(s)} +\delta_{j\ell}^2\beta^{(s)}) } ,
\end{equation*}
for every $1\le \ui\le N$ and every $1\le k \le 5$. 
By using \eqref{eq:Q}, we then get  the value of $Q((\boldsymbol \gamma, \alpha,\beta) ; (\boldsymbol {\gamma}^{(s)}, \alpha^{(s)},\beta^{(s)}))$. In the \textsc{m}-step, this quantity is maximized with respect to $(\boldsymbol \gamma, \alpha,\beta) $ and the maximizer gives the next value of the parameter $(\boldsymbol \gamma^{(s+1)}, \alpha^{(s+1)},\beta^{(s+1)})$. This step relies on  the following equations.  
\begin{equation*}
  \begin{array}{cl}
\gamma_k^{(s+1)} =& N^{-1} \sum_{\ui =1}^N \tau_{\ui k}^{(s)} , \quad  k=1,5 ,\\
\gamma_k^{(s+1)} =&  (3N)^{-1}\sum_{\ui =1}^N \tau_{\ui 2}^{(s)}+\tau_{\ui 3}^{(s)}+\tau_{\ui 4}^{(s)} , \quad k=2,3,4,\\
  \alpha^{(s+1)} =& \Big(\sum_{\ui =1}^N \tau_{\ui 1}^{(s)}( X^{\ui, 1}+X^{\ui, 2}+X^{\ui, 3}) + \tau_{\ui 2}^{(s)}X^{\ui, 3} + \tau_{\ui 3}^{(s)}X^{\ui, 2} + \tau_{\ui 4}^{(s)}X^{\ui, 1}  \Big) \\
& \Big( \sum_{\ui =1}^N 3\tau_{\ui 1}^{(s)}+ \tau_{\ui 2}^{(s)} + \tau_{\ui 3}^{(s)} + \tau_{\ui 4}^{(s)}\Big)^{-1} ,\\
 \beta^{(s+1)} =& \Big(\sum_{\ui=1}^N \tau_{\ui 2}^{(s)}(X^{\ui, 1}+X^{\ui, 2}) + \tau_{\ui 3}^{(s)}(X^{\ui, 1}+X^{\ui, 3}) + \tau_{\ui 4}^{(s)}(X^{\ui, 2}+X^{\ui, 3}) \\
&+\tau_{\ui 5}^{(s)}( X^{\ui, 1}+X^{\ui, 2}+X^{\ui, 3})\Big)
\Big(\sum_{\ui =1}^N 2\tau_{\ui 2}^{(s)} + 2\tau_{\ui 3}^{(s)}+ 2\tau_{\ui 4}^{(s)}+3\tau_{\ui 5}^{(s)}\Big)^{-1} .
  \end{array}
\end{equation*}
It should be noted that the sum over all the $N$ possible triplets reduces in fact to a sum over $8$ different possible patterns for the values of $\bX^{\ui}$. Indeed, the posterior probabilities $\tau_{\ui k}$ are constant across triplets with the same observed value.

\subsection{Unraveling the latent structure}\label{sec:latent}

The general method  we develop in this section  aims at recovering the
latent structure $\{Z_i\}_{1\le i \le  n}$ on the graph nodes. Indeed,
the procedures developed in the previous sections only focus on estimating the parameters and do not directly provide an estimate for the node groups. 

We rely here on  a simple method:   we    plug-in   the    estimators    obtained   from
the previous sections in the  complete data likelihood of the model
(namely  the likelihood of  the observations  $\{X_{ij}\}_{1\le i<j\le
  n}$ and the latent classes $\{Z_i\}_{1\le i \le n}$). As we  do not have estimates of the mixture
proportions  $\boldsymbol \pi$, we  simply remove  this part  from the
expression  of the complete data likelihood.  Then, we simply maximize this
criterion (which we call a classification likelihood) with respect to the latent structure
$\{Z_i\}_{1\le i  \le n}$. In  a latter step,  we then
  estimate   the  unknown   proportions  $\boldsymbol   \pi$   by  the
  frequencies observed on the estimated groups $\hat Z_i$.

\paragraph*{Criterion in the binary case.} In this setup, we introduce a criterion
$\mathcal{C}$, built on the complete data likelihood, 
where we plugged-in the estimators  of $\alpha$ and $\beta$ and removed
the dependency on $\boldsymbol \pi $. This criterion simply writes 
\begin{multline*}
  \mathcal{C}(\{Z_i\}_{1\le i \le n}) =
 \sum_{1\le i<j  \le n}\sum_{  q=1 }^Q Z_{iq}Z_{jq}  \big\{X_{ij} \log
\hat \alpha +(1-X_{ij})\log(1-\hat \alpha)\big\} \\
+ \sum_{1\le  i<j \le n}\sum_{1\le  q\neq \ell \le  Q} Z_{iq}Z_{j\ell}
\big\{X_{ij} \log \hat \beta +(1-X_{ij})\log(1-\hat \beta)\big\}.  
\end{multline*}

\paragraph*{Criterion in the weighted case.}
Let    us    recall     that    the    estimation    procedure    from
Section~\ref{sec:weighted} only recovers the set of unordered values
$\{\theta_{\text{in}}, \theta_{\text{out}}\}$. 
As  we know  these  parameters  up to  permutation  only, let  $\{\hat
\theta_1,\hat \theta_2\}$  be any  label choice for  the corresponding
estimators.   We  can   consider  two   different   criteria,  denoted
$\mathcal{C}^{1,2}$ and $\mathcal{C}^{2,1}$, as follows
\begin{multline*}
\mathcal{C}^{u,v}  ( \{Z_i\}_{1\le i \le n})=\\
\sum_{\substack{1\le i<j\le n\\1\le q\neq \ell \le Q}}
Z_{iq}Z_{j\ell} [1_{X_{ij\neq 0}} (\log f(X_{ij};\hat \theta_u)+\log \hat p_{q\ell}) +1_{X_{ij} =0}\log(1- \hat p_{q\ell})] \\ +
 \sum_{\substack{1\le i<j\le n\\ 1\le q \le Q}}
Z_{iq}Z_{jq} [1_{X_{ij\neq 0}} (\log f(X_{ij};\hat \theta_v)+\log \hat p_{qq}) +1_{X_{ij} =0}\log(1-\hat p_{qq})],  
\end{multline*}
where $\{u,v\}=\{1,2\}$. 
For each of these criteria, we can 
select the latent structure $ \hat Z^{u,v}=( \hat Z_1,\ldots,\hat Z_n)^{u,v}$ maximizing it. Then, choosing the couple $(u^\star,v^\star)$ maximizing the resulting quantity $\mathcal{C}^{u,v}  ( \hat Z^{u,v})  $
seems to be an interesting strategy. 
We thus finally define our estimated latent structure $(\hat Z_1,\ldots,\hat Z_n)$ as $ \hat Z^{u^\star,v^\star}$. \\

\paragraph*{Iterative estimation of the latent structure. }
In  any  case  (either binary  or  weighted),  we  propose to  use  an
iterative procedure to  compute the maximum $\hat Z$  of the criterion
$\mathcal{C}(\{Z_i\}_{i}) $. 
Starting from an initial value $Z^{(1)}=(Z_1^{(1)},\ldots, Z_n^{(1)})$ of the latent structure, we iterate the following steps. At step $s$, we (uniformly) choose a node $i_0$ and select $Z_{i_0}^{(s+1)}$ as
\begin{equation*}
Z_{i_0}^{(s+1)}  =   \argmax_{1\le  q\le  Q}  \mathcal{C}(\{Z_i^{(s)}\}_{i \neq i_0}, Z_{i_0}=q) 
\end{equation*}
while we let $Z_j^{(s+1)}=Z_j^{(s)}$ for $j\neq i_0$. At each time step, we  increase the classification likelihood $\mathcal{C}(\{Z_i\}_{1\le i \le n})  $ and thus the procedure eventually converges to a (local) maxima. By using different initial values $Z^{(1)}=(Z_1^{(1)},\ldots, Z_n^{(1)})$, we should finally find the global maxima. 
Once  we estimated  the latent  groups $\hat  Z_i$, we  may  obtain an
estimate  of the  group proportions  $\hat {\boldsymbol  \pi}$ by  simply taking  the
corresponding frequencies.
The           procedure           is           summarized           in
Function~\textbf{\texttt{latent.structure}}. \\

\begin{function}[H]
\dontprintsemicolon
\SetKwInOut{Input}{input}\SetKwInOut{Output}{output}
\Input{observed graph, parameter values}
\Output{latent structure and group proportions}

\BlankLine
Start from latent structure $\{Z_i\}$\; 
%Compute criterion $\mathcal{C}( \{Z_i\}_i)$\;
 \While{convergence is not attained}{
Choose node $i_0$\;
Replace $Z_{i_0}$ with $\argmax_{q} \mathcal{C}( \{Z_i\}_{i\neq i_0}, Z_{i_0}=q)$\;
}
Compute group proportions $\boldsymbol \pi$ from  $\{Z_i\}$\;
\BlankLine
\caption{latent.structure (graph, parameters)}\label{algo:function_latent}
\end{function}

\subsection{Complete algorithm description}

Algorithm~\ref{algo:main}  describes   the  procedures  for  analyzing
binary or weighted random graphs. We introduce a variable 'method' which can take three different values: 'moments' or 'tripletEM' in the binary case and 'weighted' for weighted random graphs. In the weighted case, we moreover use a second variable called 'sparsity' to indicate whether we estimate a global sparsity parameter $p$ ('sparsity=global') or two parameters $\alpha,\beta$ from an affiliation structure ('sparsity=affiliation').
The performances of the procedures proposed in the current section are
illustrated in the following one. \\

\begin{algorithm}[H]
  \dontprintsemicolon
\SetKwFunction{latent}{latent.structure}
 \BlankLine

 \If{method='moments'}{
Compute  $\hat m_i, i=1,2,3$ from \eqref{eq:estim_moment}\;
\CommentSty{// INITIALIZATION}\;
Start from latent structure $\{Z_i\}$ with  proportions $\boldsymbol \pi$ and compute $s_2,s_3$\;
 \While{convergence is not attained}{
 \CommentSty{// UPDATE PARAMETERS}\; 

   \If{$abs(\hat m_2-\hat m_1^2)<\epsilon$}{
   Compute $\alpha,\beta$ through \eqref{eq:bin_uniform_estim}\;
   }   
   \Else{Compute $\alpha,\beta$ through \eqref{eq:bin_non_uniform_estim}\;
   }
\CommentSty{// UPDATE LATENT STRUCTURE}\;
Apply \latent to the current parameter values\;
   }
 }

 \BlankLine
 \BlankLine

  \If{method='tripletEM'}{
Estimate $\alpha, \beta$ from \textsc{em} algorithm with triplets (Section~\ref{sec:EM_triplet})\;
%\CommentSty{// LATENT STRUCTURE} \;
Apply \latent to the parameter values\;
}

 \BlankLine
 \BlankLine

  \If{method='weighted'}{
\CommentSty{// SPARSITY PARAMETERS}\;
Transform weights $X_{ij}$ into binary variables $Y_{ij}=1_{X_{ij}\neq 0}$\;
  \If{sparsity='global'}{
Compute $\hat p=\frac 2 {n(n-1)} \sum_{i<j}Y_{ij}$\;
}
\Else{
Estimate $\alpha, \beta$ from \textsc{em} algorithm with triplets (Section~\ref{sec:EM_triplet})\;
}
\CommentSty{// CONNECTIVITY PARAMETERS} \;
Estimate $\{\theta_{\text{in}},\theta_{\text{out}}\}$ from \textsc{em}
algorithm with present edges (Section~\ref{sec:weighted})\;
\CommentSty{// LATENT STRUCTURE} \;
Apply \latent to the parameter values\;
}
 \BlankLine
\caption{Complete algorithm}
\label{algo:main}
\end{algorithm}

\section{Numerical experiments}\label{sec:simus}

We carried out a simulation study to examine the bias and variance of the proposed estimators. In the binary affiliation model, we also compared the performance of our proposal with 
the  variational  \textsc{em}   (\textsc{vem})  strategy  proposed  by
\cite{Daudin}. Note  that Gibbs sampling has already  been compared to
\textsc{vem} strategies in  \cite{Zanghi_online} and give very similar
results. Note also that the weighted affiliation model proposed here is original and we thus cannot compare our results in this case with any other existing implemented method. 

\subsection{Binary affiliation model}

\paragraph{Simulations set-up.} In  these experiments, we assumed that
edges  are  distributed  according  to the  binary  affiliation  model
described  in Section  \ref{sec:binary}.  The  data were  generated in
different settings, with  the number of groups $Q \in  \{ 2,5 \}$, the
number of  vertices $n \in \{20,50,100,500,1000\}$. For  each of these
cases, we created three settings corresponding to models with different ratios of intra and inter-group connectivity parameters (see Table \ref{tab:models}). Moreover, we considered two different cases: equal or free group proportions $\boldsymbol \pi$.

\begin{table}
\caption{ From left to right: low inter-group connectivity and strong intra-group connectivity (model 1),   strong inter-group connectivity and low intra-group connectivity (model 2),  model without structure  close to Erd\H{o}s-R\'enyi random  graph model  (model 3). The picture displays an example with $Q=2$ groups.}
 \centering
 \begin{tabular}{c c c c}
\hline
Model  & 1  & 2  & 3  \\ \hline \\
& \parbox{0.1\textwidth}{
                                         $\alpha=0.3$\\ 
                                         $\beta=0.03$} 
&  \parbox{0.1\textwidth}{$\alpha=0.03$\\ 
                                         $\beta=0.3$} 
&  \parbox{0.1\textwidth}{$\alpha=0.55$\\ 
                                         $\beta=0.45$} \\ 
& \includegraphics[width=0.2\textwidth]{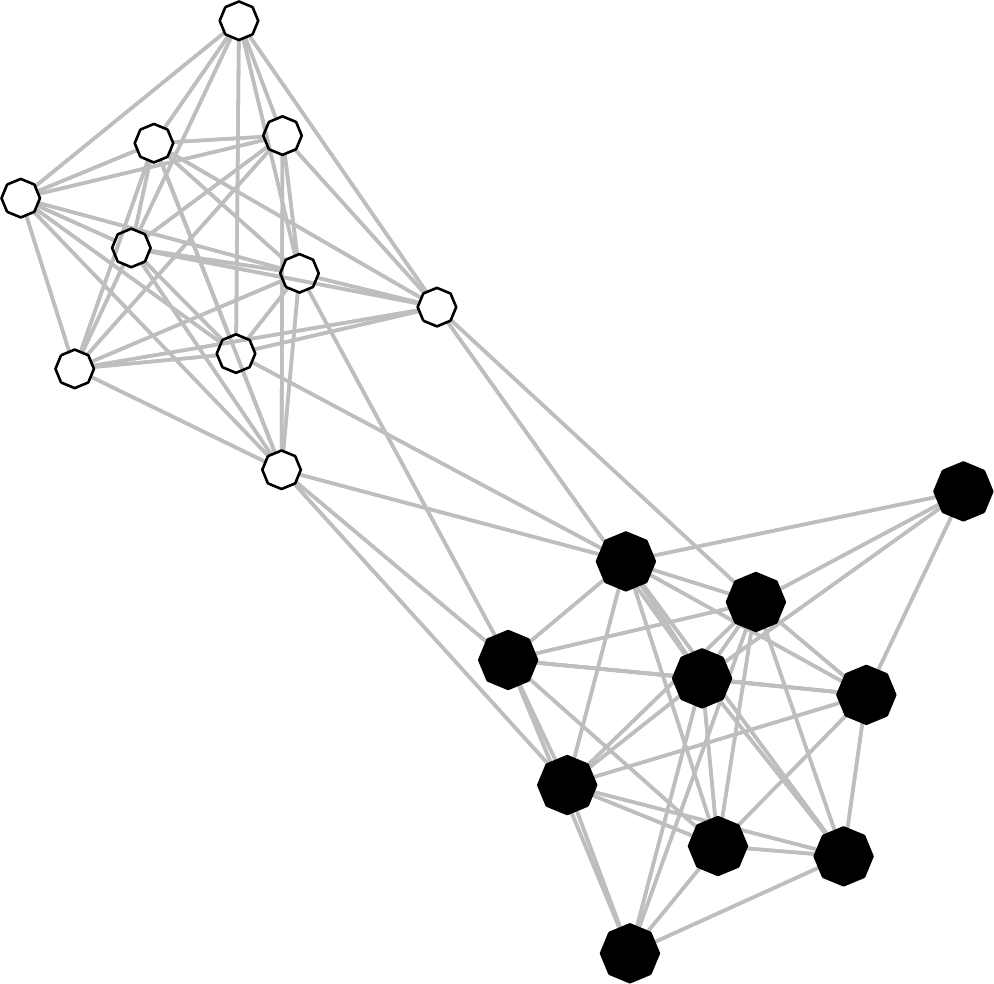}  &
\includegraphics[width=0.2\textwidth]{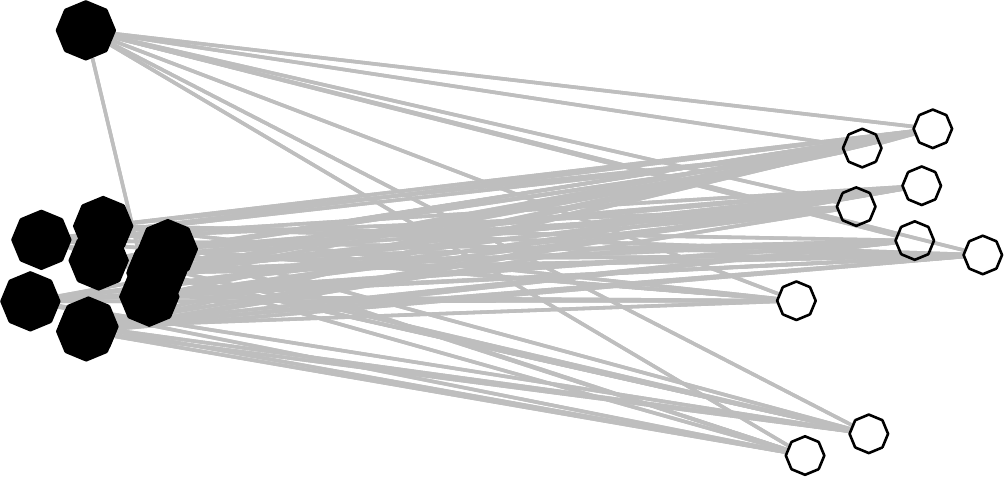} &
\includegraphics[width=0.2\textwidth]{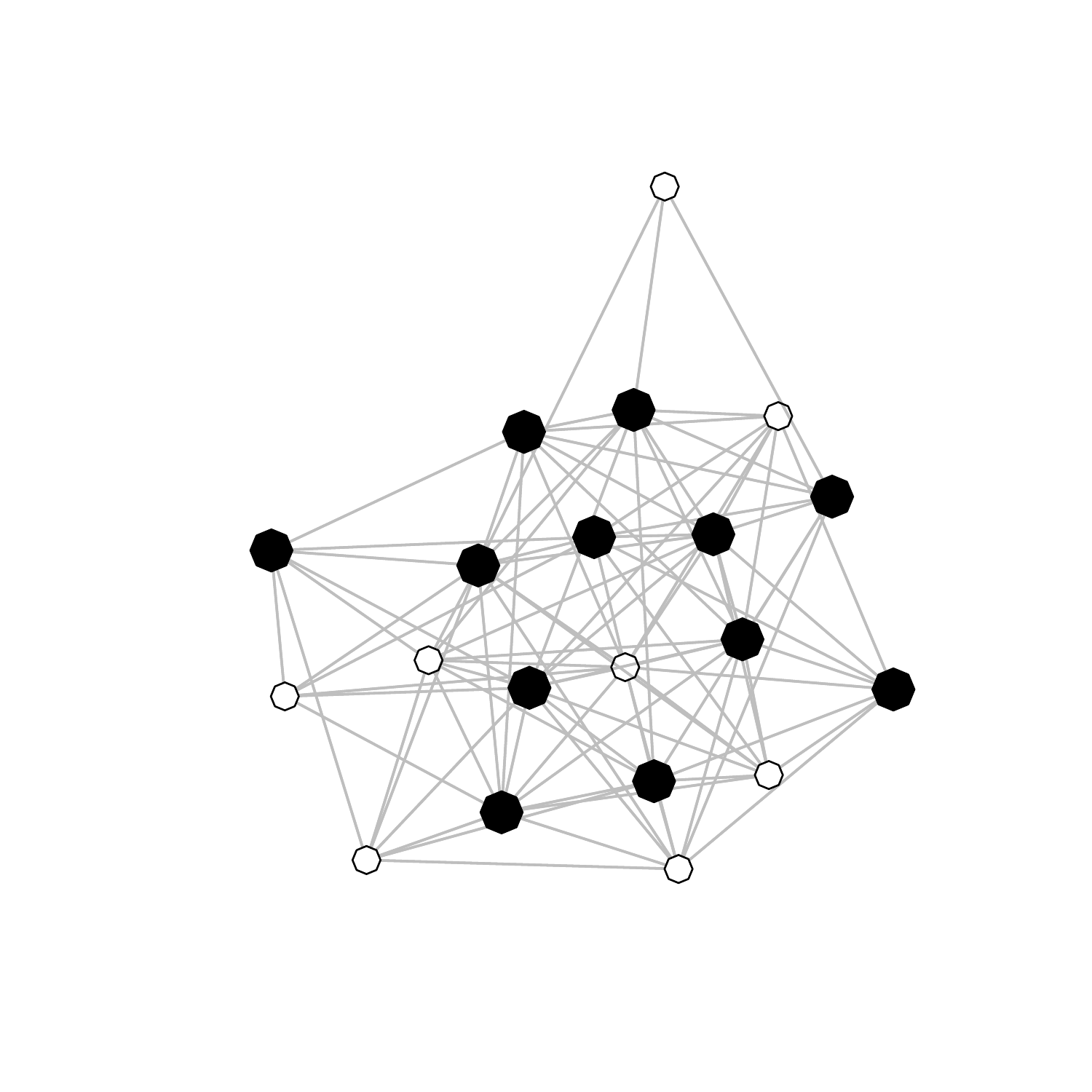} \\ 
\end{tabular}
\label{tab:models} 
\end{table}

In  each of  these settings,  we applied  three different  methods for
estimating the  model parameters: the  moment method (corresponding to Section~\ref{sec:bin_poly}), the      triplet     \textsc{em}     method      (corresponding     to
Section~\ref{sec:bin_multidim})  and  the  variational  \textsc{em}
strategy  (\textsc{vem})  proposed  by  \cite{Daudin},  that  we  
adapted to constrain it to an affiliation structure. The results for equal or free group proportions were similar and we thus present only the equal group proportions case.   

Figure~\ref{fig:contourplot}  shows the estimated  density (over $100$
graphs simulations) of the estimators $\hat{\alpha}$ and $\hat{\beta}$
for  the three algorithms  and the  three models  for graphs  with 500
vertices. 
We see  that for a given model the three methods produce estimators with similar densities. In particular,  the estimators of  $\alpha$ and $\beta$ seem to have little or no bias and  the variances are of the same order of magnitude for the three estimation methods. As the behaviour of the estimators of $\alpha$ and $\beta$ are comparable over all the simulations, we focus the discussion on the estimation of the parameter $\alpha$.

\begin{figure}[!htbp]
\centering
 \begin{tabular}{@{}l@{}r@{\hspace{0.01em}}l@{}}
&    \raisebox{-1.5cm}{\includegraphics[width=0.55\textwidth]{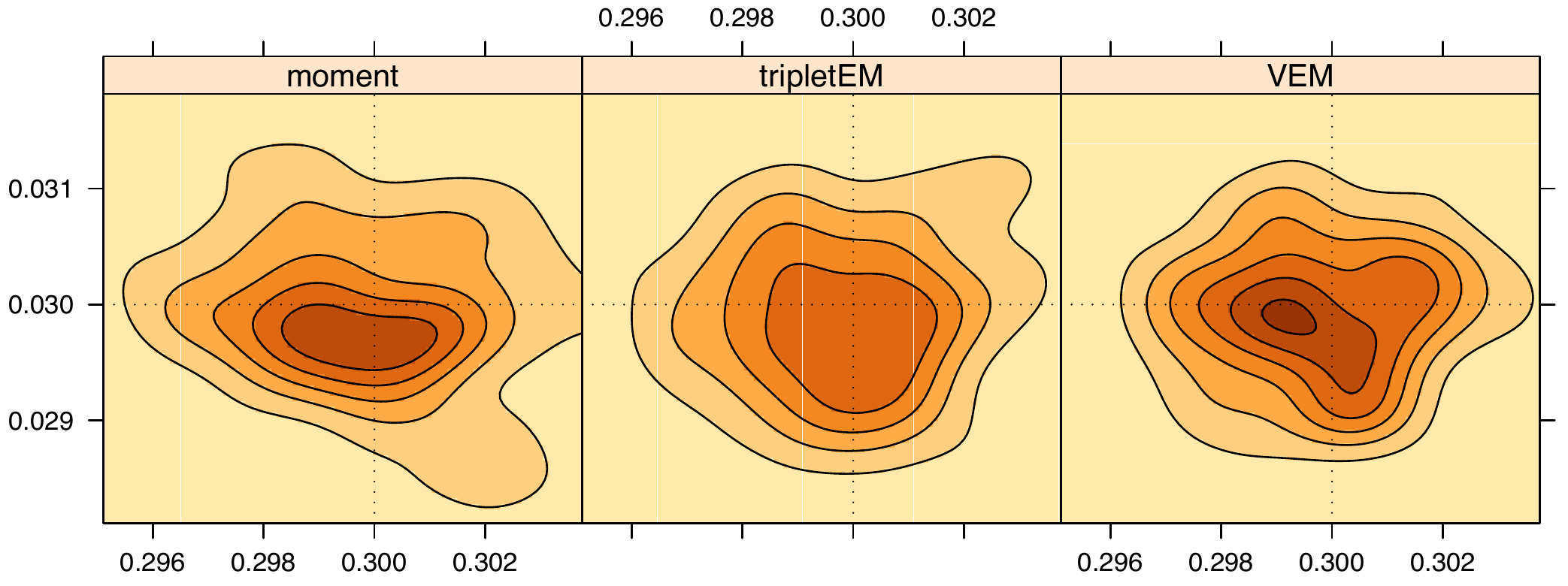}} 
&     \rotatebox{270.0}{\makebox[0.3cm]{Model 1}} \\
      \rotatebox{90.0}{\makebox[0.3cm]{$\hat{ \beta}$}} 
&    \raisebox{-1.5cm}{\includegraphics[width=0.55\textwidth]{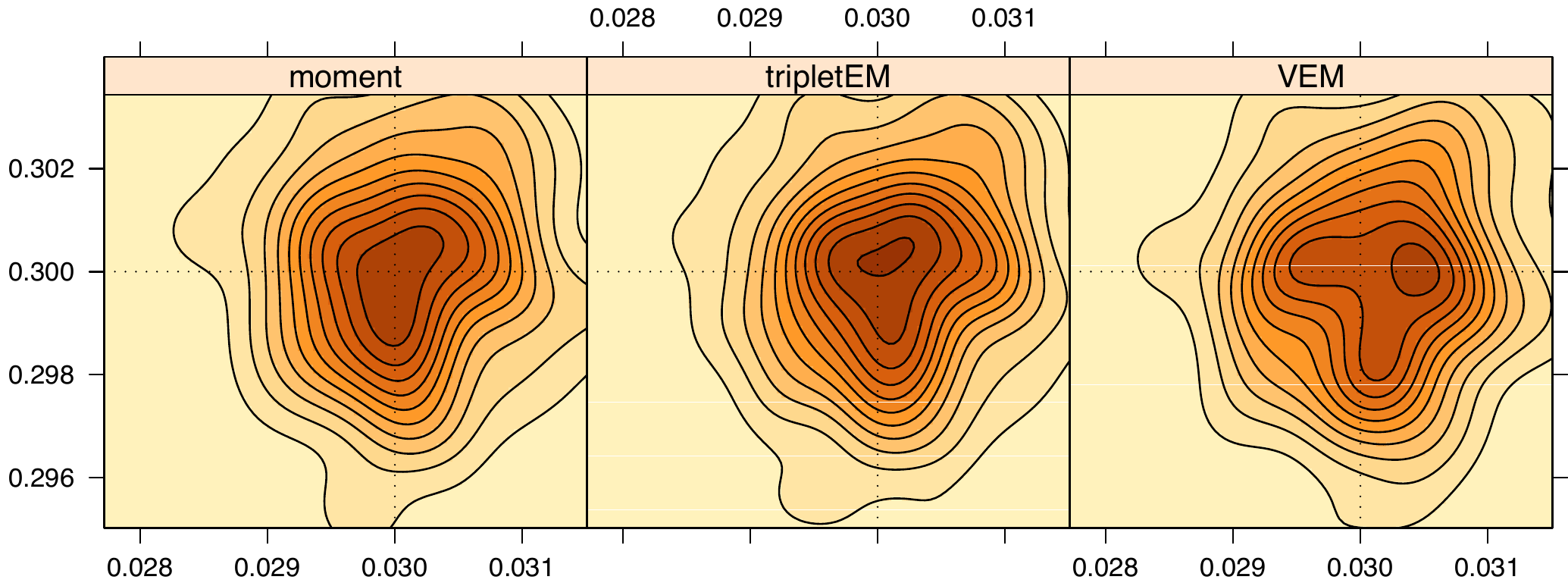}} 
&   \rotatebox{270.0}{\makebox[0.3cm]{Model 2}}\\
&    \raisebox{-1.5cm}{\includegraphics[width=0.55\textwidth]{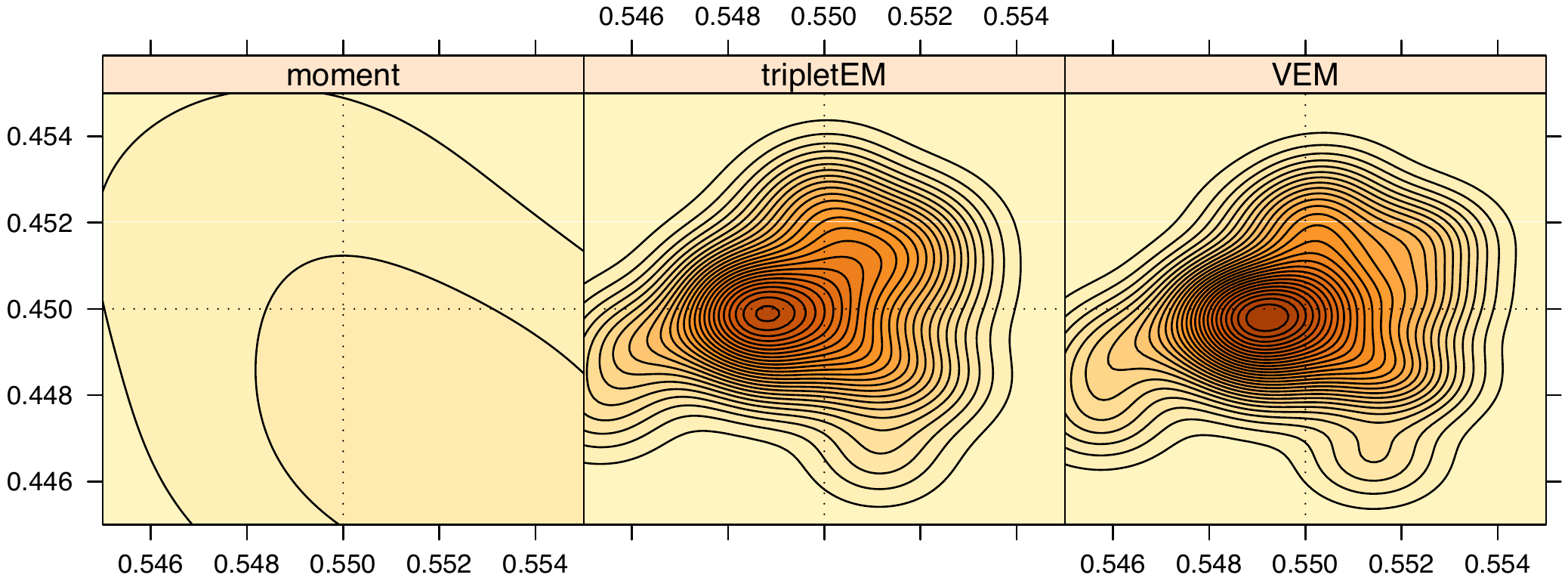}} 
&  \rotatebox{270.0}{\makebox[0.3cm]{Model 3}}\\
& \multicolumn{2}{c}{$\hat{ \alpha}$} 
 \end{tabular}
 \caption{ Empirical joint distribution of the estimators $\hat{\alpha}$ and $\hat{ \beta}$, computed over 100 simulations of graphs with 500 vertices, $Q=2$ groups and equal group proportions. The dotted lines show the true values of $\alpha$ and $\beta$.}
    \label{fig:contourplot}
\end{figure}

Figure~\ref{fig:mean_varplot} (top) displays the estimations of $\alpha$
averaged over 100 graph simulations as a function of the number of graph
vertices in log-scale.
  For all three models, we see that all
algorithms do produce unbiased estimation when the number of vertices
is large enough. In addition to the asymptotically unbiased
estimation, we observe an agreement in the sign of the bias among all
algorithms, when the graphs are small. For example, when estimating
$\alpha$ in model $1$ where $(\alpha=0.3, \beta=0.03)$, all methods
under-estimate  $\alpha$ and over-estimate $\beta$.

\begin{figure}[htbp!]
\centering
 \begin{tabular}{c c c c}
&  Model 1 & Model 2  & Model 3  \\ 
\hspace{-.5cm}\rotatebox{90}{\makebox[0.05cm]{$\hat{ \alpha}$ }} 
&    \raisebox{-2.1cm}{\includegraphics[width=0.3\textwidth]{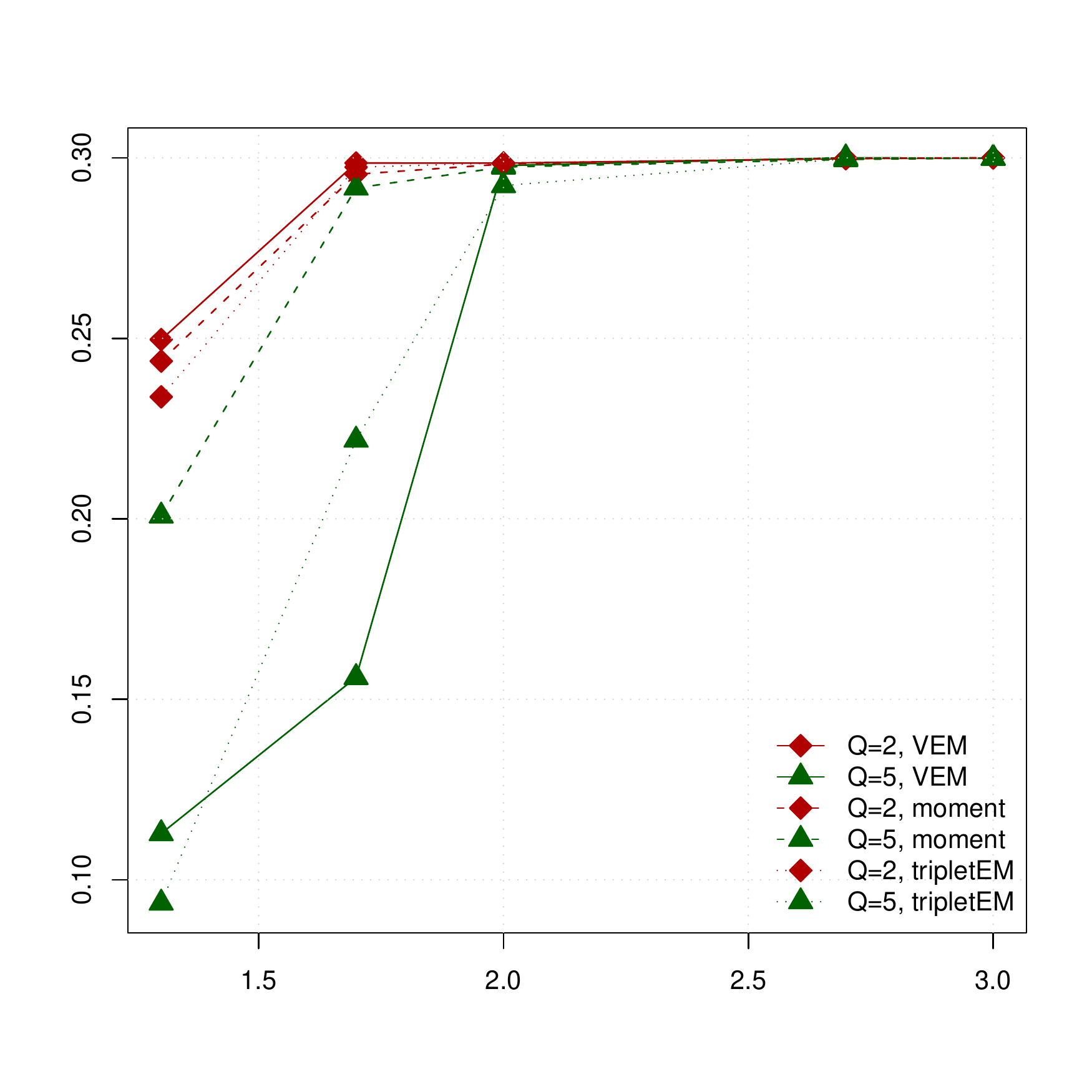}} 
&    \raisebox{-2.1cm}{\includegraphics[width=0.3\textwidth]{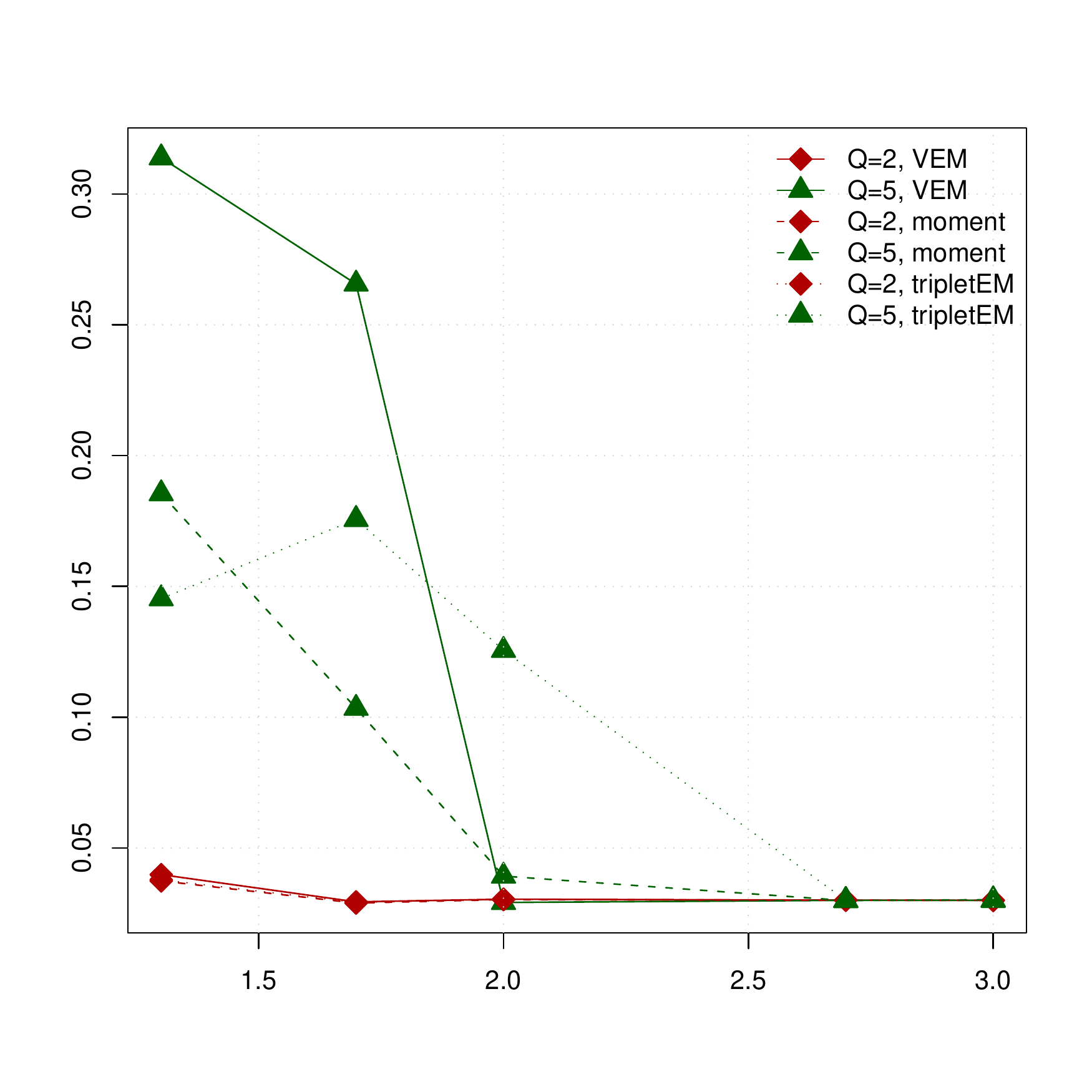}} 
&    \raisebox{-2.1cm}{\includegraphics[width=0.3\textwidth]{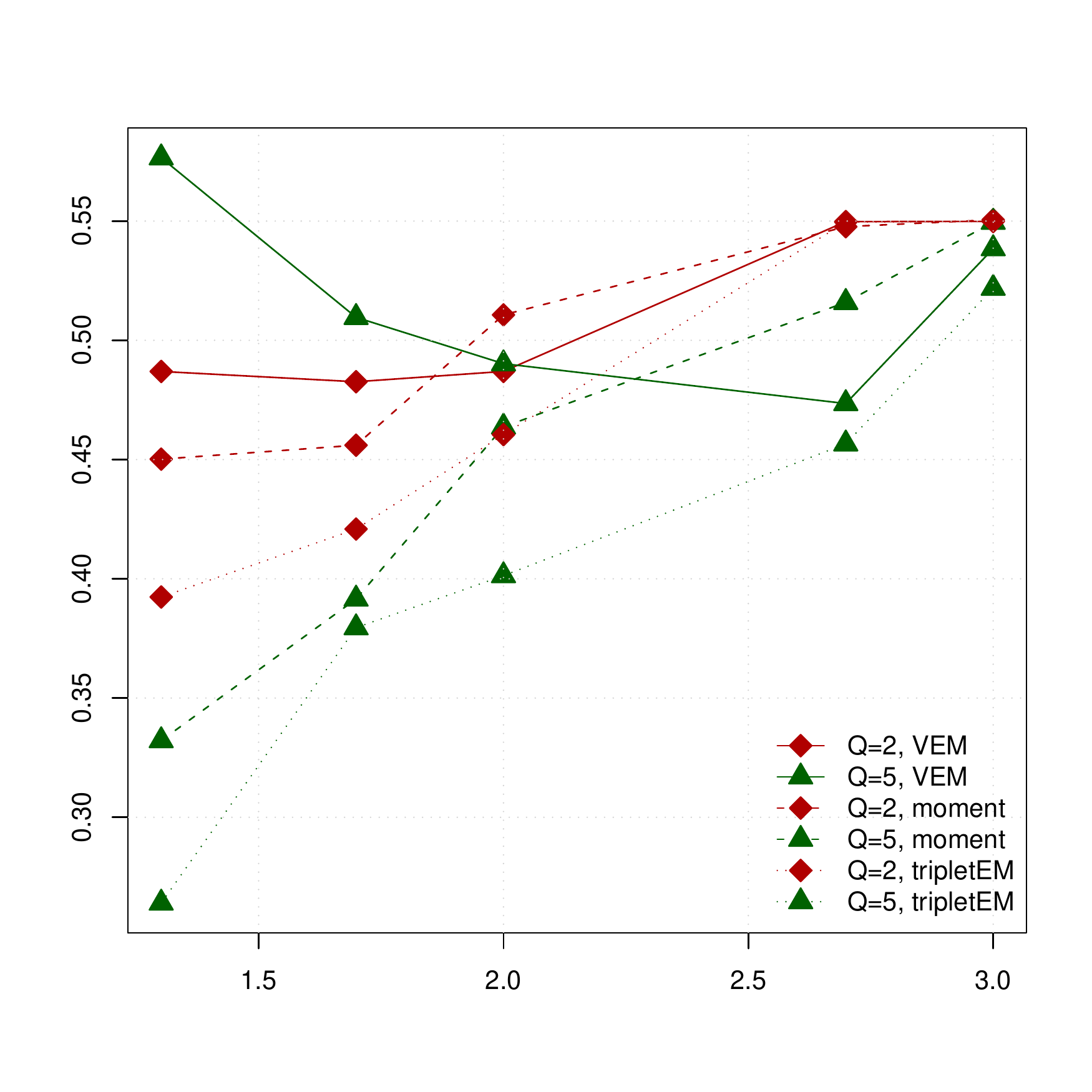}} \\
\hspace{-.5cm}\rotatebox{90}{\makebox[-0.1cm]{$\frac 1 2 \log({\text{EmpVar}(\hat{ \alpha})})$ }} 
&    \raisebox{-2.1cm}{\includegraphics[width=0.3\textwidth]{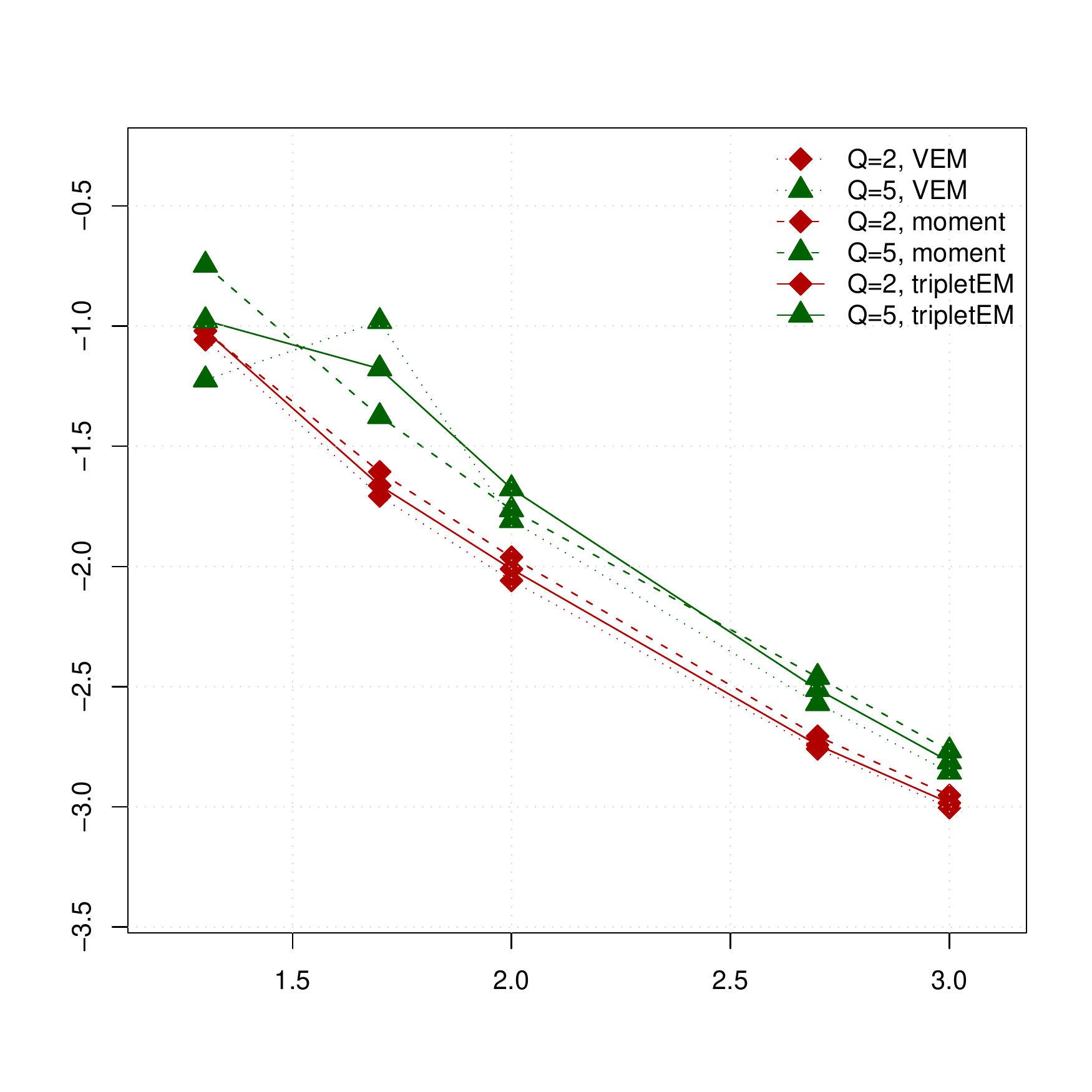}} 
&    \raisebox{-2.1cm}{\includegraphics[width=0.3\textwidth]{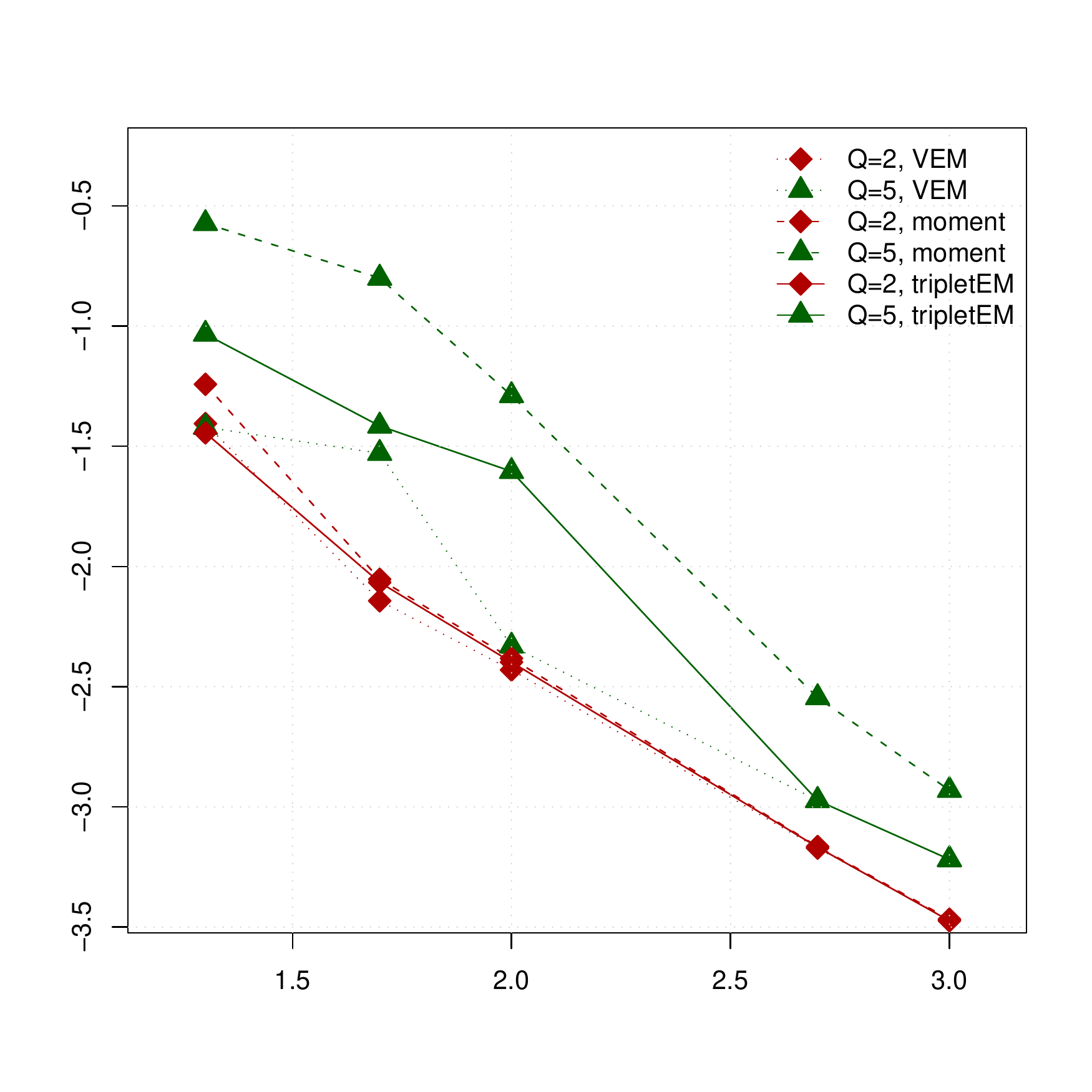}} 
&    \raisebox{-2.1cm}{\includegraphics[width=0.3\textwidth]{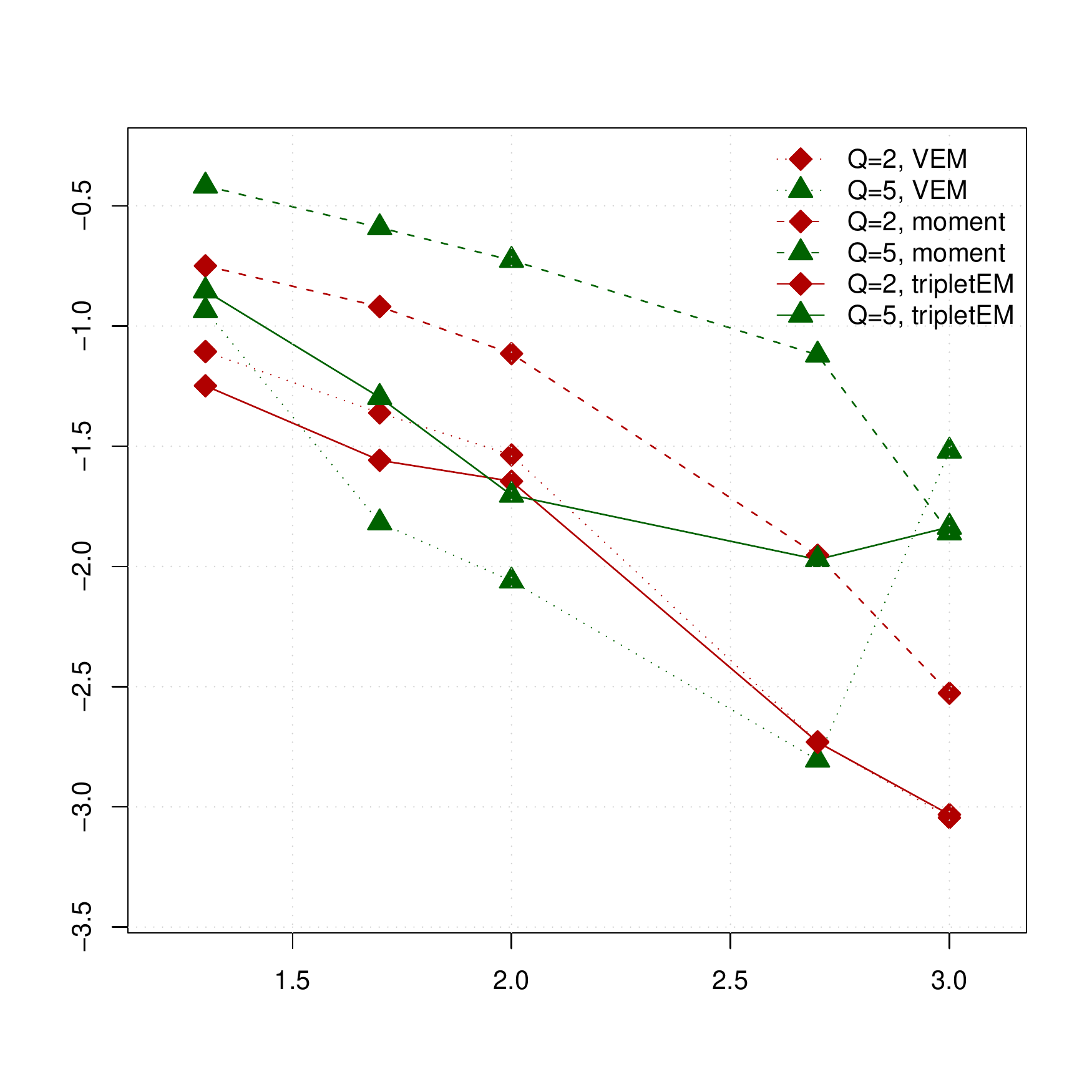}} \\
&  & $\log_{10} n$ & \\
 \end{tabular}
 \caption{$\hat{\alpha}$ (top) and the $\log$ of its empirical standard deviation (bottom) averaged over 100 graph simulations, as functions of the number of graph vertices in log-scale for equal group proportions.}
    \label{fig:mean_varplot}
\end{figure}

In order to compare the dispersion of all estimators, we consider
their empirical standard deviation computed over 100
simulations. Figure~\ref{fig:mean_varplot} (bottom) shows the evolution of the
$\log$ of the empirical standard deviation of $\hat{
  \alpha}$ when the size of the graphs grows from 20 vertices up to
1000 vertices.  We see a linear dependence between the $\log$ of the
graph size and the $\log$ standard deviation. The slope of the lines is about
$-1$ which  indicates that the standard deviation decreases with
rate of the order $1/n$ (where $n$ is the number of vertices of the
graph). The differences between the intercepts relate to constant
factors driving the relations between all rates of convergence. When
$Q=2$, we observe very similar intercepts for all methods, both for models $1$ and $2$. When
$Q=5$, \textsc{vem} appears to  converge faster but the order of the
standard deviations remain comparable among all estimation methods.
For model $3$, the moment based estimations have greater dispersion,
but still decrease with the same rate.

\begin{figure}[htbp!]
\centering
\includegraphics[width=0.9\textwidth,height=0.5\textwidth]{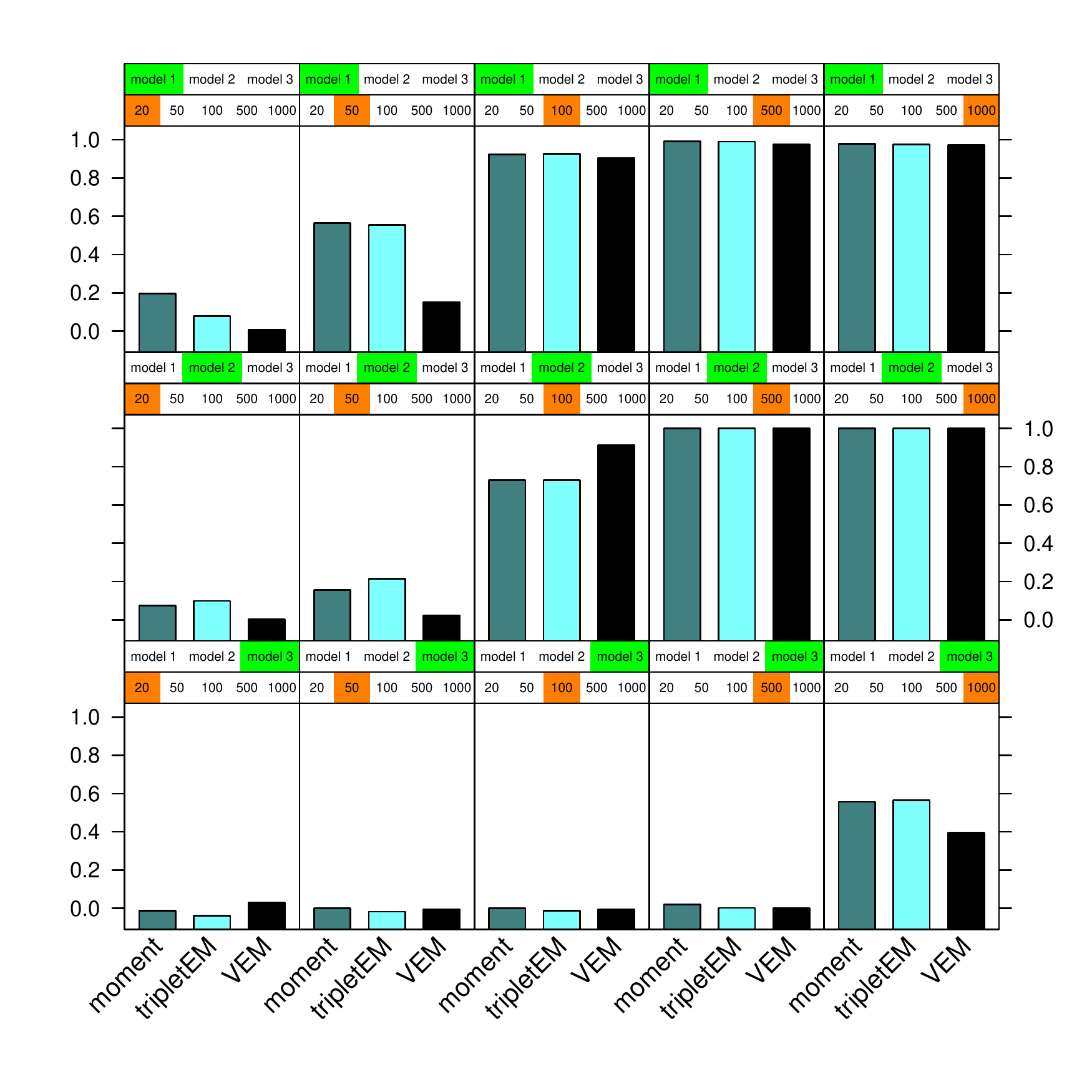} 
\caption{Evaluation of Rand Index to compare true and estimated latent structure of binary affiliation graph.}
\label{fig:randbernoulli}
\end{figure}

We use the adjusted Rand Index \citep{HA1985} to evaluate the
agreement between the estimated and the true latent structure.  The
computation of the Rand Index is based on a ratio between the number
of node pairs belonging to the same and to different classes when
considering the actual latent structure versus the estimated
one. This index lies between 0 and 1, two identical latent structures
having an adjusted Rand Index equal to 1. Figure~\ref{fig:randbernoulli} displays the Rand Index for the three models
and five different graph sizes. It appears that the three algorithms
allow a reasonable recovery of the latent structure, for models 1
and 2, when the considered graphs have more than 100 vertices.  As
expected, the larger the number of nodes, the better the recovery of
the latent structure we observe.
We also notice that our proposed strategy for recovering the latent
structure performs as well or better than the variational approach in
all cases.

The previous experiments show that the two estimation procedures
proposed in this work behave as well or better than the variational
based algorithm, both  for the parameter estimation and the recovery of
the latent structure. Notice also that the moment based method does not
depend on any sort of initialization, since it relies on the analytical
resolution of a simple system based on triads (order 3 structures).

\subsection{Weighted affiliation model}

\paragraph{Simulations set-up.}  In  the following experiments, we use
a sparsity parameter  constant across the graph and  non missing edges
are distributed according to a  Gaussian model as described in Section
\ref{sec:weighted},   with  different   means   $\mu_{\text{in}}$  and
$\mu_{\text{out}}$ and  equal variance $\sigma^2$. The  intricacy of a
model  is inversely related to  the   'distance'    between   the   parameters
$\theta_{\text{in}}$ and $\theta_{\text{out}}$.  We use the  Mahalanobis distance   $\Delta=|(   \mu_{\text{in}}   -   \mu_{\text{out}})/\sigma  |$.  Three  models are  considered with different  levels of   intricacy: we fix $\mu_{\text{in}} =2$ and $\mu_{\text{out}}=1$, thus  $\Delta=|  (\mu_{\text{in}}  -  \mu_{\text{out}})/\sigma  |=1/\sigma$
 which  takes  the values  $\Delta  =10$  (model  A), $\Delta  =2$
 (model B) and $\Delta =1$ (model C). 
We fix the number of groups $Q=2$, equal group proportions and consider different  number of vertices $n \in \{20,100,500,1000\}$. 

% \begin{table}
% \caption{Three types of simulations with increasing level of intricacy  for weighted graphs. We fixed $\mu_{\text{in}} =2$ and $\mu_{\text{out}}=1$, thus 
%  $\Delta=| (\mu_{\text{in}} - \mu_{\text{out}})/\sigma |=1/\sigma$ is the Mahalanobis distance between the two models.}
%  \centering
%  \begin{tabular}{ r| c c c c}
% \cline{1-4}&&&&\\
%  Situation & A & B & C & \hspace{3cm}\\ 
% \cline{1-4}&&&&\\ 
% $\sigma$ & $1/10$ & $1/2$&  $1$ &\hspace{3cm}\\
% $\Delta$& $10$ & $2$ & $1$ &\hspace{3cm}
% \end{tabular}
% \label{tab:situations} 
% \end{table}

\begin{figure}[htbp!]
\centering
 \begin{tabular}{c c c}
%\rotatebox{90}{\makebox[0.2cm]{$\hat{\mu}_{in}$ }}&
\raisebox{-2.1cm}{ \includegraphics[width=0.3\textwidth,height=0.3\textwidth]{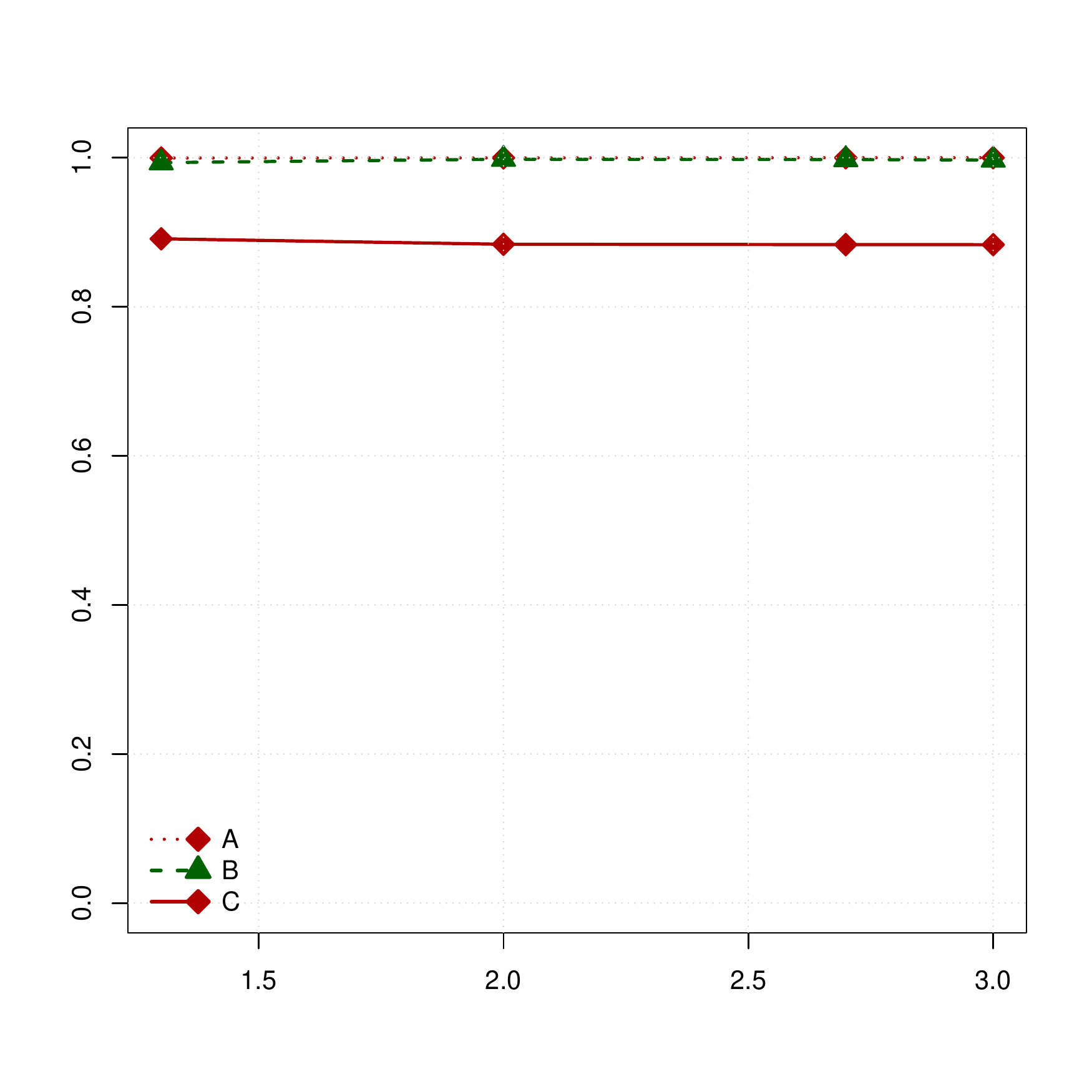}} &
%\rotatebox{90}{\makebox[0.2cm]{$\sqrt{\text{EmpVar}(\hat{\mu}_{in})}$ }}&
\raisebox{-2.1cm}{ \includegraphics[width=0.3\textwidth,height=0.3\textwidth]{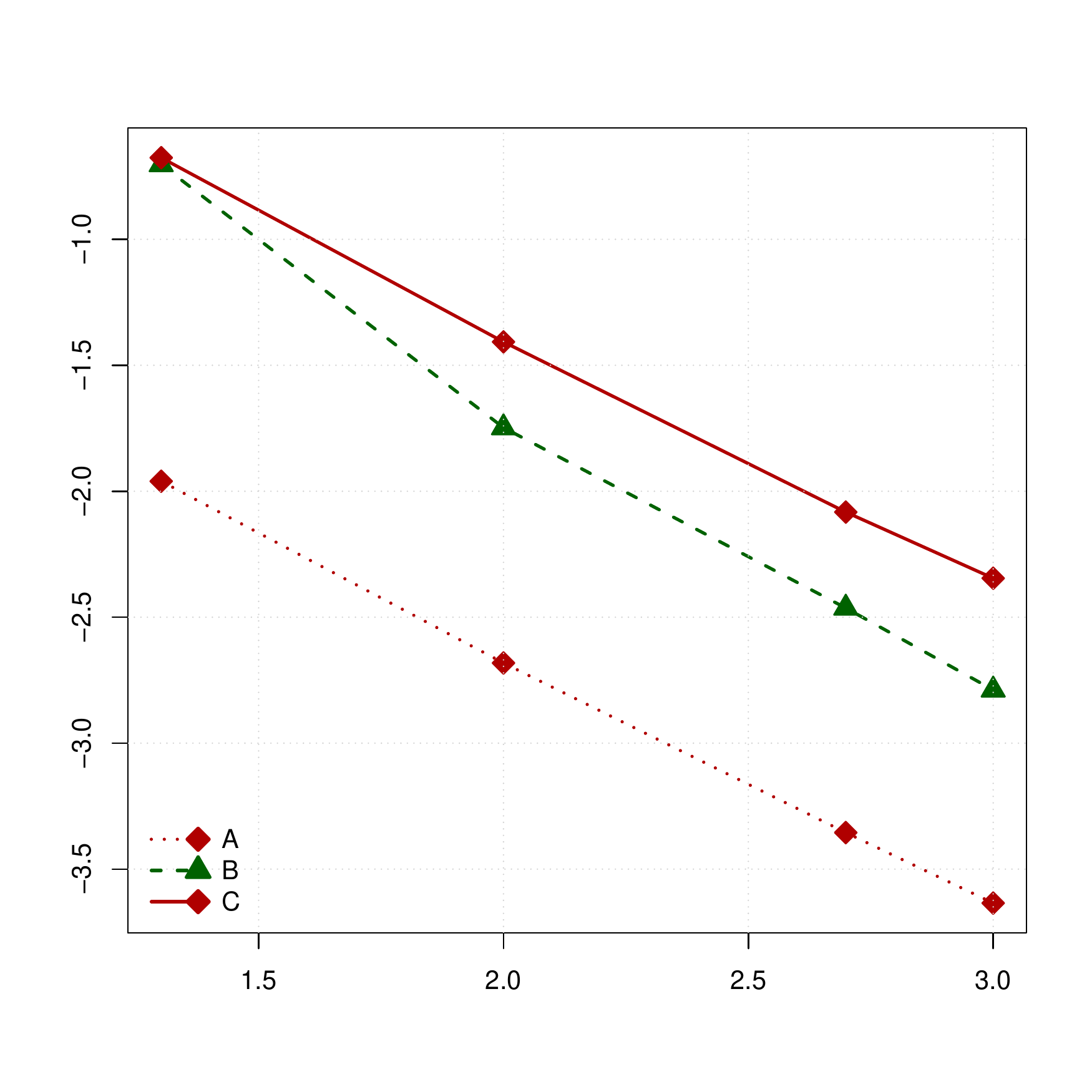}} &
\raisebox{-2.1cm}{\includegraphics[width=0.3\textwidth,height=0.3\textwidth]{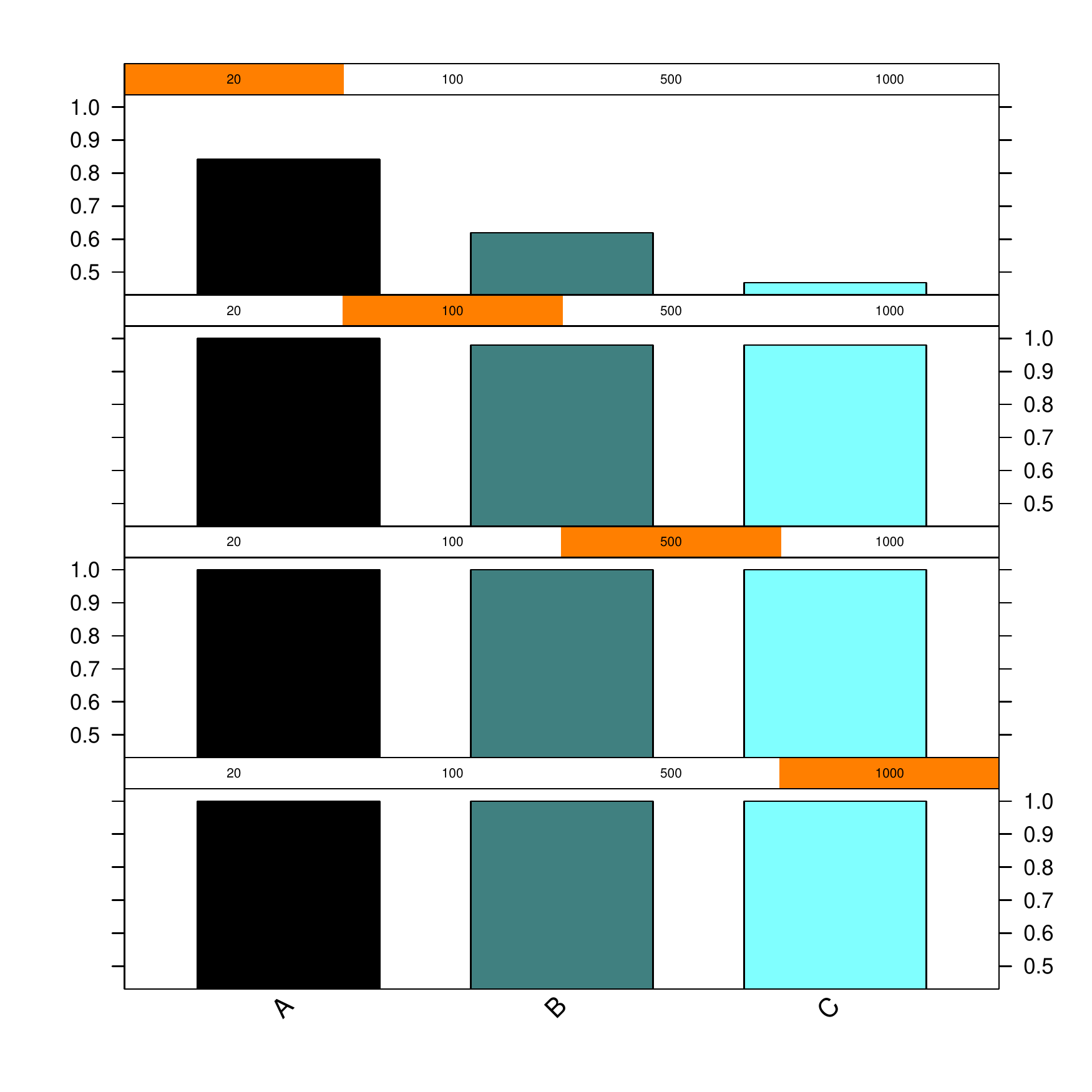}}\\ 
$\hat \mu_{\text in}$ vs $\log_{10} n$&$\frac 1 2 \log(\text{EmpVar}(\hat{\mu}_{in}))$ vs $\log_{10} n$ &\\
 (a)  &  (b) & (c)  
\end{tabular}
\caption{Evaluation of the estimation of the parameters of a weighted graph.
(a) $\hat{\mu}_{\text{in}}$ and (b) $\log$ of its empirical standard
deviation, both as  functions of the number of graph vertices, expressed in log-scale. Each estimation is averaged over 100 graph simulations.
(c) Rand Index computation comparing  the true latent structure to the
estimated one for the three models (columns A, B, C) and four graph sizes (rows $n=20, 100, 500, 1000$).}
\label{fig:weighted}
\end{figure}

We computed bias and empirical standard deviations over 100 simulations.
As illustrated by Figure~\ref{fig:weighted}(a) in the case of $\hat \mu_{\text{in}}$, the method recovers the parameters with no bias, except for model C where a small bias occurs due to the high level of intricacy of the model. 
Figure~\ref{fig:weighted}(b) displays the evolution of the $\log$ of the
empirical standard deviation of $\hat{\mu}_{\text{in}}$ when the size of the
graphs grows from 20 vertices up to 1000 vertices.  As for the binary
affiliation model estimators, we observe a linear dependence between the $\log$ of the graph size and the $\log$ standard deviation, the slope of the lines lying
in  $[-1/2,-1]$.

Figure~\ref{fig:weighted}(b) displays  the Rand  Index for  the three
different models (A,B,C) and four different graph sizes. When graphs have more than 100 nodes, recovery of the hidden structure is almost perfect in all situations as previously observed in the binary case. 

The previous experiments show  that when dealing with weigthed affiliation graphs, the estimation of   the parameters and of  the graph latent structure can be efficiently  achieved considering only edges (order 2 structures).

\subsection{Cross-citations of economics journals}

Let us  illustrate the difference  between weighted and  binary models
for  graph   clustering  using  a  real  data   example.  We  consider
cross-citations  of 42  economics  journals over  the years  1995-1997
\citep{PB2002}.  The raw data corresponds to a weighted non symmetric graph where vertices are
journals and directed edges the number of citations from one  journal to another one.  We first  take the mean value of citations  between each pair of
journals (leading to a symmetric adjacency matrix) and work with its 
normalized Laplacian.  
 Figure~\ref{fig:crosscitations}    displays    the
affiliation  matrices  structured according  to  a  partition in  four
classes.  Clustering based  on the  binary model  and on  the weighted
model      (respectively     left      and     right      sides     of
Figure~\ref{fig:crosscitations})   exhibit   very  different   cluster
structures.  The   binary  model  finds  classes  which   tend  to  be
homogeneous  in   terms  of  probability  of   intra-group  and  inter-group
connections,  while  the  weighted   model  finds  classes  which  are
homogeneous in terms of intra-group and inter-group connection weights. This distinction results in completely different interpretations.

\begin{figure}[htbp!]
\centering
\begin{tabular}{c c}
\includegraphics[width=0.4\textwidth,height=0.4\textwidth]{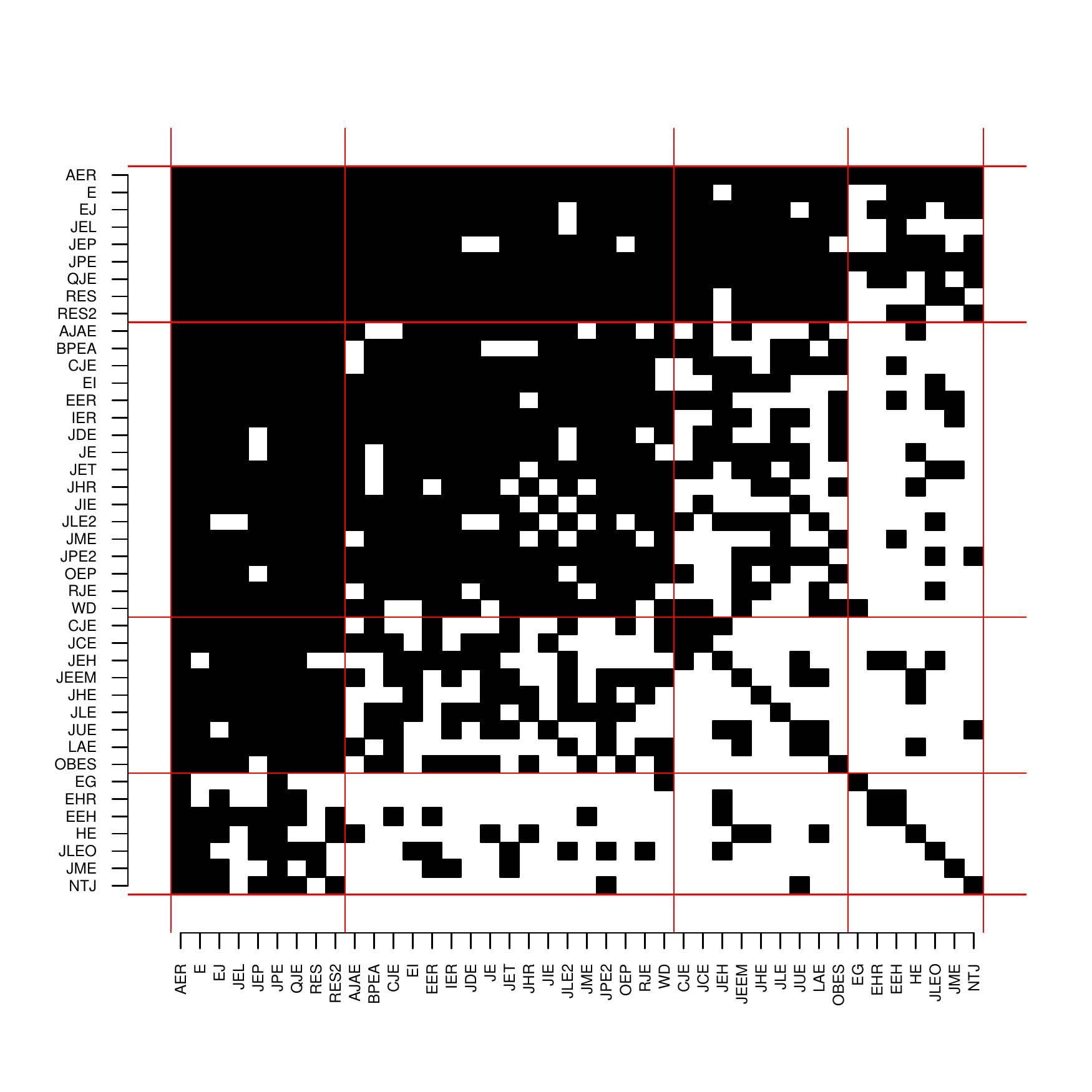}  &
\includegraphics[width=0.4\textwidth,height=0.4\textwidth]{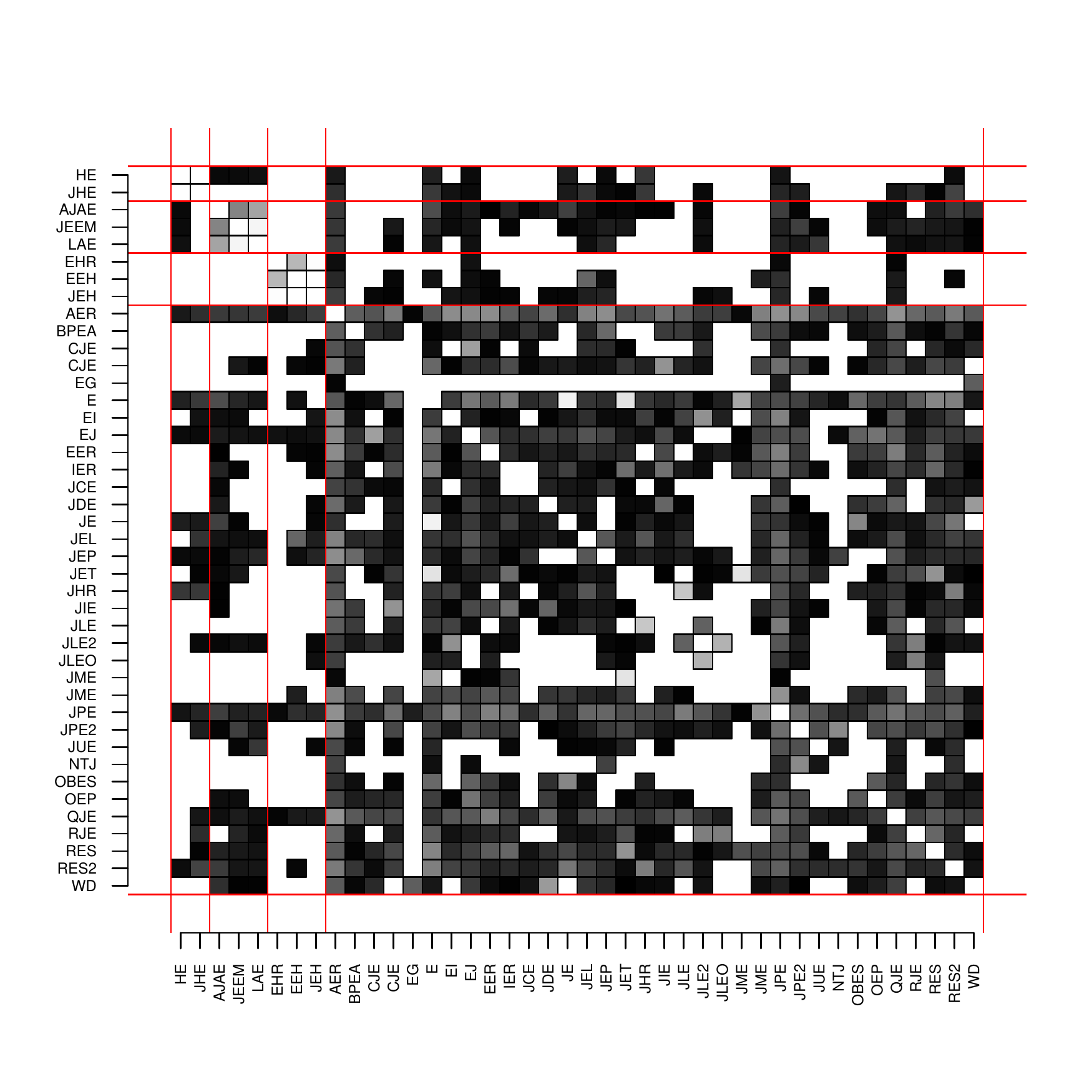}\\
(a) & (b) 
\end{tabular}
\caption{Matrices  of cross-citations  between  42 economics  journals
  with rows and columns reorganized according to groups found by the binary random graph
  mixture model (a) compared to  groups found with the weighted random
  graph mixture model (b).}
\label{fig:crosscitations}
\end{figure}

The binary model finds two groups of nodes which are strongly connected within
their  groups but  also  with nodes  from  the other  groups. It  also
exhibit  two   other  smaller    classes  with  low
intra-group  connectivity and  nodes that  preferentially link  to the
first  class which plays  the role  of a  reference class.  Indeed the
first  class (top  left)  found by  the  binary model  is composed  by
journals  with high  impact factors:  American Economic  Review (AER),
Econometrica (E),  Journal of Economic Literature (JEL), Journal of
Economic  Perspectives  (JEP),  Journal  of Political  Economy  (JPE),
Quarterly Journal of Economics (QJE), Review of Economic Studies (RES)
and Review of Economics and Statistics (RES2).

The result produced by the weighted model shows a main  class of strongly interconnected journals and  three smaller classes of journals, which  weakly cross-cite each other:
\begin{description}
\item[Class 1] (health)   Health Economics (HE), Journal of Health Economics (JHE);
\item[Class 2] (natural resources)  Journal of Agricultural Economics (AJAE), Land Economics (LAE),
  Journal of Environmental Economics and Management (JEEM);
\item[Class 3] (economic history) Exploration of Economic History (EEH), Journal of Economic History
  (JEH), Economic History Review (EHR).
\end{description}
Each  of these  three classes  is  composed by  journals dedicated  to
similar topics  (respectively, health, natural  resources and economic
history). They preferentially cite journals from the first class which
contains journals with less specific topics.

\section{Proofs}\label{sec:proofs}
\begin{proof}[Proof of Theorem~\ref{thm:main}]  In order to facilitate
  the reading of the proof, we decompose it into several stages. 

\paragraph*{Preliminaries.}
We fix $k,s\ge 1$ and $p = \binom k 2$.
Let us recall that   $\mathcal{V}_Q$ is the set  of $Q$-size vectors  such that for
any $v=(v_1,\ldots, v_q) \in  \mathcal{V}_Q$, we have $v_i\in \{0,1\}$
and $\sum_{i=1}^Q v_i=1$. We also let $\Q=\{1,\ldots, Q\}$. 
We then consider the set 
\begin{equation*}
\mathcal{Z}=\Big\{z\in     \mathcal{V}_Q^{\mathbb{N}}    ;    \forall
\uq=(q_1,\ldots, q_k)\in \Q^k, \; \frac{ (n-k) !  }{n!}n_{\uq} :=\frac {(n-k)!}
{n!} \sum_{\ui \in \Ik} \prod_{l=1}^k z_{i_lq_l}
\mathop{\to}_{n \to \infty} \prod_{l=1}^k \pi_{q_l} \Big\}.  
\end{equation*}
Moreover, we let 
$N_{\uq} =\sum_{\ui\in \Ik} \prod_{l=1}^k Z_{i_lq_l}$. 
The strong law of large  numbers gives the almost sure convergence, as
$n$   tends   to   infinity,   of   $ [(n-k)   !/   n! ]  N_{\uq}  $   to
$\prod_{l=1}^k\pi_{q_l}$.  
This implies that  $ \mathbb{P}(\{Z_n\}_{n\ge 1}\in \mathcal{Z}) =1$. 

\paragraph*{Consistency of $\hat m_{g}$. }
%Let us now focus on the consistency of $\hat m_{g}$. 
We first introduce the conditional mean of $g(\bX^{\ui})$ given that the hidden groups at position $\ui$ are given by $\uq$
\begin{equation*}
   m_g(\uq)=\mathbb{E}\Big(g(\bX^{\ui}) \Big| \prod_{l=1}^k Z_{i_lq_l}=1\Big) .
\end{equation*}
Using the equalities 
\begin{equation}
\forall \ui  \in \Ik,  \sum_{\uq \in \Q^q}  \prod_{l=1}^k Z_{i_lq_l}=1
\quad \text{ and }  \quad m_{g}=\sum_{\uq \in \Q^k} \Big (\prod_{l=1}^k
\pi_{q_l} \Big) m_g(\uq), 
\label{eq:cles}
\end{equation}
  we may write  the decomposition 
\begin{multline}
 \hat m_{g} -m_{g}
=   \frac{(n-k)!}{n!}   \sum_{\uq    \in   \Q^k}\sum_{\ui   \in   \Ik}
\prod_{l=1}^k   Z_{i_lq_l}  g(  \bX^{\ui})   -\sum_{\uq  \in   \Q^k  }
\Big(\prod_{l=1}^k \pi_{q_l} \Big) m_g(\uq) \\
=  \sum_{\uq  \in \Q^k}  \Big[  \frac{(n-k)!}{n!}  \sum_{\ui \in  \Ik}
\Big (\prod_{l=1}^k Z_{i_l q_l} \Big)(g(\bX^{\ui})-m_g(\uq)) 
+m_g(\uq) \Big( \frac{(n-k)!  }{n!} N_{\uq} -\prod_{l=1}^k \pi_{q_l}\Big) 
\Big] . \label{eq:decompose}
\end{multline}

In order  to establish the consistency  of $\hat m_{g}$, we  rely on a
conditioning argument.  Let $A$ be  the event $  \limsup_{n\to \infty}
|\hat m_{g}-m_{g}|=0 $. We then have 
\begin{equation}
  \mathbb{P}(A ) = \mathbb{E}[\mathbb{E} (1_A| \{Z_n\}_{n\ge 1})] . 
\label{eq:dominated}
\end{equation}
Now,  conditional  on  $\{Z_n\}_{n\ge  1}=z $,  the  random  variables
$\{\bX^{\ui};  \ui \in  \Ik, \prod_{l=1}^k  z_{i_lq_l} =1  \}$  form a
$n_{\uq}$-sample of independent and identically distributed random variables. Letting $B$ be the event 
\begin{equation*}
  \limsup_{n\to\infty} \frac  1 {N_{\uq}}\Big| \sum_{\ui \in
      \Ik ; \prod_{l=1}^k Z_{i_lq_l} =1}  (g(\bX^{\ui})-m_g(\uq)) \Big| = 0,
\end{equation*}
 the strong law of large numbers yields that for any $z\in \mathcal{Z}$,
\begin{equation*}
  \mathbb{E}(1_B |\{Z_n\}_{n\ge 1} =z) =1 .
\end{equation*}

Conditional on  $\{Z_n\}_{n\ge 1}=z \in \mathcal{Z}$, we may thus re-write the decomposition \eqref{eq:decompose} as  
\begin{multline*}
  \hat  m_{g}  -m_{g}=  \sum_{\uq  \in  \Q^k }  \Big[  \frac{
    (n-k)!}{n! }n_{\uq  } \times\frac  1 {n_{\uq}}\sum_{\ui \in \Ik ;
      \prod_{l=1}^k z_{i_lq_l} =1} (g(\bX^{\ui})-m_g(\uq)) \\
+m_g(\uq) \Big( \frac{ (n-k) !} {n!}n_{\uq} -\prod_{l=1} ^k \pi_{q_l}\Big) 
\Big] ,
\end{multline*}
which establishes that for any $z \in \mathcal{Z}$, we have 
$  \mathbb{E}(1_A | \{Z_n\}_{n\ge 1}=z) =1$. 
Coming back to \eqref{eq:dominated}, we thus obtain 
\begin{equation*}
  \mathbb{P}(\lim_{n\to \infty}\hat m_{g} =m_{g}) =1.
\end{equation*}

\paragraph*{Asymptotic normality of $\hat m_g$.} Let us now prove a central limit result for $\sqrt{n}(\hat m_{g}-m_{g})$. 
First, the central limit theorem applied to the $Q$-size vector $\sum_{i=1}^n (Z_i-\boldsymbol{\pi})/\sqrt{n}$ gives the following convergence
\begin{equation}
  \frac 1 {\sqrt{n}} \sum_{i=1}^n (Z_i-\boldsymbol{\pi}) \leadsto \mathcal{N}(0,\Sigma), \text{ as } n \to \infty, \label{eq:cv_vector}
\end{equation}
where $\Sigma_{qq}=\pi_q(1-\pi_q)$ and $\Sigma_{q\ell}=-\pi_q\pi_\ell$
when $q\neq \ell$.

Now, let us consider the second  term appearing in the right hand side
of  \eqref{eq:decompose}. To  establish  a central  limit theorem  for
$N_{\uq}$, we decompose the sum of products 
$$
\sum_{\ui \in  \I_k} \prod_{l=1}^k Z_{i_l  q_l} = \sum_{\ui  \in \I_k}
\prod_{l=1}^k ( Z_{i_l q_l} -\pi_{q_l} +\pi_{q_l})
$$
into   sums   of   products   of   centered  terms   $  ( Z_{i_l   q_l}
-\pi_{q_l})$. This leads to 
\begin{equation*}
  \frac{ (n-k) ! }{n !}N_{\uq} -\prod_{l=1}^k \pi_{q_l} = \sum_{u=1}^k 
\sum_{L\subset     \{1,\ldots,k\};    |L|=u}     \frac    {(n-u)!}{n!}
\Big(\prod_{l\notin L} \pi_{q_l} \Big) \sum_{ \ui \in \I_L} \prod_{l
  \in L} (Z_{i_lq_l} -\pi_{q_l}) ,
\end{equation*}
where $|L|$ denotes the cardinality  of the set $L$ and $\I_L$ denotes
the set  of injective  maps from $L$  to $\I=\{1,\ldots,n\}$.  In this
expression,  the  leading  term  (obtained  for  singleton  sets  $L$,
\emph{i.e.} when $u=1$)  gives the rate of convergence  in the central
limit theorem. In other words,
\begin{multline}
 \sqrt{n} \left(\frac{  (n-k) ! }{n  !} N_{\uq}-\prod_{l=1}^k \pi_{q_l}
 \right)= \sum_{l=1}^k \Big(\prod_{j\neq l} \pi_{q_j} \Big ) \frac 1 {\sqrt n }\sum_{i=1}^n (Z_{iq_l}-\pi_{q_l})\\+\sum_{u=2}^k 
\sum_{L\subset \{1,\ldots,k\}; |L|=u} \frac {\sqrt n (n-u)!}{n!}
\Big(\prod_{l\notin L} \pi_{q_l} \Big) \sum_{ \ui \in \I_L} \prod_{l
  \in L} (Z_{i_lq_l} -\pi_{q_l} ) . \label{eq:terme_domi}
\end{multline}
The  first  term  in  the  right hand  side  of  \eqref{eq:terme_domi}
converges   to  a  linear   combination  of   the  coordinates   of  a
$\mathcal{N}(0,\Sigma)$ vector,  whereas the second  term converges to
zero. Indeed, for  any value $u \ge 2$ and any  set $L$ of cardinality
$u$, we may write 
$$
\frac {\sqrt n (n-u)!}{n!} \sum_{ \ui \in \I_L} \prod_{l
  \in  L} (Z_{i_lq_l}  -\pi_{q_l} )  =  \frac 1  {\sqrt n  (n-1)\cdots
  (n-u+1)} \sum_{ \ui \in \I_L} \prod_{l
  \in L} (Z_{i_lq_l} -\pi_{q_l} )
$$
which converges to zero. 
Thus we get, 
\begin{equation*}
\sqrt{n}   \sum_{\uq   \in   \Q^k}   m_g(\uq)  \Big[   \frac{(n-k)   !
  }{n! }N_{\uq} -\prod_{k=1}  ^l \pi_{q_l} \Big] =  \sum_{\uq \in \Q^k}
m_g(\uq) \Big[ \sum_{l=1}^k 
\frac{\prod_{j\neq  l  }\pi_{q_j}  }{\sqrt{n}} \sum_{i=1}^n  (Z_{iq_l}
-\pi_{q_l}l) + R_{n,\uq} \Big] , 
\end{equation*}
where   $R_{n,\uq}=o_P(1)$   are   negligible  terms   converging   in
probability to zero, as $n$ tends to infinity. According to \eqref{eq:cv_vector}, we obtain that 
\begin{equation*}
  \sqrt{n}  \sum_{\uq   \in  \Q^k}  m_g(\uq)  \Big[   \frac{  (n-k)  !
    }{n!}N_{\uq}-\prod_{l=1} ^k \pi_{q_l}\Big] \mathop{\leadsto}_{n\to
    \infty}\sum_{\uq\in \Q^k } m_g(\uq) \big[\sum_{l=1}^k \big(\prod_{j\neq
    l }\pi_{q_j}\big)W_{q_l} \big] , 
\end{equation*}
where $W=(W_1,\ldots,W_Q)\sim\mathcal{N}(0,\Sigma)$.

To obtain a central limit theorem for $\hat m_g$ it now suffices to prove that the first term in the right hand side of \eqref{eq:decompose} is negligible, when scaled by the rate of convergence $\sqrt{n}$.  
Indeed, we may write this term as 
\begin{equation*}
\tilde R_n  = \sum_{\uq \in  \Q^k } \left(\frac{(n-k)!}{(n-1)!}\right)
^{1/2} \times 
\left(\frac{(n-k)   ! }{n   !  } N_{\uq}\right)^{1/2}   \times\frac  1
{\sqrt{N_{\uq}}}\sum_{\ui \in \Ik ; \prod_{l} Z_{i_lq_l} =1}  (g(\bX^{\ui})-m_g(\uq)), 
\end{equation*} 
which satisfies, for any $k\ge 2$,  any $\epsilon >0$ and any $z\in \mathcal{Z}$, 
\begin{equation*}
  \mathbb{P}(|\tilde R_n|\ge \epsilon| \{Z_n\}_{n\ge 1}=z) \to_{n\to \infty} 0.
\end{equation*}
Using dominated convergence, we also have $\mathbb{P}(|\tilde R_n|\ge \epsilon) \to_{n\to \infty} 0$, for 
 any $\epsilon >0$. Now, going back to \eqref{eq:decompose}, we finally obtain 
\begin{equation*}
  \sqrt{n}(\hat m_{g}-m_{g}) 
\mathop{\leadsto}_{n\to \infty}
\sum_{\uq\in \Q^k } m_g(\uq) \big[\sum_{l=1}^k \big(\prod_{j\neq
    l }\pi_{q_j}\big)W_{q_l} \big] \sim \mathcal{N}(0,\Sigma_g).
\end{equation*}

%%%%%
%%%%%%
\paragraph*{Expression for the limiting variance $\Sigma_g$.}
The computation of the variance $\Sigma_g$ could be done using the above expression, but this leads to tedious formulas. A simpler expression of the limiting variance is obtained in the following way. We  prove that $\sqrt{n}U_n:=\sqrt{n}(\hat m_{g}-m_{g})$ has a bounded third order moment. This is sufficient to claim that  $\Sigma_g$ can be obtained as the limiting variance of $\sqrt{n}U_n$.  

First, since non adjacent edges form independent variates, 
it is easy to see that we have
\begin{equation*}
  \mathbb{E}(\|\sqrt{n}U_n\|^3)\le  \left(\frac {(n-k)!}{\sqrt{n}(n-1)!}\right)^3 
\sum_{\ui,\underline{\mathfrak j},\underline{\mathfrak k} ; \ui \cap \underline{\mathfrak j} \cap \underline{\mathfrak k } \neq \emptyset} \esp( \|g(\bX^{\ui})-m_g\|\|g(\bX^{\underline{\mathfrak j}})-m_g\|\|g(\bX^{\underline{\mathfrak k}})-m_g\|) ,
\end{equation*}
where $\ui \cap \underline {\mathfrak j}$ stands for the intersection of $\ui$ and $\underline {\mathfrak j}$ viewed as index sets (instead of $k$-tuples). 
The above sum contains at most $k^2n[(n-1)\cdots(n-k+1)]^3$ terms, which are bounded (there are finitely many of them). Thus this quantity converges to zero as $n$ tends to infinity. Moreover,
\begin{equation*}
  \Var(\sqrt{n}U_n)= \left(\frac {(n-k)!}{\sqrt{n}(n-1)!}\right)^2 
\sum_{\ui,\underline{\mathfrak j} ; \ui \cap \underline{\mathfrak j} \neq \emptyset}
 \Cov(g(\bX^{\ui}),g(\bX^{\underline {\mathfrak j}}))  .
\end{equation*}
The above sum may be decomposed according to the cardinality of the set $\ui \cap \underline{\mathfrak j} $. It is then easy to see that the dominating term is obtained when $|\ui \cap \underline{\mathfrak j}|=1 $, while the other terms converge to zero. Namely 
\begin{equation*}
  \Var(\sqrt{n}U_n)= \left(\frac {(n-k)!}{\sqrt{n}(n-1)!}\right)^2 
\sum_{\ui,\underline{\mathfrak j} ;|\ui \cap \underline{\mathfrak j}|=1 }
 \Cov(g(\bX^{\ui}),g(\bX^{\underline {\mathfrak j}}))  +o(1) .
\end{equation*}
To describe all the possible configurations where  $|\ui \cap \underline{\mathfrak j}|=1 $, we may fix the first index $\ui$ to $(1,\ldots,k)$ and let the second index $\underline{\mathfrak j}$ describe the set of indexes where some position $s$ takes one of the values $\{1,\ldots,k\}$ (corresponding to the intersection $\ui \cap \underline{\mathfrak j}$) and at any other position, there is some value in $\{k+1,\ldots,n\}$. For any $s,t\in \{1,\ldots,k\}$, we thus let  $\underline{e}_s^t\in \I_k$ satisfying $\underline{e}_s^t(s)=t$ and $\underline{e}_s^t(j)\in \{k+1,\ldots,n\}$ for any $j\neq s$.
With this notation,  we obtain 
\begin{equation*}
 \Sigma_g= \lim_{n\to\infty} \Var(\sqrt{n}U_n)= \sum_{s=1}^k\sum_{t=1}^k \Cov(g(\bX^{(1,\ldots,k)}),g(\bX^{\underline{e}_s^t})).
\end{equation*}
Note that in the  case of an affiliation structure with equal group proportions, we could prove from this expression that $\Sigma_g=0$ (using for instance the results of Lemma~\ref{lem:degeneracy} presented below). Anyway this will be a consequence of the following developments.

%%%%%%%%%%
\paragraph*{The degenerate case.}
Let us now finish this proof by considering the specific case where we have an affiliation structure \eqref{eq:model_affil} and equal group proportions \eqref{eq:model_equalgroup}. Coming back to \eqref{eq:decompose}, we write $ \hat m_g -m_g= T_1+T_2$ where 
\begin{eqnarray*}
 T_1&=&   \sum_{\uq  \in \Q^k}    \frac{(n-k)!}{n!}  \sum_{\ui \in  \Ik}
\Big (\prod_{l=1}^k Z_{i_l q_l} \Big)(g(\bX^{\ui})-m_g(\uq)), \\
 T_2&=&    \sum_{\uq  \in \Q^k} 
m_g(\uq) \Big( \frac{(n-k)!  }{n!} N_{\uq} -\prod_{l=1}^k \pi_{q_l}\Big) .
\end{eqnarray*}
We first deal with the second term $T_2$. According to \eqref{eq:terme_domi}, we have
\begin{multline*}
T_2=   \sum_{\uq  \in \Q^k} m_g(\uq) 
 \sum_{l=1}^k  \frac 1 {Q^{k-1} }\frac 1 n \sum_{i=1}^n (Z_{iq_l}-\pi_{q_l})\\+\sum_{u=2}^k \sum_{L\subset \{1,\ldots,k\}; |L|=u} \frac { (n-u)!}{n!}
\frac 1 {Q^{k-u}} \sum_{ \ui \in \I_L} \prod_{l
  \in L} (Z_{i_lq_l} -\pi_{q_l} ):=T_{2,1}+T_{2,2}. 
\end{multline*}
We now prove that the first term in the right hand side of this equality, namely $T_{2,1}$ is zero. This result relies on the following lemma, stating that the model is invariant under a permutation of the values of the node groups.

\begin{lemma}\label{lem:degeneracy}
  Under the assumptions and notations of Theorem~\ref{thm:main}, assuming moreover \eqref{eq:model_affil} and \eqref{eq:model_equalgroup}, for any $\sigma\in \mathcal{S}_{\Q}$ the set of permutations of $\Q$, we have   
  \begin{equation*}
 (\{Z_i\}_{1\le i \le n}, \{X_{ij}\}_{1\le i<j\le n} ) \mathop{=}^{d}  (\{\sigma(Z_i)\}_{1\le i \le n},\{X_{ij}\}_{1\le i<j\le n} ), 
  \end{equation*}
where $ \mathop{=}^{d} $ means equality in distribution. 
As a consequence, for any value $\uq \in \Q^k$, the conditional expectation $m_g(\uq)$ is constant along the orbit (induced by $\mathcal{S}_{\Q}$) of the point $\uq$, \emph{i.e.} the set  of values $\{m_g(\sigma(\uq)) ; \sigma \in \mathcal{S}_{\Q}\}$ is a singleton for any fixed $\uq\in\Q^k$.
\end{lemma}

Indeed, according to \eqref{eq:model_gen},  \eqref{eq:model_affil} and \eqref{eq:model_equalgroup}, and using that any permutation $\sigma$ is a one-to-one application, we have
\begin{multline*}
  \pr(\{Z_i\}_{1\le i \le n}, \{X_{ij}\}_{1\le i<j\le n} )
= \prod_{i=1}^n \pr(Z_i) \prod_{1\le i<j\le n} \pr(X_{ij}|1_{Z_i=Z_j})\\
=\frac 1 {Q^n} \prod_{1\le i<j\le n} \pr(X_{ij}|1_{\sigma(Z_i)=\sigma(Z_j)})= \pr(\{\sigma(Z_i)\}_{1\le i \le n}, \{X_{ij}\}_{1\le i<j\le n} ).
\end{multline*}
As a consequence, for any $\sigma\in\mathcal{S}_{\Q}$ and any value $\uq\in\Q^k$, the conditional expectation $m_g(\sigma(\uq))$ satisfies
\begin{multline*}
m_g(\sigma(\uq))= \esp(g(\bX^{(1,\ldots,k)}) | (Z_1,\ldots,Z_k)=\sigma(\uq))\\
=\esp(g(\bX^{(1,\ldots,k)}) | (\sigma(Z_1),\ldots,\sigma(Z_k))=\sigma(\uq))
= m_g(\uq).
\end{multline*}
Thus  the set of values $\{m_g(\sigma(\uq)) ; \sigma \in \mathcal{S}_{\Q}\}$ is reduced to a singleton.
This finishes the proof of the lemma.\\

%%%%%%%%%%%%%%%%%%
Now, going back to the term $T_{2,1}$, the set $\Q^k$ may be partitioned into the disjoint union of the orbits induced by $\mathcal{S}_{\Q}$, namely $\Q^k =\cup_{\mathcal{O} orbit}\mathcal{O}$, with $\uq \to m_g(\uq)$ being constant on each orbit $\mathcal{O}$. We let $m_{g,\mathcal{O}}$ denote the value of the function $\uq \to m_g(\uq)$ on the orbit $\mathcal{O}$. Then we write
\begin{equation*}
  T_{2,1}= \frac 1 {n Q^{k-1} }\sum_{\mathcal{O} orbit} m_{g,\mathcal{O}} \sum_{l=1}^k\sum_{i=1}^n \sum_{\uq\in \mathcal{O}}    (Z_{iq_l}-\pi_{q_l}).
\end{equation*}
For each orbit $\mathcal{O}$ and any position $l\in\{1,\ldots,k\}$, if we fix some $\uq \in \mathcal{O}$, then we argue that $\mathcal{O}$ contains all the points of the form $(q_1,\ldots,q_{l-1},j,q_{l+1},\ldots,q_k)$ for any $1\le j\le Q$. Indeed, all these points are images of $\uq$ by the simple transpositions $(q_l\; j)$. Thus, the sum $\sum_{\uq \in \mathcal{O}} (Z_{iq_l}-\pi_{q_l})$ contains the sum $\sum_{q_l \in \Q} (Z_{iq_l}-\pi_{q_l})$ which is zero. This proves that 
$T_{2,1}=0$ and thus 
\begin{multline*}
 n( \hat m_g-m_g) = n(T_1+T_{2,2})=   \sum_{\uq  \in \Q^k}    \frac{(n-k)!}{(n-1)!} N_{\uq}^{1/2} \frac1 {N_{\uq}^{1/2}} \sum_{\substack{\ui \in  \Ik \\\prod_{l=1}^k Z_{i_l q_l} }} (g(\bX^{\ui})-m_g(\uq)) \\
+ \frac 1 {Q^{k-2}} \sum_{q,\ell\in\Q, q\neq \ell} \frac { 1}{(n-1)}
\sum_{ 1\le i\neq j\le n} (Z_{iq} -\pi_{q} )(Z_{j\ell} -\pi_{\ell} ) +o(1),
\end{multline*}
where, as in the non degenerate case, we argued that the terms in $T_{2,2}$ involving sets $L$ with cardinality $u\ge 3$ are negligible. We then  obtain that for $k=2$, we have
\begin{equation*}
  n( \hat m_g-m_g) \leadsto_{n\to\infty} \frac 1 Q \sum_{q,\ell \in \Q} V_{q\ell}
+\sum_{q,\ell\in \Q, q\ne \ell}(W_qW_\ell + \frac 1 {Q^2}) ,
\end{equation*}
where for any $1\le q,\ell\le Q$, the random variables $V_{q\ell}$ are independent, with distribution $ \mathcal{N}(0,\Var(g(X_{12})|Z_{1q}Z_{2\ell}=1))$ and $W=(W_1,\ldots,W_Q)$ is independent from the $V_{q\ell}$'s, with distribution $ \mathcal{N}_Q(0,\Sigma)$, and in the equal group proportions case $\Sigma$ simplifies to  $\Sigma_{q\ell}=-1/Q^2$ when $q\neq \ell$ and $\Sigma_{qq}=(Q-1)/Q^2$. 

In the same way, whenever $k\ge 3$, all the terms appearing in $T_1$ are now negligible and we get 
\begin{equation*}
  n( \hat m_g-m_g) \leadsto_{n\to\infty} \sum_{q,\ell\in \Q, q\ne \ell}(W_qW_\ell + \frac 1 {Q^2}) .
\end{equation*}

\end{proof}

\begin{proof}[Proof of Theorem~\ref{thm:bin_moment}]
Following the proof of Theorem~\ref{thm:main}, we can easily write a joint central limit theorem for the triplet $(\hat m_1,\hat m_2,\hat m_3)$. Namely, 
\begin{equation*}
  \sqrt{n} 
\begin{pmatrix}
    \hat m_1 - m_1 \\ \hat m_2 -m_2 \\ \hat m_3 -m_3
  \end{pmatrix}
\mathop{\leadsto}_{n \to \infty} 
\mathcal{N}_3\left( 0 , V \right) ,
\end{equation*}
with some covariance matrix $V$. 
Thus, we can apply a delta-method \cite[see for instance][Chapter 3]{VDV}  to the  estimators  $\hat \beta = \phi(\hat m_1,\hat m_2,\hat m_3)$ and $\hat \alpha =\psi(\hat m_1,\hat m_2,\hat m_3)$ where the functions $\phi$ and $\psi$ are differentiable. This gives the convergence of the estimators $(\hat \alpha,\hat \beta)$ and guarantees the same rates of convergence for $\hat \alpha, \hat \beta$ than for the $\hat m_i$'s. 
\end{proof}

\begin{proof}[Proof of Theorem~\ref{thm:multiv_bernoulli}]
Following the classical proof of \cite{Wald} \citep[see also][]{VDV}, we may obtain the almost sure convergence  of $(\hat  {\boldsymbol  {\gamma}}_n, \hat  \alpha_n,\hat
\beta_n)$ to  the true value of the  parameter $(\boldsymbol {\gamma}^\star,
\alpha^\star,\beta^\star)$,  provided the  parameter space  is compact
and the three following assumptions are satisfied:
\begin{itemize}
\item [$i)$] Convergence of the criterion 
\begin{multline*}
\ell_n (\boldsymbol \pi,\alpha,\beta) :=  \frac{1}{n(n-1)(n-2)} \sum_{(i,j,k)\in \I_3 }\log  \pr_{\boldsymbol \pi,\alpha,\beta} (X_{ij},X_{ik},X_{jk}) \\
\mathop{\rightarrow}_{n\to \infty}  H ((\boldsymbol \pi,\alpha,\beta);
(\boldsymbol                                     \pi^\star,\alpha^\star
,\beta^\star))                                :=\mathbb{E}_{\boldsymbol
  \pi^\star,\alpha^\star,\beta^\star}       \log      \pr_{\boldsymbol
  \pi,\alpha,\beta} (X_{12},X_{13},X_{23}) ,
\end{multline*}
$ \mathbb{P}_{\boldsymbol \pi^\star,\alpha^\star,\beta^\star} $-almost
surely ; 
\item   [$ii)$]  Identification  of   the  parameter   $  (\boldsymbol
  \gamma,\alpha,\beta) $ 
  \begin{equation*}
   H ((\boldsymbol \pi,\alpha,\beta);
(\boldsymbol \pi^\star,\alpha^\star ,\beta^\star)) \leq  H ((\boldsymbol \pi^\star,\alpha^\star,\beta^\star);
(\boldsymbol \pi^\star,\alpha^\star ,\beta^\star)) , 
  \end{equation*}
 with equality if and only if $ (\boldsymbol \gamma,\alpha,\beta) = 
(\boldsymbol   \gamma^\star,\alpha^\star    ,\beta^\star)   $,   where
$\boldsymbol \gamma$ and $\boldsymbol \pi$ are related through \eqref{eq:multiBernAffil}; 
\item [$iii)$] Uniform equicontinuity of the family of functions $ (\boldsymbol
    \pi,\alpha,\beta) \to \ell_n   (\boldsymbol
    \pi,\alpha,\beta)$.  Namely, for any  $\epsilon >0$,  there exists
    some $\nu>0$ such that for all $n\ge 1$ and as soon as $\| (\boldsymbol
    \pi,\alpha,\beta)-   (\boldsymbol
    \pi',\alpha',\beta') \|_{\infty} \le \nu$, we have $|\ell_n (\boldsymbol
    \pi,\alpha,\beta)-\ell_n   (\boldsymbol
    \pi',\alpha',\beta')| \le \epsilon$.
\end{itemize}
Item  $i)$ follows  from Theorem~\ref{thm:main},  while  $ii)$ follows
from      Jensen's     inequality     and     identifiability of the parameters, \emph{i.e.} Assumption~\ref{hyp:ident_binary}.
Let us know  establish $iii)$. We fix for the  moment some $\nu >0$
and consider $\eta =(\boldsymbol
 \pi,\alpha,\beta)$ and $\eta' =(\boldsymbol \pi',\alpha',\beta')$ such
 that $\|\eta-\eta'\|_{\infty}\le \nu$. 
We recall that $(X_{ij},X_{ik},X_{jk})=\bX^{(i,j,k)}$. We then write 
 \begin{multline*}
   |       \log       \mathbb{P}_{\eta}      (Z_{iq}Z_{j\ell}Z_{km}=1,
 \bX^{(i,j,k)} )-\log  \mathbb{P}_{\eta'}      (Z_{iq}Z_{j\ell}Z_{km}=1,
   \bX^{(i,j,k)}) | \\
\le |\log  \pi_q -\log \pi'_q| + |\log \pi_\ell
   -\log \pi'_\ell| + |\log \pi_m -\log \pi'_m| \\
+  |       \log       \mathbb{P}_{\eta}      (
   \bX^{(i,j,k)}|Z_{iq}Z_{j\ell}Z_{km}=1)-\log  \mathbb{P}_{\eta'}      (
   \bX^{(i,j,k)}|Z_{iq}Z_{j\ell}Z_{km}=1) |. 
 \end{multline*}
The  second term  in the  right hand  side of  this inequality  may be
bounded as follows 
\begin{multline*}
   |       \log       \mathbb{P}_{\eta}      (
 \bX^{(i,j,k)} |Z_{iq}Z_{j\ell}Z_{km}=1)-\log  \mathbb{P}_{\eta'}      (
  \bX^{(i,j,k)}|Z_{iq}Z_{j\ell}Z_{km}=1)    |   \\ \le   3\max(|\log
   \alpha-\log \alpha '|, |\log (1-\alpha)-\log (1-\alpha ')|, |\log
   \beta-\log \beta '|,\\
 |\log (1-\beta)-\log (1-\beta ')|). 
\end{multline*}
We now  make use of the fact  that we restricted our  attention to the
parameter space  $\Pi_\delta$, in which  all the parameters  are lower
bounded by $\delta$ (Assumption \ref{hyp:compact1}). Moreover, for any
$x,y  >0$, we  may use  $|\log  x -\log  y|\le |x-y|/\min(x,y)$.  This
finally leads to 
\begin{equation*}
   |       \log       \mathbb{P}_{\eta}      (Z_{iq}Z_{j\ell}Z_{km}=1,
\bX^{(i,j,k)})-\log  \mathbb{P}_{\eta'}      (Z_{iq}Z_{j\ell}Z_{km}=1,
\bX^{(i,j,k)} ) | 
\le 6\delta^{-1}\nu.
 \end{equation*}
Now, we obtain
\begin{multline*}
   \mathbb{P}_{\eta}    ( \bX^{(i,j,k)})    =    \sum_{q\ell   m}
   \mathbb{P}_{\eta} (Z_{iq}Z_{j\ell}Z_{km}=1, \bX^{(i,j,k)}) \\ \le
   \exp(6\delta^{-1}\nu    )    \sum_{q\ell   m}    \mathbb{P}_{\eta'}
   (Z_{iq}Z_{j\ell}Z_{km}=1,\bX^{(i,j,k)}) = \exp(6\delta^{-1}\nu    )     \mathbb{P}_{\eta'}   (\bX^{(i,j,k)}) ,
\end{multline*}
and thus 
\begin{equation*}
 \log \mathbb{P}_{\eta} (\bX^{(i,j,k)}) \le
  \frac {6\nu} {\delta}+  \log \mathbb{P}_{\eta'}
   (\bX^{(i,j,k)}) .
\end{equation*}
As this inequality is symmetric with respect to $\eta$ and $\eta'$, we
further obtain 
\begin{equation*}
 | \log \mathbb{P}_{\eta} ( \bX^{(i,j,k)}) - \log \mathbb{P}_{\eta'}
   (\bX^{(i,j,k)}) |\le
  \frac {6\nu} {\delta}. 
\end{equation*}
Finally, 
\begin{equation*}
|   \ell_n   (\eta)-\ell_n   (\eta')   |   \le   \frac{1}{n(n-1)(n-2)}
\sum_{i,j,k} | \log \mathbb{P}_{\eta} (\bX^{(i,j,k)} ) - \log \mathbb{P}_{\eta'}
   (\bX^{(i,j,k)}) |\le
  \frac {6\nu} {\delta}, 
\end{equation*}
which establishes $iii)$.\\ 

To further obtain the rates of convergence  of the estimators, one usually proceeds to a Taylor expansion of the derivative $\partial \ell_n(\boldsymbol \pi^\star,\alpha^\star,\beta^\star)$ in a vicinity of the estimator $(\hat {\boldsymbol \pi}_n,\hat \alpha_n, \hat \beta_n)$.
Let us write
\begin{equation*}
  0=\partial \ell_n (\hat {\boldsymbol \pi}_n,\hat \alpha_n, \hat \beta_n) = \partial \ell_n (\boldsymbol \pi^\star,\alpha^\star,\beta^\star) 
+[ (\hat {\boldsymbol \pi}_n,\hat \alpha_n, \hat \beta_n)-(\boldsymbol \pi^\star,\alpha^\star,\beta^\star)] \partial^2 \ell_n (\tilde {\boldsymbol \pi}_n,\tilde \alpha_n, \tilde \beta_n) ,  
\end{equation*}
where $ (\tilde {\boldsymbol \pi}_n,\tilde \alpha_n, \tilde \beta_n) $
is some point between $(\hat {\boldsymbol \pi}_n,\hat \alpha_n, \hat \beta_n)$ and $(\boldsymbol \pi^\star,\alpha^\star,\beta^\star)$.
Applying  Theorem~\ref{thm:main}  to  the  quantity  $\partial  \ell_n
(\boldsymbol  \pi^\star,\alpha^\star,\beta^\star)  $,  we  obtain  its
almost      sure      convergence     to      $\mathbb{E}_{\boldsymbol
  \pi^\star,\alpha^\star,\beta^\star}( \partial \log $ $\pr_{\boldsymbol
  \pi^\star,\alpha^\star,\beta^\star} (X_{12},X_{13},X_{23}) ) =0$, as well as the asymptotic normality 
\begin{equation*}
  \sqrt{n} \partial \ell_n (\boldsymbol \pi^\star,\alpha^\star,\beta^\star) \leadsto_{n \to \infty} \mathcal{N}(0, J) .
\end{equation*}
%where 
% \begin{multline*}
%    J_{\boldsymbol \pi,\alpha,\beta} =9 \text{Cov}_{\boldsymbol \pi^\star,\alpha^\star,\beta^\star}[ \partial \log  \pr_{\boldsymbol \pi,\alpha,\beta} (\bX^ {123}) ,  \partial \log  \pr_{\boldsymbol \pi,\alpha,\beta} (\bX^{145})  ] \\
%  =9 \mathbb{E}_{\boldsymbol \pi^\star,\alpha^\star,\beta^\star}[ \partial \log  \pr_{\boldsymbol \pi,\alpha,\beta} (\bX^{123}) \cdot  \partial \log  \pr_{\boldsymbol \pi,\alpha,\beta} (\bX^{145}) ^\intercal  ] .
% \end{multline*}
%(The specific form of the limiting variance is derived in the proof of Theorem~\ref{thm:main}).
Now, at  a fixed point $(\boldsymbol \pi,\alpha,\beta)  $, the Hessian
matrix $\partial^2  \ell_n (\boldsymbol \pi,\alpha,\beta)  $ converges
from      Theorem~\ref{thm:main}      to      $\mathbb{E}_{\boldsymbol
  \pi^\star,\alpha^\star,\beta^\star}         (\partial^2         \log
\pr_{\boldsymbol                                      \pi,\alpha,\beta}
(X_{12},X_{13},X_{23})  )$. Combining the  almost sure  convergence of
$(\hat   {\boldsymbol   \pi}_n,\hat   \alpha_n,  \hat   \beta_n)$   to
$(\boldsymbol   \pi^\star,\alpha^\star,\beta^\star)$,   with   uniform
equicontinuity    of   the    family   of    functions   $(\boldsymbol
\pi,\alpha,\beta) \to \partial^2 \ell_n (\boldsymbol \pi,\alpha,\beta)
$  (the  proof is  similar  to point  $iii)$  above  and is  therefore
omitted), we obtain the almost sure convergence 
\begin{equation*}
  \partial^2  \ell_n   (\tilde  {\boldsymbol  \pi}_n,\tilde  \alpha_n,
  \tilde \beta_n) \to_{n \to \infty} \mathbb{E}_{\boldsymbol
  \pi^\star,\alpha^\star,\beta^\star}         (\partial^2         \log
\pr_{\boldsymbol                    \pi^\star,\alpha^\star,\beta^\star}
(X_{12},X_{13},X_{23})           )         :=      -    K .
%_{\boldsymbol  \pi^\star,\alpha^\star,\beta^\star} .
\end{equation*}
If  the Fisher information matrix $ K$ is invertible, we  obtain
\begin{equation*}
\sqrt{n}   [    (\hat   {\boldsymbol   \pi}_n,\hat    \alpha_n,   \hat
\beta_n)-(\boldsymbol \pi^\star,\alpha^\star,\beta^\star)] \leadsto_{n
  \to          \infty}          \mathcal{N}(0,  K ^{-1} J K ^{-1}). 
\end{equation*}
In this case, the inverse of the limiting variance is known as Godambe information \citep{Varin08}. Its form is due to the fact that $K^{-1}\neq J$ in general, resulting in a loss of efficiency of the estimators. 
In case where $ K $ is not invertible, or when $ J =0$, the rate of convergence of the estimators is faster than $1/\sqrt{n}$. In particular, when the group proportions are equal, we know from Theorem~\ref{thm:main} that   $n \partial \ell_n (\boldsymbol \pi^\star,\alpha^\star,\beta^\star)$ converges in distribution and then the rate of convergence of $ (\hat   {\boldsymbol   \pi}_n,\hat    \alpha_n,   \hat
\beta_n)$ is at least $1/n$.
\end{proof}

\begin{proof}[Proof of Theorem~\ref{thm:cv_weighted}]
The   proof   follows  the   scheme   described   in   the  proof   of
Theorem~\ref{thm:multiv_bernoulli}.         We        denote        by
$(\boldsymbol{\pi^\star}, \mathbf{p^\star}, \boldsymbol {\theta^\star})$
the  true  value  of  the  parameter  and  by  $\mathbb{P}^\star$  and
$\mathbb{E}^\star$ the corresponding probability and expectation. 
First,   let  us   establish   the  consistency   of  the   normalized
composite likelihood (point $i)$). According to Theorem~\ref{thm:main}, we have for any fixed value of $(\boldsymbol \pi, \mathbf p, \boldsymbol \theta)$,
\begin{multline*} 
 \frac{2}{n(n-1)}   \sum_{1\le  i<j\le   n}  1_{X_{ij}\neq   0}  \log
  \mathbb{P}_{\boldsymbol{\pi}, \mathbf{p}, \boldsymbol \theta}(X_{ij}) \mathop{\longrightarrow}_{n\to +\infty} \mathbb{E}^\star(1_{X_{12}\neq 0} \log
  \mathbb{P}_{\boldsymbol{\pi},    \mathbf{p},   \boldsymbol   \theta}
  (X_{12})) \\
:=   H(   (\boldsymbol{\pi},   \mathbf{p},   \boldsymbol   \theta)   ;
(\boldsymbol{\pi^\star}, \mathbf{p^\star}, \boldsymbol \theta^\star)),
\quad \mathbb{P}^\star \text{ a.s.}
\end{multline*}
Here, we need to deal with the fact that we use a random value for $\mathbf p$ (a preliminary step estimate) in the definition of $\hat {\boldsymbol \theta}$. It is thus necessary to prove that this convergence happens uniformly with respect to $\mathbf p$. But this is going to be a consequence of point $iii)$ below. 
Combining this with the almost sure convergence of $\hat {\mathbf p}_n$  to the true value $\mathbf{p^\star}$ (this is either a consequence of Theorem~\ref{thm:main} when $\mathbf p=p$ is constant, or a consequence of Sections~\ref{sec:bin_poly} and \ref{sec:bin_multidim} when $\mathbf p=(\alpha,\beta)$), we get 
\begin{equation*}
  \frac{2}{n(n-1)} \mathcal{L}_{X}^{\text{compo}}(\boldsymbol \pi, \hat{\mathbf p}_n, \boldsymbol \theta)  
 \mathop{\longrightarrow}_{n\to +\infty}   H(   (\boldsymbol{\pi},   \mathbf{p},   \boldsymbol   \theta)   ;
(\boldsymbol{\pi^\star}, \mathbf{p^\star}, \boldsymbol \theta^\star)),
\quad \mathbb{P}^\star \text{ a.s.}
\end{equation*}
Moreover, we assumed that $f(\cdot, \theta)$ has a continuous c.d.f. and the distribution of a present edge is given by \eqref{eq:weighted_mixture}, so that we have 
\begin{multline*}
  H(   (\boldsymbol{\pi},   \mathbf{p},   \boldsymbol   \theta)   ;
(\boldsymbol{\pi^\star},  \mathbf{p^\star}, \boldsymbol \theta^\star))
=    \int_{x}    \log(\gamma_{\text{in}}    f(x;\theta_{\text{in}})    +\gamma_{\text{out}}
f(x;\theta_{\text{out}}))\\
 \times (\gamma_{\text{in}}^\star    f(x;\theta_{\text{in}}^\star)    +\gamma_{\text{out}}^\star
f(x;\theta_{\text{out}}^\star)) dx, 
\end{multline*}
where $(\gamma_{\text{in}},\gamma_{\text{out}})$ as well as $(\gamma_{\text{in}}^\star,\gamma_{\text{out}}^\star)$ are defined through $(\boldsymbol{\pi},\mathbf{p})$ and $(\boldsymbol{\pi^\star},\mathbf{p^\star})$ respectively. Thus, the difference 
\begin{equation*}
   H(   (\boldsymbol{\pi^\star},   \mathbf{p^\star},   \boldsymbol   \theta^\star)   ;
(\boldsymbol{\pi^\star},  \mathbf{p^\star}, \boldsymbol \theta^\star))
- H(   (\boldsymbol{\pi},   \mathbf{p},   \boldsymbol   \theta)   ;
(\boldsymbol{\pi^\star}, \mathbf{p^\star}, \boldsymbol \theta^\star))
\end{equation*}
 is a Kullback-Leibler divergence between two mixture distributions of the form  \eqref{eq:weighted_mixture}. This entails positivity of this difference. Moreover, Assumption~\ref{hyp:ident} ensures that the difference is zero if and only if 
 \begin{equation*}
   \gamma_{\text{in}}\delta_{\theta_{\text{in}}} +\gamma_{\text{out}}\delta_{\theta_{\text{out}}}
=    \gamma_{\text{in}}^\star\delta_{\theta_{\text{in}}^\star} +\gamma_{\text{out}}^\star\delta_{\theta_{\text{out}}^\star},
 \end{equation*}
which establishes point $ii)$, up to a permutation on the label parameters $\{\text{in, out}\}$. 
Finally, the proof of point $iii)$ follows the same lines as in the proof of Theorem~\ref{thm:multiv_bernoulli}, and uses the continuity of the map $\theta \mapsto f(\cdot, \theta)$, which is a consequence of Assumption~\ref{hyp:regular}.

To further obtain the rates of convergence of our estimators,  we proceed exactly as we did in the proof of Theorem~\ref{thm:multiv_bernoulli}.
\end{proof}

\paragraph*{Acknowledgments} The authors thank Jérôme Dedecker for  helpful insights concerning this work as well as the associate editor and an anonymous referee for their remarks leading to considerable improvements on the manuscript.
 The authors have been supported  by the French Agence Nationale de la Recherche  under grant NeMo ANR-08-BLAN-0304-01.

\bibliographystyle{plainnat}
\bibliography{mixnet}

\end{document}